\documentclass[a4paper,reqno,oneside]{amsart}

\usepackage[
style=numeric,
sorting=nty,
maxnames=99,
maxalphanames=5,
natbib=true,
sortcites]{biblatex}

\usepackage[british]{babel}
\usepackage{csquotes}  
\numberwithin{equation}{section}
\usepackage{amssymb}
\usepackage[protrusion=true,expansion=true]{microtype}	
\usepackage{amsmath,amsfonts,amsthm} 
\usepackage[pdftex]{graphicx}	
\usepackage{url}
\usepackage{svg}
\usepackage[title,titletoc,page]{appendix} 
\usepackage{listings} 
\usepackage{color}
\usepackage{caption}
\usepackage{subcaption}
\usepackage{algorithm}
\usepackage{dsfont} % for 1 bbmath
\usepackage{bm}
\usepackage{upgreek}
\usepackage[normalem]{ulem}
\usepackage{enumerate}
\usepackage{mathtools}
\usepackage{lipsum}
\usepackage{adjustbox}
\usepackage{tikz}
\usepackage{circuitikz}
\usepackage{tcolorbox}
\usetikzlibrary{tikzmark,calc}
\usepackage{xargs}      % Use more than one optional parameter

\usepackage[T1]{fontenc} % To switch to the T1 encoding
\usepackage{lmodern} % To switch to Latin Modern
\rmfamily 
\DeclareFontShape{T1}{lmr}{b}{sc}{<->ssub*cmr/bx/sc}{}
\DeclareFontShape{T1}{lmr}{bx}{sc}{<->ssub*cmr/bx/sc}{}

\usepackage[left=2.4cm,right=2.4cm,top=3.5cm,bottom=3.5cm,footskip=1.5cm,headsep=1cm]{geometry}
\usepackage[%
pdfauthor={Yannick De Bruijn, Erik Orvehed Hiltunen},%
pdftitle={PHOTONIC BANDGAPS II},
pdfsubject={PHOTONIC BANDGAPS II},
pdfkeywords={Subwavelength resonances, evanescent modes, band gap, interface eigenmodes, Bloch theory, complex Brillouin zone, layer potentials}]{hyperref}

\usepackage{fancyhdr}
\fancypagestyle{plain}{%
\fancyhead[C]{}
\fancyfoot[C]{}

}

\fancypagestyle{nosection}{%
\fancyhead[CE]{}
\fancyhead[CO]{}
\fancyfoot[CE,CO]{\thepage}
}
\numberwithin{equation}{section}		% Equationnumbering: section.eq#
\numberwithin{figure}{section}			% Figurenumbering: section.fig#
\numberwithin{table}{section}			% Tablenumbering: section.tab#

\newcommand{\Cb}{\mathbb{C}}
\DeclareMathOperator{\C}{\mathbb{C}}

\newcommand{\F}{\mathcal{F}}

\newcommand{\N}{\mathbb{N}}

\newcommand{\R}{\mathbb{R}}

\renewcommand{\S}{\mathcal{S}}

\newcommand{\I}{\mathcal{I}}

\newcommand{\Z}{\mathbb{Z}}

\renewcommand{\epsilon}{\varepsilon}

\newcommand{\x}{\mathbf{x}}

\newcommand{\uf}{\mathfrak{u}}

\newcommand{\Cf}{\mathfrak{C}}

\renewcommand{\Re}{\mathrm{Re}}

\newcommand{\ds}{\displaystyle}
\newcommand{\iu}{\mathrm{i}\mkern1mu}

\DeclareMathOperator{\capmat}{\mathcal{C}}

\renewcommand{\tilde}{\widetilde}
\renewcommand{\hat}{\widehat}
\renewcommand{\i}{\mathrm{i}}
\renewcommand{\d}{\,\mathrm{d}}
\renewcommand{\Re}{\mathfrak{Re}}
\renewcommand{\Im}{\mathfrak{Im}}

\newcommand{\iL}{\mathsf{L}}
\newcommand{\iR}{\mathsf{R}}

\usepackage[noabbrev,capitalize]{cleveref}
\crefname{proposition}{Proposition}{Propositions}
\crefname{equation}{}{}

\newtheorem{theorem}{Theorem}[section]
\newtheorem{lemma}[theorem]{Lemma}
\newtheorem{proposition}[theorem]{Proposition}
\newtheorem{corollary}[theorem]{Corollary}

\theoremstyle{definition}
\newtheorem{definition}[theorem]{Definition}

\newtheorem{remark}[theorem]{Remark}

\crefname{assumption}{Assumption}{Assumptions}
\crefname{definition}{Definition}{Definitions}
\crefname{corollary}{Corollary}{Corollaries}
\crefname{enumi}{item}{items}

\begin{document}

\title[Complex Brillouin Zone]{
Complex Brillouin Zone for Localised Modes in Hermitian and Non-Hermitian Problems}

\author[Y. De Bruijn]{Yannick De Bruijn}
\address{\parbox{\linewidth}{Yannick De Bruijn\\
 ETH Z\"urich, Department of Mathematics, Rämistrasse 101, 8092 Z\"urich, Switzerland.}}
\email{ydebruijn@student.ethz.ch}

\author[E. O. Hiltunen]{Erik Orvehed Hiltunen}
\address{\parbox{\linewidth}{Erik Orvehed Hiltunen\\
 Department of Mathematics, University of Oslo.}}
\email{erikhilt@math.uio.no}

%%%-------------------------------------------------------
\begin{abstract}
    We develop a mathematical and numerical framework for studying evanescent waves in subwavelength band gap materials. By establishing a link between the complex Brillouin zone and various Hermitian and non-Hermitian phenomena, including defect localisation in band gap materials and the non-Hermitian skin effect, we provide a unified perspective on these systems. In two-dimensional  structures, we develop analytical techniques and numerical methods to study singularities of the complex band structure. This way, we demonstrate that gap functions effectively predict the decay rates of defect states. Furthermore, we present an analysis of the Floquet transform with respect to complex quasimomenta. Based on this, we show that evanescent waves may undergo a phase transition, where local oscillations drastically depend on the location of corresponding frequency inside the band gap.

\end{abstract}

\maketitle

\vspace{3mm}
\noindent
\textbf{Mathematics Subject Classification (MSC2010): 35J05, 35C20, 35P20, 81Q12.}  % (81Q12) Nonselfadjoint operator theory.

\vspace{3mm}

\vspace{5mm}

\section{Introduction}
Building upon our prior research \cite{debruijn2024complexbandstructuresubwavelength}, this study expands existing framework to investigate novel phenomena characterised by the complex band structure. These advances are closely related to the study of band gap materials, which form the basis for the achievements of wave localisation, guiding, and topological robustness \cite{ammari2018subwavelength,lemoult2016soda,yves2017crytalline,fefferman2014topologically,fefferman2018honeycomb}. For frequencies in the spectral bulk, the corresponding eigenmodes are propagating and, according to Bloch's theorem, attains a phase shift when passing over a fundamental cell. Allowing the Bloch quasimomentum to become complex opens up the possibility of studying evanescent or localised modes for frequencies in the band gap \cite{goodwin1939electronic,dang2014complex,reuter2016unified,chang1982complex,tomfohr2002complex}. The fundamental significance of the complex Bloch parameter is predicting the decay rate of such waves, and may be used to explicitly determine energy leakage or tunnelling probabilities of localised waves. Similarly to propagating modes, the real part of the complex quasimomentum describes the local oscillations of the wave.

The interaction of waves with subwavelength-scale objects has recently gained significant attention because of its potential to exceed the diffraction limit. Using a significant contrast in the material parameters allows the opening of band gaps in spatially periodic materials of length scales significantly smaller than the wavelength of the wave \cite{AmmariGapOpening}. One approach to achieving subwavelength interactions involves high-contrast metamaterials, composed of a background medium embedded with highly contrasting resonators. A pioneering example of such resonance was demonstrated with air bubbles in water \cite{minnaert1933musical,ammari2018minnaert,miao2024generalized}. Since then, similar phenomena have been observed in a variety of settings, including high-contrast dielectric particles \cite{ammari2019dielectric,ammari2023mathematical,meklachi2018asymptotic}, Helmholtz resonators \cite{ammari2015superresolution}, and plasmonic nanoparticles \cite{ammari2017plasmonicscalar,ammari2016plasmonicMaxwell}. In particular, evanescent waves have been shown to enhance near-field transmission, playing a critical role in subwavelength imaging beyond the diffraction limit (see, for example, \cite{luo2003subwavelength, engelen2009subwavelength,lezec2004diffracted,sukhovich2009experimental,belov2005canalization}). 

The introduction of a complex quasimomentum transforms the spectral wave problem, which is Hermitian for real quasimomenta, into a non-Hermitian problem \cite{debruijn2024complexbandstructuresubwavelength}. Conversely, the non-Hermitian skin effect, a phenomenon where the bulk eigenmodes of a non-reciprocal system are all localised at one edge of an open chain, is naturally described in terms of a complex Brillouin zone. The skin effect has been realised experimentally in topological photonics, phononics, and other condensed matter systems \cite{ghatak.brandenbourger.ea2020Observation,skinadd1,skinadd2, skinadd3}. 
By introducing an imaginary gauge
potential, all the bulk eigenmodes condensate in the same direction \cite{hatano,yokomizo.yoda.ea2022NonHermitian,rivero.feng.ea2022Imaginary}. The introduction of complex quasimomenta also resolves the discrepancy of open versus periodic boundary conditions seen in non-Hermitian skin effect models (see, for example, \cite{verma2024topological}).

This paper is organised as follows. Section \ref{Sec: One D PDE Theory} treats one-dimensional resonator systems, whereby a spectral parametrisation for the band gap is presented. This shows a one-to-one correspondence between band gaps and complex bands, a property which does not hold in higher dimensionalities. Moreover, we establish a link between the complex band structure and non-Hermitian systems, providing applications related to the non-Hermitian skin effect and disordered systems. We also show that in topologically protected systems such as Su–Schrieffer–Heeger (SSH) chains, the localisation strength is robust to perturbations.
In Section \ref{sec: Two dimensional resonator}, we provide a detailed analysis of the standing wave patterns that can emerge at specific complex quasimomenta, resulting in singularities in the complex band structure. These singularities exhibit significant physical implications, since they provide a limit to the decay rate. 
In Section \ref{Sec: Numerical Methods}, we introduce and compare different numerical approaches to determine the complex band structure. These methods utilise layer potential techniques along with the multipole expansion method to calculate the subwavelength gap bands.
In Section \ref{Sec: Bandstructure for defect modes}, we consider localised defect modes in two-dimensional structures, and show that the complex band structure provides an accurate prediction for the decay length of such modes. Moreover, we expand the Floquet transform to include complex quasimomenta. This allows us to illustrate a phase transition of the defect modes depending on the location of the frequency inside the band gap. 
In Section \ref{Three dimensional resonators}, we discuss limitations of the complex band structure, where the resonators are not repeated periodically in all spatial dimensions. Finally, in Section \ref{Sec: concluding remarks}, we examine the potential applications of our findings and outline possible directions for future research. The \texttt{MATLAB} code that supports the findings of this paper is made openly available in Section \ref{Sec: Data availability}.

\section{Band structure of one-dimensional boundary value problems.} \label{Sec: One D PDE Theory}
In this section, we review fundamental results from Floquet-Bloch theory and extend them to complex quasimomenta. We develop a complete parametrisation of the spectrum for a class of elliptic and periodic differential operators. These results are then applied to Hermitian and  non-Hermitian problems, including phenomena such as the non-Hermitian skin effect, disordered systems, and topologically protected interface modes.

\subsection{Elliptic and periodic PDE theory.}
We begin by revisiting some basic results from the theory of elliptic operators. Thorough discussion of this field may be found, for example, in \cite{kuchment1993Floquet,kuchment2016overview}. Let $L(x, D)$ be a scalar elliptic operator with smooth periodic coefficients.

\begin{definition}\label{def: lattice}
    A \emph{lattice} $\Lambda \in \R^n$ is the set of all integer linear combinations of $n$ linearly independent lattice vectors $l_1, \dots, l_n$, i.e.
    \begin{equation}
        \Lambda := \bigl\{ \ell \in \R^n ~|~ \ell =  \sum_{j=1}^n m_j l_j, \hspace{2mm} m_j \in \Z \bigr\}.
    \end{equation}
\end{definition}

\begin{definition}\label{def: reciprocal lattice}
    The \emph{dual} or \emph{reciprocal lattice} to $\Lambda$, denoted by $\Lambda^*$ is defined as
    \begin{equation}
        \Lambda^*:= \bigl\{ k \in \R^n ~|~  k\cdot \ell \in 2 \pi \Z, \hspace{2mm}\forall \ell \in \Lambda \bigr\}.
    \end{equation}
\end{definition}

\begin{definition}\label{def: cell}
    Denote by $Y \subset \R^{n}$ a \emph{fundamental domain} or \emph{unit cell} of the given lattice. Explicitly, we take 
    \begin{equation} 
        Y := \left\{ \sum_{i = 1}^n l_i c_i ~\bigg|~ c_i \in [0,1] \right\}.
    \end{equation}
    The \emph{(real) Brillouin zone} $Y^*$ is now defined as $Y^*:= \R^{n} / \Lambda^*$.
\end{definition}

\begin{definition}
    We say that $u$ is a \emph{Bloch solution} to $Lu = \lambda u$ if it is of the form
    \begin{equation}
        u(x + \ell ) = e^{\i k\ell}u(x), \ \forall \ell\in \Lambda,  \quad k \in Q \subseteq \C,
    \end{equation}
    where $Q$ denotes the set of quasimomenta.
    We refer to \emph{propagating Bloch modes} if $k \in Y^*$ and \emph{evanescent Bloch modes} if $\Im(k) \neq 0$.
 \end{definition}

Bloch solutions typically require that the crystal studied is \emph{strictly periodic}. This means that a $n$ dimensional resonator contained in the unit cell has to be periodically repeated in all $n$ lattice dimensions \cite[Section 1]{Bloch1929}.
\begin{definition}
    By \emph{spectral band} we understand the continuous range of eigenvalues of the spectral problem. 
    By \emph{band gap} we understand the complement of the spectral bands, therefore, we will define it as a subset of $\R^n$.
\end{definition}
The subsequent result establishes a condition under which the set of quasimomenta attains complex values.

\begin{theorem}\cite[Theorem 4.1.4]{kuchment1993Floquet}\label{Thm: Fredholm index quasiomenta connection}
    If the Fredholm index of $L(x,D)$ on $\mathbb{T}^n$ is positive:
    \begin{equation}
        \text{index}_{ \hspace{1mm} \mathbb{T}^n} L > 0,
    \end{equation}
    then $Q = \C^n$.
\end{theorem}

It is a well-established fact from functional analysis that the Fredholm index of a self-adjoint operator is zero. This is because the kernel of a self-adjoint operator is orthogonal to its range. The following theorem establishes a connection between localised eigenmodes and the underlying operator having non-trivial winding number.

\begin{theorem}\cite[Theorem 4.1.5]{kuchment1993Floquet} \label{Thm: decay complex disperion equivalence}
    The following are equivalent:
    \begin{enumerate}[i)]
        \item $Q = \C^n$.
        \item The equation $Lu = 0$ has a non-zero solution such that for any $a > 0$ the following bound holds,
        \begin{equation}
            \lvert u(x) \rvert \leq C_a e^{-a\lvert x \rvert},\quad  \forall a > 0.
        \end{equation}
    \end{enumerate}
\end{theorem}

In the case of one-dimensional differential operators such as in the non-Hermitian skin effect, the hermiticity is broken by introducing a complex gauge potential into the resonators, see \cite{ammari2023mathematicalSkin}. In \cite{ ammari2024spectrapseudospectratridiagonalktoeplitz}, a  connection between a non-trivial winding number of the operator and the localised eigenmodes is established. This statement is highly non-trivial and extending it to a broader class of differential problems remains beyond reach. This serves as motivation to propose an alternative method that extends this concept, enabling the examination of a broader spectrum of condensation and localisation phenomena (see Section \ref{sec: Complex Brillouin zone for non-self adjoint diff operators}).

Our main objective is to identify quasiperiodic solutions to a differential operator with periodic boundary conditions. Utilising Floquet-Bloch theory, we obtain quasiperiodic eigenmodes along with their associated band functions. We begin by revisiting fundamental results on the spectrum of self-adjoint, elliptic, and periodic operators.

\begin{theorem}\cite[Theorem 5.5]{kuchment2016overview}\label{thm: spectrum self adjoint operator}
    Let $L$ be a self-adjoint elliptic periodic operator, then its spectrum consists of the range of the (real) band functions, i.e.
    \begin{align}
        \sigma(L) &= \bigcup_{k \in \mathcal{B}} \sigma\bigl(L(k)\bigr)\\
        &= \bigl\{\lambda \in \R \mid \exists k \in \mathcal{B}\subset \R^n, \text{ such that }\lambda \in \sigma\bigl(L(k)\bigr)\bigr\}.
    \end{align}
\end{theorem}
The following theorem extends the classical Bloch's theorem, allowing evanescent solutions associated with complex quasimomenta.

\begin{theorem}\cite[Theorem 4.3.1]{kuchment1993Floquet}\label{thm: kuchment estimate}
    Let $L(x, D)$ be a scalar periodic elliptic operator with smooth coefficients in $\R^n$. If $Lu = 0$ has a non-trivial solution satisfying the inequality
    \begin{equation}
        \lvert u(x)\rvert \leq C e^{a\lvert x \rvert}
    \end{equation}
    for some $a > 0$, then it also has a Bloch solution with a quasimomentum $k$ such that $\Im (k_j) \leq a$, for all $j = l, \dots, n$.
\end{theorem}
From a heuristic point of view, the imaginary part of the quasimomentum can be bounded above by $\operatorname{dist}\bigl(\lambda, \sigma(L)\bigr)$. The further the frequency is from the spectrum, the greater the imaginary part of the quasimomentum. Conversely, for frequencies $\lambda \in \sigma(L)$, this upper limit is $0$, and we expect a real quasimomentum. As we shall see, this heuristic is no longer valid in higher dimensions, where the decay rate might be bounded by a universal constant independent of $\operatorname{dist}\bigl(\lambda, \sigma(L)\bigr)$ (see Section \ref{sec:defect}).

Bloch's Theorem follows directly from Theorem \ref{thm: kuchment estimate}, as the condition for a bounded solution, $\lvert u(x) \rvert \leq C$, necessitates $a = 0$. Consequently, this leads to $\Im(k) \leq 0$, which indicates that the quasimomentum $k$ must be restricted to real values.

\begin{theorem}[Bloch’s Theorem]\cite[Section 4.3]{kuchment1993Floquet} \label{thm: eigenvalue in spectrum implies real floquet paramenter}
    The existence of a non-trivial bounded solution of a periodic elliptic equation $Lu = \lambda u$ implies the existence of a Bloch solution with a real quasimomentum and thus $\lambda \in \sigma(L)$.
\end{theorem}

\subsection{One-dimensional Helmholtz boundary value problem.}
A frequently examined differential operator is the Helmholtz eigenvalue problem. In this setting, we consider a one-dimensional resonator chain where the wave speeds and wave numbers inside the resonators and the background medium are denoted respectively by $v_i$, $v$ and $k_i$, $k$. Let us denote the resonators by $D_i$ and define $D:= \cup_{i} D_i$. We define the contrast $\delta:= \rho_1 / \rho_0$ as the ratio between the densities as well as the material parameters $\kappa(x)$, where
\begin{equation}
    \rho(x) = \begin{cases}
        \rho_1 , &\quad x \in D, \\
        \rho_0 , &\quad x \in \R \setminus D,
    \end{cases} \qquad \kappa(x) = \begin{cases}
        \kappa_i, & \quad x \in D_i,  \\
        \kappa_0, & \quad x \in \R \setminus D.
    \end{cases}
\end{equation} 
The different parameters are linked by,
\begin{align}
    v_i:=\sqrt{\frac{\kappa_i}{\rho_1}}, \qquad v:=\sqrt{\frac{\kappa_0}{\rho_0}},\qquad
    k_i:=\frac{\omega}{v_i},\qquad k:=\frac{\omega}{v}.
\end{align}
We consider the Helmholtz problem with transmission boundary conditions, which can be formulated as follows:

\begin{equation}\label{eq:u_pde}
    \begin{cases}
		\ds\Delta u + k^2u = 0, & \text{in } Y \setminus \overline{D}, \\
		\ds\Delta u + k_i^2u = 0,  & \text{in } D_i, \\
		\ds u\rvert_+ -u\rvert_- =0, &\text{on }\partial D,\\
		\ds\frac{\partial u}{\partial \nu}\bigg|_{-} - \delta \frac{\partial u}{\partial \nu} \bigg|_{+} = 0, & \text{on } \partial D, \\
        \ds u(x + \ell) = e^{\i (\alpha+\i \beta) \cdot \ell}u(x), &\text{for all }\ell \in \Lambda,
    \end{cases}
\end{equation}
and may be rewritten as,
\begin{equation}\label{eq: scattering operator}
\begin{cases}
    \kappa(x)\nabla \cdot \left( \frac{1}{\rho(x)} \nabla u(x) \right) + \omega^2 u(x) = 0\\
    u(x + \ell) = e^{\i(\alpha +\i \beta)\ell} u(x),
\end{cases}
\end{equation}
which is a spectral problem for the differential operator $\mathcal{L}u = -\kappa(x)\nabla \cdot \left( \frac{1}{\rho(x)} \nabla u(x) \right)$. We let $\mathcal{L}(\alpha)$ denote its restriction to $\alpha$-quasiperiodic functions while $\mathcal{L}(\alpha,\beta)$ denotes its restriction to functions with (complex) quasimomentum $(\alpha+\i \beta)$. Note that an eigenmode to \eqref{eq: scattering operator} must satisfy the continuity of the field and the continuity of the flux at the boundaries of the resonator.
Following \cite[Theorem 2.12.]{debruijn2024complexbandstructuresubwavelength}, the band functions can therefore be expressed via the \emph{transfer matrix} $T(w)$,
\begin{equation}
\operatorname{det}\left(T(w) -e^{\i(\alpha + \i \beta)}\right) = 0
\end{equation}
The solutions to this equation are transcendental, and finding a classical band function parametrised by a function is out of reach.
Nevertheless, one can parametrise the gap functions, allowing us to offer a counterpart to Theorem \ref{thm: spectrum self adjoint operator}.

\begin{theorem}\label{Thm: General spectral theorem}
    Let $\mathcal{L}$ be the differential operator defined in \eqref{eq: scattering operator}, let its spectrum be defined as in Theorem \ref{thm: spectrum self adjoint operator}, and  define its resolvent set $\rho(\mathcal{L}) := \R \setminus \sigma(\mathcal{L})\subset \R$.
    It is parametrised by
    \begin{equation}\label{eq: resolvent_Spectrum}
        \rho(\mathcal{L}) = \bigcup_{k \in \tilde{\mathcal{B}}} \sigma \bigl ( \mathcal{L}(\alpha, \beta)\bigr) \cap \R,
    \end{equation}
    where $\tilde{\mathcal{B}}:= \bigl\{(\alpha, \beta) \in \{-\pi, 0, \pi\}\times \R \setminus \{0\}\bigr\}$ denotes a subset of the complex Brillouin zone.
\end{theorem}

\begin{proof}
    The operator $\mathcal{L}(\alpha, \beta)$ is elliptic on a compact domain, therefore has a compact resolvent, and by the Fredholm alternative it has a discrete spectrum. We demonstrate that the full spectrum is accounted for without any spectral band overlap.
    
    Suppose that there exists a $k \in\rho(\mathcal{L})$, then $\beta \neq 0$, as otherwise the solution would propagate freely and the frequency would be inside of the bulk. Outside of the spectrum $\alpha \in \{-\pi, 0, \pi \}$ (see \cite[Theorem 2.8]{debruijn2024complexbandstructuresubwavelength}) and so $k \in \bigcup_{k \in \tilde{\mathcal{B}}} \sigma \bigl ( \mathcal{L}(\alpha, \beta)\bigr) \cap \R$. 
    
    Suppose now that $k \not\in \rho(\mathcal{L})$. If $\alpha \in \{ - \pi, 0, \pi \}$ then $\beta = 0$ and so $k \in \sigma(\mathcal{L})$. Suppose $\alpha \not\in \{-\pi, 0, \pi\}$, then the eigenvalues of the transfer matrix are non-real. As a consequence, they are complex conjugates of each other (see \cite[Lemma 2.11]{debruijn2024complexbandstructuresubwavelength}). If $\lambda_1 = e^{\i(\alpha + \i\beta)}$, then $\lambda_2 = e^{-\i(\alpha + \i\beta)}$. Since the determinant of the transfer matrix is equal to $1$, it holds $\lambda_1\lambda_2 = 1$ if and only if $\beta = 0$, but then $k \not \in \bigcup_{k \in \tilde{\mathcal{B}}} \sigma \bigl ( \mathcal{L}(\alpha, \beta)\bigr) \cap \R$.
\end{proof}
In higher dimensions, the method used in the proof is not applicable because there is no transfer matrix or equivalent to \cite[Theorem 2.12]{debruijn2024complexbandstructuresubwavelength} present. In fact, there might be evanescent modes associated to bulk band frequencies \cite{debruijn2024complexbandstructuresubwavelength,EvanescentWaves}.

\subsection{Complex Brillouin zone for a class of non-self-adjoint elliptic operators.}\label{sec: Complex Brillouin zone for non-self adjoint diff operators}
The complex Brillouin zone has been explored in various contexts, such as in the study of exponentially localised defect modes \cite[Section 2.2.5.]{debruijn2024complexbandstructuresubwavelength}. In this section, we aim to extend this idea to a broader category of non-Hermitian problems. We introduce a methodology to reformulate a set of non-Hermitian problems within a band gap framework.

Consider a one-dimensional resonator chain and let us denote the resonators by $D_i$ and define $D:= \cup_{i} D_i$.
We define the concept of the \emph{decay density}, also known as the \emph{local rate of exponential decay}, denoted by $\beta(x)$. This function characterises the decay rate at a given location $x$. Subsequently, we assume that $\beta(x)$ is a piecewise continuous and $\Lambda$-periodic function, meaning that it satisfies the relation $\beta(x + L) = \beta(x)$ for all $L \in \Lambda$. In this setting, we seek solutions $u$ of \eqref{eq:u_pde} with variable rate of decay in the sense that the function $v(x)$, defined as 
\begin{equation}\label{eq: change of functions variable beta}
    v(x) := e^{\int_0^x \beta(t) \d t} u(x),
\end{equation}
satisfies the Floquet condition with (real) quasimomentum $\alpha\in Y^*$. An important characteristic of the decay density is the following result.
\begin{lemma}\label{Cor: spatial decay}
    Let $u$ be an eigenmode with decay density $\beta(x)$, and let $L$ be the length of the unit cell, then $u$ satisfies the complex Floquet condition
    \begin{equation}\label{eq: average floquet condition}
        u(x + nL) = e^{\i ( \alpha + \i \overline{\beta}_L ) nL} u(x) , \ \forall n \in \N,
    \end{equation}
    where $\overline{\beta}_L := \frac{1}{L}\int_0^L \beta(x) \d x$ denotes the mean of $\beta(x)$ over the unit cell.
\end{lemma}

\begin{proof}
    Since $\beta(x + L) = \beta(x), \hspace{2mm} \forall L \in \Lambda,$ it follows
    \begin{equation}\label{eq: initial floquet condition}
        \int_x^{x + L} \beta(t) \d t = \int_0^L \beta(t) \d t.
    \end{equation}
    The total decay across $n$ successive unit cells can be expressed as 
    \begin{equation}\label{eq: Universal decay constant}
        \int_0^{nL} \beta(x) \d x = n \int_0^{L} \beta(x)\d x = \overline{\beta}_L n L,
    \end{equation}
    where the integral's linearity and the periodic condition of $\beta(x)$ were applied. Here, $\overline{\beta}_L$ represents the mean value of $\beta(x)$ over a single unit cell of length $L$.
    Using \eqref{eq: change of functions variable beta} together with \eqref{eq: Universal decay constant} results in the expression for the complex Floquet condition,
    \begin{align}
        u(x + nL) &= e^{-\int_0^{x+nL}\beta(t)\d t}v(x + nL)\\
        &= e^{-nL \overline{\beta}_L}e^{\i \alpha n L}e^{-\int_0^{x}\beta(t)\d t}v(x)\\
        &= e^{\i (\alpha + \i \overline{\beta}_L) n L} u(x).
    \end{align}
    The assertion follows.
\end{proof}

In other words, upon averaging across the unit cell, variable decay densities satisfy the complex Floquet condition. It is crucial to emphasise that an eigenmode described by \eqref{eq: average floquet condition} is exponentially localised with decay rate $\overline{\beta}_L$.

\begin{remark}\label{rem: slower decay rate beta not periodic}
 The key assumption of Lemma \ref{Cor: spatial decay} is that the decay density $\beta(x)$ is periodic with respect to the unit cell. In Section \ref{sec: Non-periodic gauge potential.}, we will present a counterexample where the assumption that $\beta(x)$ periodic is no longer met and we observe that the eigenmodes are no longer exponentially localised. The key mechanism that underpins the exponential decay is that a Bloch solution requires the medium to be fully periodic in space.
\end{remark}
Substituting \eqref{eq: change of functions variable beta} into the boundary value problem \eqref{eq:u_pde} yields,
\begin{equation}\label{eq: PDE variable beta}
    \begin{cases}
    \Delta v(x) - 2 \beta(x) \cdot \nabla v(x)  + \bigl(|\beta(x)|^2 - \nabla \beta(x) + k^2\bigr)v(x) = 0, \quad &\text{in } Y \setminus \overline{D},\\
    \Delta v(x) - 2 \beta(x) \cdot \nabla v(x)  + \bigl(|\beta(x)|^2 - \nabla \beta(x) + k_i^2\bigr)v(x) = 0, \quad &\text{in } D_i,\\
    \left. v \right|_- = \left. v \right|_+, &\text{on }\partial D,\\
    \ds\frac{\partial v}{\partial \nu}\bigg|_{-} - (\beta|_-\cdot \nu) v - \delta \left(\frac{\partial v}{\partial \nu} \bigg|_{+} - (\beta|_+\cdot \nu) v \right) = 0, & \text{on } \partial D, \\
     \ds v(x + \ell) = e^{\i \alpha \cdot \ell}v(x), &\text{for all }\ell \in \Lambda.
    \end{cases}
\end{equation}
Here, we have used the assumption that $\beta(x)$ is piecewise continuous, so that the cumulative integral up to $x$ is a continuous function of $x$.

The spatial dependence of $\beta(x)$ opens the door to the study of a variety of new problems, such as the investigation of non-homogeneous materials, where we introduce a gauge potential that mimics the effect of $\beta$. This will be examined in greater detail in the subsequent sections.

\subsection{One-dimensional localisation effects.}
This section aims to revisit key phenomena observed in one-dimensional resonator chains and establish a connection to the complex band structure of their infinite counterpart. Leveraging the extensive literature and numerical methods developed for the subwavelength regime, we numerically illustrate that the complex band structure accurately predicts the decay rate of a range of non-Hermitian problems.

\subsubsection{Subwavelength regime.}
The \emph{capacitance matrix} is a well-known tool to study subwavelength resonance (see, for example, \cite{ammari.fitzpatrick.ea2018Mathematical,ammari2024functional,Edge_Modes}). In the subwavelength limit, one considers a medium with high contrast inclusions, that is $ 0 < \rho_1 / \rho_0 = \delta \ll 1$. 

\begin{definition}
    Given $\delta > 0$, a \emph{subwavelength resonant frequency} $\omega = \omega(\delta)$ is defined to have non-negative real part and,
    \begin{enumerate}[(i)]
    \item there exists a non-trivial solution to the scattering problem \eqref{eq:u_pde}.
    \item $\omega$ depends continuously on $\delta$ and satisfies $\omega(\delta) \xrightarrow{\delta \to 0} 0$.
    \end{enumerate}
\end{definition}

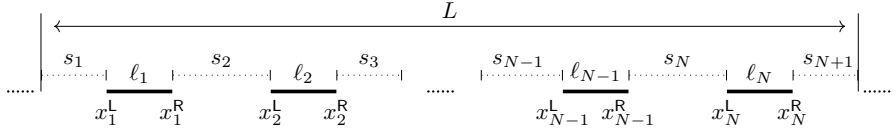
\begin{figure}[htb]
    \centering
    \begin{adjustbox}{width= 12cm} % \textwidth
    \begin{tikzpicture}
        \coordinate (x1l) at (1,0);
        \path (x1l) +(1,0) coordinate (x1r);
        \coordinate (s0) at (0.5,0.7);
        \path (x1r) +(0.75,0.7) coordinate (s1);
        
        \path (x1r) +(1.5,0) coordinate (x2l);
        \path (x2l) +(1,0) coordinate (x2r);
        \path (x2r) +(0.5,0.7) coordinate (s2);
        \path (x2r) +(1,0) coordinate (x3l);
        \path (x3l) +(1,0) coordinate (x3r);
        \path (x3r) +(1,0.7) coordinate (s3);
        \path (x3r) +(2,0) coordinate (x4l);
        \path (x4l) +(1,0) coordinate (x4r);
        \path (x4r) +(0.4,0.7) coordinate (s4);
        \path (x4r) +(1,0) coordinate (dots);
        \path (dots) +(1,0) coordinate (x5l);
        \path (x5l) +(1,0) coordinate (x5r);
        \path (x2r) +(1.2,0) coordinate (x6l);
        \path (x5r) +(0.875,0.7) coordinate (s5);
        \path (x6l) +(1,0) coordinate (x6r);
        \path (x6r) +(1.25,0) coordinate (x7l);
        \path (x6r) +(0.57,0.7) coordinate (s6);
        \path (x7l) +(1,0) coordinate (x7r);
        \path (x7r) +(1.5,0) coordinate (x8l);
        \path (x7r) +(0.75,0.7) coordinate (s7);
        \path (x8l) +(1,0) coordinate (x8r);
        \coordinate (s8) at (12,0.7);

        \draw[dotted, line cap=round, line width=1pt, dash pattern=on 0pt off 2\pgflinewidth] ($(x1l) - (1.5, 0)$) -- ($(x1l)- (1.1, 0)$);
        \draw ($(x1l)- (1, 0)$) -- ++(0,1.2);

        \draw[dotted,|-|] ($(x1l)- (1, -0.25)$) -- ($(x1l) + (0, 0.25)$);
        
        \draw[ultra thick] (x1l) -- (x1r);
        \node[anchor=north] (label1) at (x1l) {$x_1^{\iL}$};
        %\draw (x1l) -- ++(0,1.2);  % Line from node upwards
        \node[anchor=north] (label1) at (x1r) {$x_1^{\iR}$};
        \node[anchor=south] (label1) at ($(x1l)!0.5!(x1r)$) {$\ell_1$};
        \draw[dotted,|-|] ($(x1r)+(0,0.25)$) -- ($(x2l)+(0,0.25)$);
        \draw[ultra thick] (x2l) -- (x2r);
        \node[anchor=north] (label1) at (x2l) {$x_2^{\iL}$};
        \node[anchor=north] (label1) at (x2r) {$x_2^{\iR}$};
        \node[anchor=south] (label1) at ($(x2l)!0.5!(x2r)$) {$\ell_2$};
        \draw[dotted,|-|] ($(x2r)+(0,0.25)$) -- ($(x3l)+(0,0.25)$);
        \draw[dotted, line cap=round, line width=1pt, dash pattern=on 0pt off 2\pgflinewidth] ($(x3l)+(0.4,0)$) -- ($(x6r)-(0.4,0)$);
        
        \draw[dotted,|-|] ($(x6r)+(0,0.25)$) -- ($(x7l)+(0,0.25)$);
        \draw[ultra thick] (x7l) -- (x7r);
        \node[anchor=north] (label1) at (x7l) {$x_{N-1}^{\iL}$};
        \node[anchor=north] (label1) at (x7r) {$x_{N-1}^{\iR}$};
        \node[anchor=south] (label1) at ($(x7l)!0.5!(x7r)$) {$\ell_{N-1}$};
        \draw[dotted,|-|] ($(x7r)+(0,0.25)$) -- ($(x8l)+(0,0.25)$);
        \draw[ultra thick] (x8l) -- (x8r);
        \draw[dotted,|-|] ($(x8r)+(0,0.25)$) -- ($(x8r)+(1,0.25)$);
        \draw ($(x8r)+ (1,0)$) -- ++(0,1.2);  % Line from node upwards

        \draw[<->] ($(x1l) + (-0.8,1)$) -- ($(x8r) + (0.8,1)$) node[midway, above] {$L$};
    
        \draw[dotted, line cap=round, line width=1pt, dash pattern=on 0pt off 2\pgflinewidth] ($(x8r)+(1.1, 0)$) -- ($(x8r) + (1.5, 0)$);
        \node[anchor=north] (label1) at (x8l) {$x_{N}^{\iL}$};
        \node[anchor=north] (label1) at (x8r) {$x_{N}^{\iR}$};
        \node[anchor=south] (label1) at ($(x8l)!0.5!(x8r)$) {$\ell_N$};
        \node[anchor=north] (label1) at (s0) {$s_1$};
        \node[anchor=north] (label1) at (s1) {$s_2$};
        \node[anchor=north] (label1) at (s2) {$s_3$};
        
        \node[anchor=north] (label1) at (s6) {$s_{N-1}$};
        \node[anchor=north] (label1) at (s7) {$s_{N}$};
        \node[anchor=north] (label1) at (s8) {$s_{N+1}$};
    \end{tikzpicture}
    \end{adjustbox}
    \caption{Unit cell of an infinite chain of subwavelength resonators, with lengths
    $(\ell_i)_{1\leq i\leq N}$ and spacings $(s_{i})_{1\leq i\leq N+1}$.}
    \label{fig:setting}
\end{figure}
Within an infinite one-dimensional resonator chain, as depicted in Figure \ref{fig:setting}, the quasiperiodic capacitance matrix is defined as follows.

\begin{definition} Consider solutions $V^{\alpha, \beta}_i$ $:\mathbb{R} \to \R$ of the problem
\begin{equation}\label{eq: chain equation}
    \begin{cases}
        \frac{\mathrm{d}^2}{\mathrm{d}x^2}V^{\alpha, \beta}_i, = 0, & x\in Y \setminus D,\\
        V^{\alpha, \beta}_i(x) = \delta_{ij}, & x \in D_j, \\
        V^{\alpha, \beta}_i(x + mL) = e^{\i (\alpha + \i\beta)mL}V^{\alpha, \beta}_i(x), & m \in \Z,
    \end{cases}
\end{equation}
for $1 \leq i, j \leq N.$ Then the \emph{quasiperiodic capacitance matrix} is defined coefficient-wise by
\begin{equation}\label{eq: cap}
    {C}^{\alpha, \beta}_{ij} = - \int_{\partial D_i} \frac{\partial V^{\alpha, \beta}_j}{\partial \nu}\mathrm{d} \sigma,
\end{equation}
    where $\nu$ is the outward-pointing normal.
\end{definition}

Following \cite[Lemma 2.4.]{debruijn2024complexbandstructuresubwavelength} the capacitance matrix \eqref{eq: cap} for a one dimensional resonator chain is given by
\begin{equation}\label{def: capacitance_matrix}
    {C}^{\alpha, \beta}_{ij} = -\frac{1}{s_{j-1}}\delta_{i(j-1)} + \left(\frac{1}{s_{j-1}} + \frac{1}{s_j}\right)\delta_{ij} - \frac{1}{s_j}\delta_{i(j+1)} - \delta_{1j}\delta_{iN} \frac{e^{-\i (\alpha + \i\beta) L}}{s_N} - \delta_{1i}\delta_{jN} \frac{e^{\i (\alpha + \i\beta) L}}{s_N}.
\end{equation}
The subwavelength resonances and eigenmodes can now be determined following \cite[Proposition 2.5]{debruijn2024complexbandstructuresubwavelength}.

\begin{proposition}\label{prop: subwavelength band functions}
    Assume that the eigenvalues of the generalised capacitance matrix $ \mathcal{C}^{\alpha, \beta}$ are simple. Then there are $N$ subwavelength band functions $(\alpha, \beta) \mapsto \omega_i^{\alpha, \beta}$ which satisfy
    \begin{equation} \label{eq:sqrt}
        \omega^{\alpha, \beta}_i = \sqrt{\delta \lambda^{\alpha, \beta}_i} + \mathcal{O}(\delta),
    \end{equation}
    where $(\lambda^{\alpha,\beta}_i)_{1 \leq i \leq N}$ are the eigenvalues of $ \mathcal{C}^{\alpha, \beta}$. Here, we choose the branch of \eqref{eq:sqrt} with positive imaginary parts.
\end{proposition}

\subsubsection{Non-Hermitian skin effect.} \label{Sec: Non-Hermitian skin effect 1D}
The non-Hermitian skin effect describes the localisation of a significant fraction of bulk eigenmodes at a boundary within an open resonator system. By incorporating a complex gauge potential $\gamma$, the hermiticity of a subwavelength resonator arrangement is disrupted \cite{yokomizo.yoda.ea2022NonHermitian}. This modification leads to a generalised Sturm-Liouville problem. The Helmholtz equation that governs the system is given by 

\begin{equation}\label{eq: wave equation skin effect}
\begin{cases}u^{\prime \prime}(x)+\gamma u^{\prime}(x)+\frac{\omega^2}{v^2} u=0, & x \in D \text{ and } \gamma \in \R^*, \\ u^{\prime \prime}(x)+\frac{\omega^2}{v_b^2} u=0, & x \in \mathbb{R} \backslash D, \\ \left.u\right|_{\mathrm{R}}\left(x_i^{\mathrm{L}, \mathrm{R}}\right)-\left.u\right|_{\mathrm{L}}\left(x_i^{\mathrm{L}, \mathrm{R}}\right)=0, & \text {for all } 1 \leq i \leq N, \\ \left.\frac{\mathrm{d} u}{\mathrm{~d} x}\right|_{\mathrm{R}}\left(x_i^{\mathrm{L}}\right)=\left.\delta \frac{\mathrm{d} u}{\mathrm{~d} x}\right|_{\mathrm{L}}\left(x_i^{\mathrm{L}}\right), & \text {for all } 1 \leq i \leq N, \\  \left.\frac{\mathrm{d} u}{\mathrm{~d} x}\right|_{\mathrm{L}}\left(x_i^{\mathrm{R}}\right)=\left. \delta \frac{\mathrm{d} u}{\mathrm{~d} x}\right|_{\mathrm{R}}\left(x_i^{\mathrm{R}}\right), & \text {for all } 1 \leq i \leq N, \\ 
\text{(complex) quasiperiodicity condition.}
\end{cases}
\end{equation}

For a finite resonator chain with a gauge potential applied to the resonators, the gauge capacitance matrix can be used to determine the resonances and eigenmodes within this chain \cite[Section 2.2.]{ammari2023mathematicalSkin}. For a variable gauge potential applied on a finite resonator chain of length $N$ we introduce the adjusted gauge capacitance matrix,

\begin{align}
    \capmat_{i,j}^{(\gamma_k)_{1 \leq k \leq N}} \coloneqq \begin{dcases}
        \frac{\gamma_1}{s_1} \frac{\ell_1}{1-e^{-\gamma_1 \ell_1}}, & i=j=1,\\
         \frac{\gamma_i}{s_i} \frac{\ell_i}{1-e^{-\gamma_i \ell_i}} -\frac{\gamma_{i}}{s_{i-1}} \frac{\ell_i}{1-e^{\gamma_i \ell_i}},  & 1< i=j< N,\\
       - \frac{\gamma_j}{s_i} \frac{\ell_i}{1-e^{-\gamma_j \ell_j}},  & 1\leq i=j-1\leq N-1,\\
       \frac{\gamma_j}{s_j} \frac{\ell_i}{1-e^{\gamma_j \ell_j}}, & 2\leq i=j+1\leq N,\\
      - \frac{\gamma_{N}}{s_{N-1}} \frac{\ell_N}{1-e^{\gamma_N \ell_N}}, & i=j=N.\\
    \end{dcases}\label{eq: variable gauge capacitance}
\end{align} 
The subwavelength resonances are now detailed analogously to Proposition \ref{prop: subwavelength band functions}. For a precise analysis, refer to \cite[Corollary 2.6]{ammari2023mathematicalSkin}.

The boundary value problem \eqref{eq: wave equation skin effect} corresponds to \eqref{eq: PDE variable beta} for $\beta(x)$ piecewise constant
\begin{equation} \beta(x) =
    \begin{cases}
        \frac{\gamma}{2}, \quad &x \in D^\circ,\\
        0, &\text{else. }
    \end{cases}
\end{equation}
 This becomes evident upon substituting $v(x) = e^{\beta(x) } u(x)$ into \eqref{eq: wave equation skin effect},
\begin{equation}
    \begin{cases}
        v^{\prime \prime}(x) + \frac{\omega^2}{v^2} v  = 0, & x \in D, \\
        v^{\prime \prime}(x) + \frac{\omega^2}{v^2} v  = 0, &x \in \mathbb{R} \backslash D, \\
        \left.u\right|_{\mathrm{R}}\left(x_i^{\mathrm{L}, \mathrm{R}}\right)-\left.u\right|_{\mathrm{L}}\left(x_i^{\mathrm{L}, \mathrm{R}}\right)=0, & \text {for all } 1 \leq i \leq N, \\
        \ds\frac{\partial v}{\partial x}\bigg|_{-} - \beta(x) v\bigg|_{-} - \delta \left(\frac{\partial v}{\partial x} \bigg|_{+} - \beta(x)v\bigg|_{+} \right) = 0, & \text{on } \partial D,\\
        v(x + \ell ) = e^{\i \alpha \ell}v(x), & \text{for all } \ell \in \Lambda.
    \end{cases}
\end{equation}

In the case of a fully periodic resonator arrangement, the gauge potential $\gamma$ and consequently $\beta(x)$ are $\Lambda$-periodic.
The complex quasimomentum, which dictates the decay across one unit cell of length $L$, can be computed via Lemma \ref{Cor: spatial decay} and is given by,
\begin{align}
    \overline{\beta}_L &= \frac{1}{L} \int_0^L \beta(x) \d x\label{eq: decay condition continuous}\\
    &= \frac{1}{2L } \sum_{i = 1}^N \ell_i \gamma,\label{eq: Decay condition}
\end{align}
where in the last step we leverage the fact that the gauge potential remains piecewise constant on the specified resonators. 

\subsubsection{Numerical illustrations.}
We investigate a dimer and trimer arrangement of subwavelength resonators composed of $20$ replicated unit cells. We use the gauge capacitance matrix (see \cite[Definition 2.4.]{ammari2023mathematicalSkin}) to compute the eigenmodes of a finite resonator chain.

\begin{figure}[htb]
    \centering
    \subfloat[][Periodic trimer with length $\ell_1 = \ell_3 = 1$, $ \ell_2 = 2$, $\gamma = 1$ and spacings $s_1 = s_2 = s_3 = 3$.]{{\includegraphics[width=0.3\linewidth]{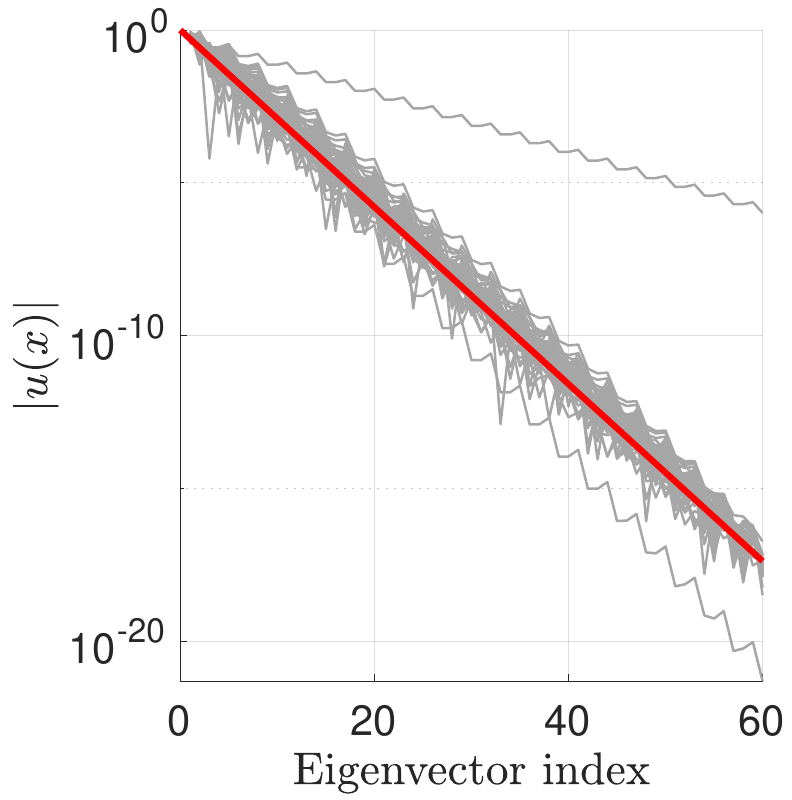} }}\quad
    \subfloat[][Periodic dimers with length $\ell_1 = \ell_2 = 1$, $\gamma = 1$ and spacings $s_1 = 1, s_2 = 2$.]{{\includegraphics[width=0.3\linewidth]{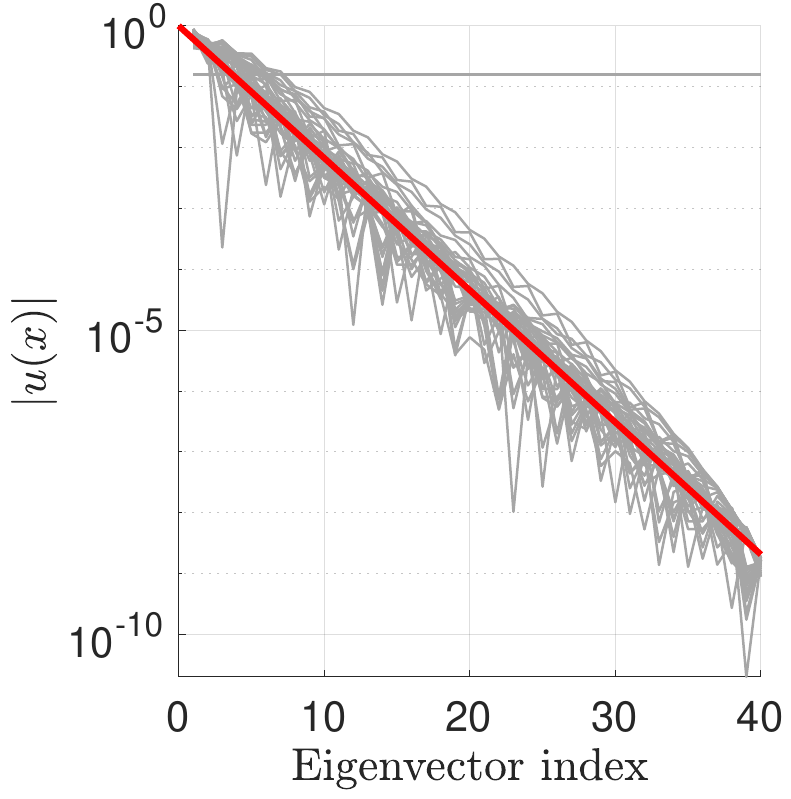}}} \quad 
    \subfloat[][Periodic dimers with length $\ell_1 = \ell_2 = 1$, $\gamma = 2$ and spacings $s_1 = 1, s_2 = 2$.]{{\includegraphics[width=0.3\linewidth]{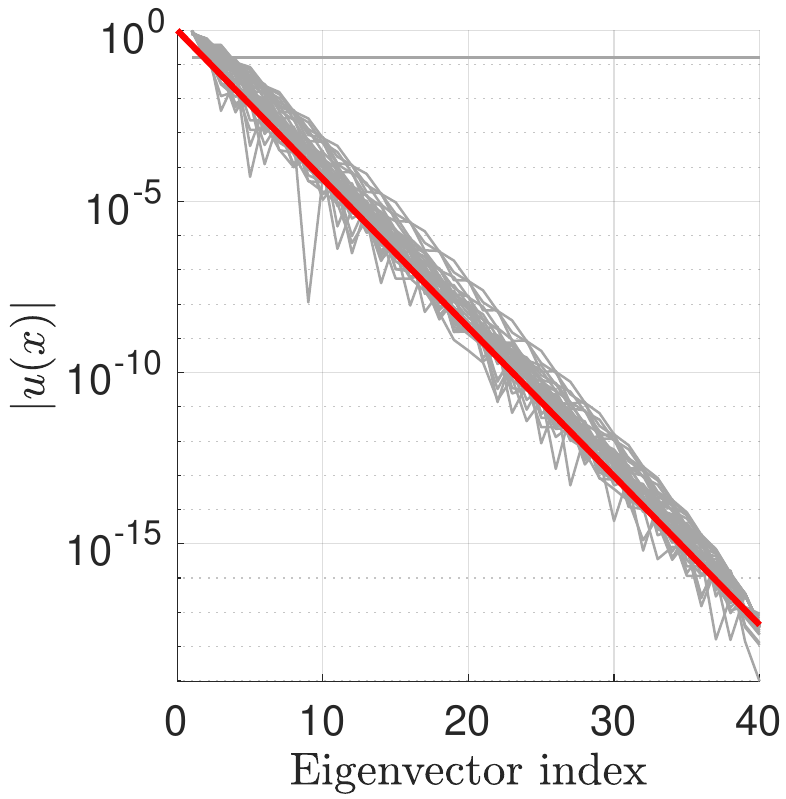}}}
    \caption{The complex quasimomentum (red line) given by \eqref{eq: Decay condition} accurately predicts the exponential decay rates of the eigenmodes superposed onto each other (black lines). In each case, there is a single outlier arising from the fact that the one-dimensional capacitance matrix has a one-dimensional kernel with distinct decay properties.}
   \label{Fig: Skin effect}
\end{figure} 

The numerical example presented in Figure \ref{Fig: Skin effect} demonstrates that, even for relatively small resonator systems, the complex quasimomentum of an infinite configuration accurately predicts the finite system's localisation length with notable accuracy. 
For a constant gauge potential $\gamma$, equation \eqref{eq: Decay condition} provides a necessary and sufficient condition for the eigenvectors to condensate towards one side of the system. This occurs because the quasimomentum becomes complex if and only if $\overline{\beta}_L \neq 0$. This is closely related to the decay condition presented in \cite{ammari2024spectrapseudospectratridiagonalktoeplitz}, where a connection was established between a non-trivial winding number of the differential operator and the existence of an exponentially localised eigenmode. 

\subsubsection{Non-periodic gauge potential.}\label{sec: Non-periodic gauge potential.} This section illustrates the idea outlined in Remark \ref{rem: slower decay rate beta not periodic} and considers a resonator chain where the unit cell is not periodically repeated. Therefore, we introduce a gauge potential that is not periodic with respect to the unit cell and is described by
\begin{equation}
    \gamma_i = \frac{a}{1 + i}.
\end{equation}
As a consequence, it is no longer sufficient to calculate the mean of the gauge potential over a single unit cell. Instead, we average the eigenmode over $n$ consecutive unit cells of length $L$. We assume the resonator lengths inside the unit cell satisfy $\sum_{i = 1}^N \ell_i = 1$ and obtain the following bounds
\begin{equation}
    \tilde{\beta}_{nL} = \frac{1}{2}\sum_{i = 1}^n \frac{a}{1 + i} \leq \int_0^{n} \frac{a}{1 + x} \d x = \frac{a}{2} \log(1+n),\hspace{2mm}\forall n \in \N.
\end{equation}
Similarly, we get a lower bound on the decay rate
\begin{equation}\label{eq: lower decay bound}
    \tilde{\beta}_{nL} = \frac{1}{2}\sum_{i = 1}^n \frac{a}{1 + i} \geq \frac{1}{2} \int_1^{n + 1} \frac{a}{1 + x}\d x = a \log\left(\frac{1 + n}{2}\right),\hspace{2mm}\forall n \in \N,
\end{equation}
and an eigenmode $u$ has to satisfy the following relation,
\begin{equation}\label{eq: algebraic decay rate}
    u(x + n) \leq e^{\i \alpha n -\frac{a}{2}\log(1+n)}u(x) = \frac{1}{(1 + n)^{\frac{a}{2}}} e^{\i \alpha n}u(x).
\end{equation}
To put it differently, an eigenmode is algebraically localised with decay rate $\frac{a}{2}$ and due to the bound \eqref{eq: lower decay bound} the decay rate cannot be faster than algebraic. The eigenmodes of a finite resonator chain can be determined by applying the adjusted gauge capacitance matrix \eqref{eq: variable gauge capacitance}.

\begin{figure}[htb]
    \centering
    \subfloat[][Finite resonator chain with resonator length $\ell = 1$, spacings $s = 1$ and $a = 10$.]{{\includegraphics[width=0.31\linewidth]{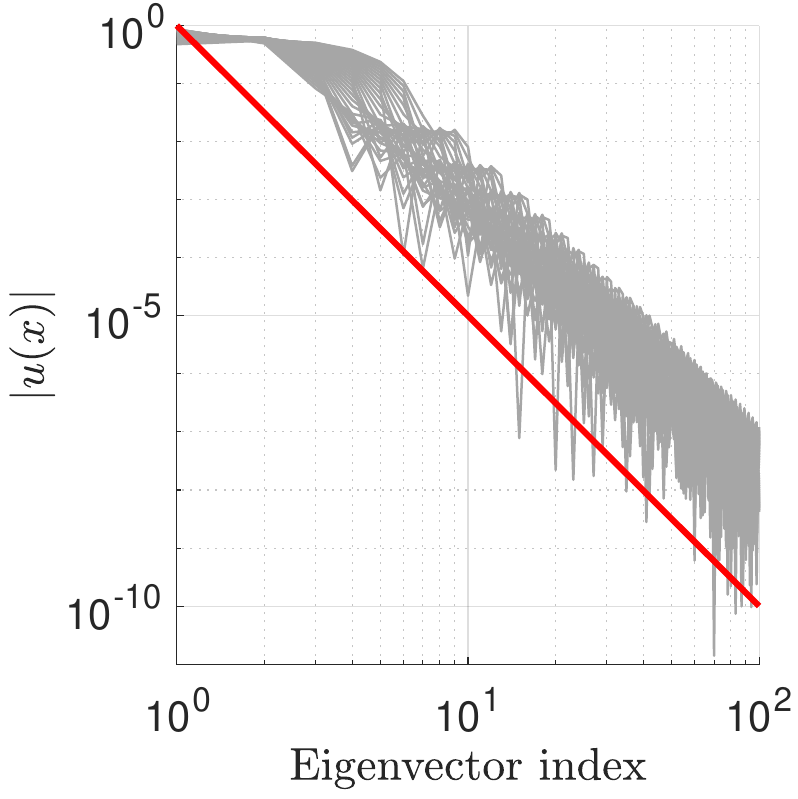} }}\quad
    \subfloat[][Finite dimer chain with $\ell_1 = \ell_2 = 0.5$, spacings $s_1 = 1, s_2 = 2$ and $a = 5$.]{{\includegraphics[width=0.3\linewidth]{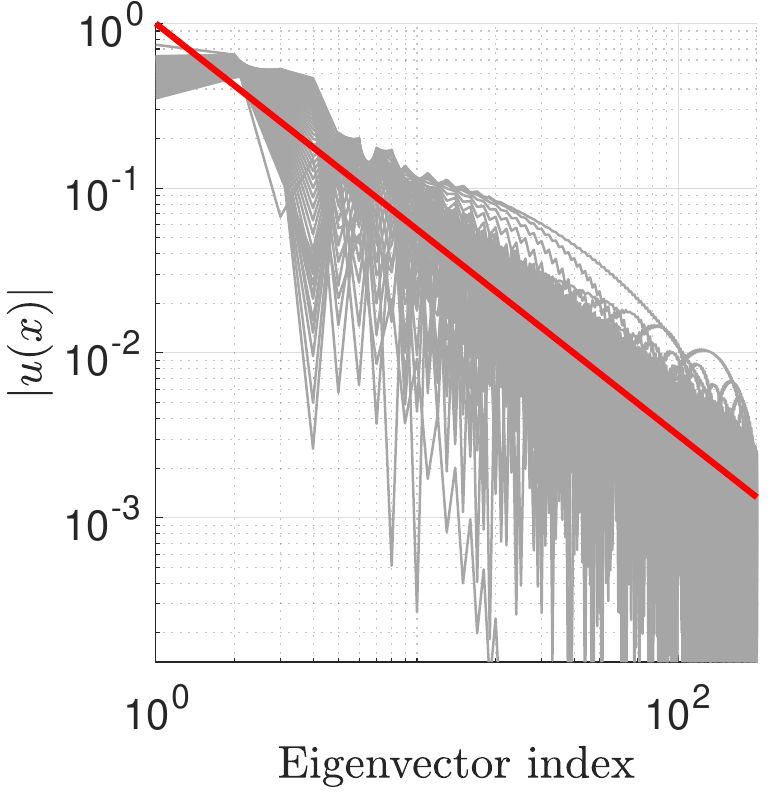}}} \quad 
    \subfloat[][Finite dimer chain with $\ell_1 = \ell_2 = 0.5$, spacings $s_1 = 1, s_2 = 2$ and $a = 3$.]{{\includegraphics[width=0.3\linewidth]{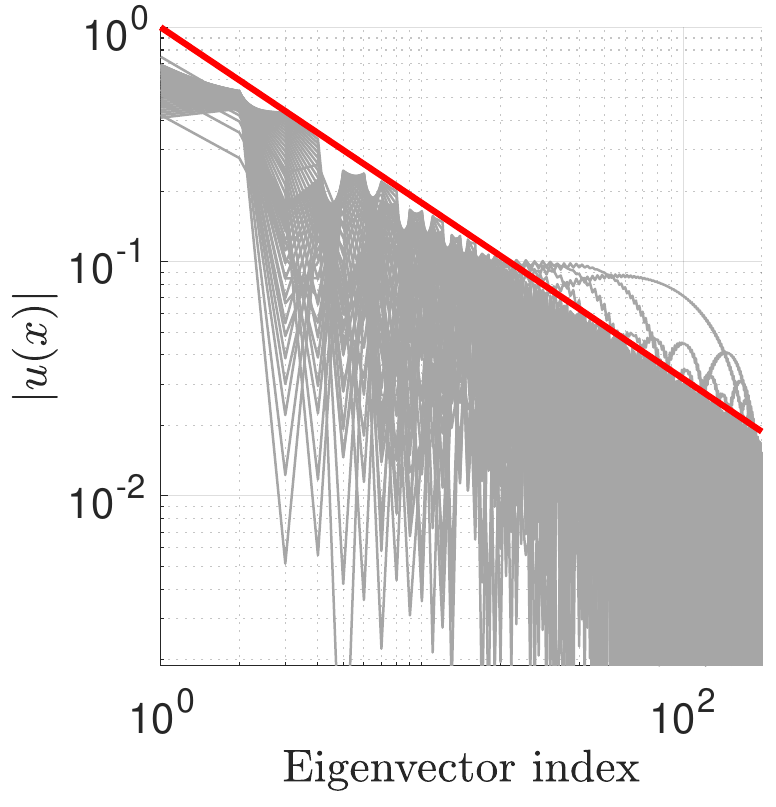}}}
    \caption{In contrast to Figure \ref{Fig: Skin effect}, when the gauge potential lacks lattice periodicity, the eigenmodes exhibit algebraic decay at a rate given in \ref{eq: algebraic decay rate}. This is because the system is no longer strictly periodic and there is no complex band structure promoting the exponential decay of gap modes.}
   \label{Fig: Skin effect algebraic decay}
\end{figure} 

The simulation carried out in Figure \ref{Fig: Skin effect algebraic decay} illustrates that the eigenmodes continue to cluster at one boundary of the system, adhering to the standard definition of the skin effect. However, the intensity of localisation transitions from an exponential decay to an algebraic one.
The unit cell lacks spatial periodicity, preventing the formation of a Bloch solution, and no complex band structure exists to secure the exponential decay of an eigenmode. In Section \ref{sec: Existence of band gaps} we will consider a similar scenario, where the resonators are not periodically repeated and where there is no complex band structure to ensure the exponential decay gap modes.

\subsubsection{Random gauge potential}\label{sec: random Gauge potential} In the following section, we consider the non-Hermitian skin effect with a random gauge potential applied to the resonators. The unit cells are indexed by $i$, whereas the $N$ the resonators inside unit cell are indexed by $(D_{i_j})_{1 \leq j \leq N}$ and have lengths $(\ell_j)_{1 \leq j \leq N}$, then $D := \cup_{i \geq 1}\cup_{1 \leq j \leq N} D_{i_j}$.  Let $(\gamma_{i_j})_{1 \leq j \leq N}$ be a sequence of independently identically distributed random variables following some probability distribution and $\mathbb{E}[\gamma_{i_j}] =: \tilde{\gamma} < \infty$. The wave equation with a random gauge potential is described by,
\begin{equation}\label{eq: Random wave equation skin effect}
    \begin{cases}u^{\prime \prime}(x)+\gamma_{i_j} u^{\prime}(x)+\frac{\omega^2}{v^2} u=0, & x \in D_{i_j}, \\ u^{\prime \prime}(x)+\frac{\omega^2}{v_b^2} u=0, & x \in \mathbb{R} \backslash D, \\ \left.u\right|_{\mathrm{R}}\left(x_j^{\mathrm{L}, \mathrm{R}}\right)-\left.u\right|_{\mathrm{L}}\left(x_j^{\mathrm{L}, \mathrm{R}}\right)=0, & \text {for all } 1 \leq j \leq N, \\ \left.\frac{\mathrm{d} u}{\mathrm{~d} x}\right|_{\mathrm{R}}\left(x_j^{\mathrm{L}}\right)=\left.\delta \frac{\mathrm{d} u}{\mathrm{~d} x}\right|_{\mathrm{L}}\left(x_j^{\mathrm{L}}\right), & \text {for all } 1 \leq j \leq N, \\  \left.\frac{\mathrm{d} u}{\mathrm{~d} x}\right|_{\mathrm{L}}\left(x_j^{\mathrm{R}}\right)=\left. \delta \frac{\mathrm{d} u}{\mathrm{~d} x}\right|_{\mathrm{R}}\left(x_j^{\mathrm{R}}\right), & \text {for all } 1 \leq j \leq N, \\
    \text{(complex) quasiperiodicity condition.}
    \end{cases}
\end{equation}
The following result now characterises the macroscopic properties of an eigenmode of \eqref{eq: Random wave equation skin effect}.
\begin{theorem}
    Let $(\gamma_{i_j})_{i \geq 1}$ be a sequence of independently identically distributed random variables following some probability distribution and $\mathbb{E}[\gamma_{i_j}] =: \tilde{\gamma} < \infty$. Let $u$ be an eigenmode of \eqref{eq: Random wave equation skin effect}, then the complex quasimomentum converges almost surely to 
    \begin{equation}\label{eq: expected decay rate skin effect}
        \mathbb{E}[\tilde{\beta}_i] = \frac{1}{2L} \sum_{j = 1}^N\ell_j \tilde{\gamma},
    \end{equation}
    where $\tilde{\beta_i}$ is a random variable denoting the decay over the $i$-th unit cell defined by
    \begin{equation}
        \tilde{\beta}_i := \frac{1}{2L} \sum_{j=1}^N \gamma_{i_j} \ell_j,
    \end{equation}
    such that the complex Floquet condition converges almost surely to
    \begin{equation}
        u(x + L) = e^{\i (\alpha +  \i \mathbb{E}[\tilde{\beta}_i]) L} u(x), \hspace{2mm}\forall L \in \Lambda.
    \end{equation}
\end{theorem}

\begin{proof}
    By Lemma \ref{Cor: spatial decay} respectively \eqref{eq: Decay condition} the decay rate over the $i$-th unit cell is given by 
    \begin{equation}
    \tilde{\beta}_i = \frac{1}{2L} \sum_{j=1}^N \gamma_{i_j} \ell_j.
    \end{equation}
    When passing over $n$ unit cells of length $L$, the complex Floquet condition is 
    \begin{equation}\label{eq: decay a.s. convergence}
        u(x + n L) = e^{\i \left(n \alpha  + \i \sum_{i = 1}^n \tilde{\beta_i} \right) L} u(x), \hspace{2mm} \forall n \in \N, \forall L \in \Lambda.
    \end{equation}
    By the strong law of large numbers, almost surely 
    \begin{equation}\label{eq: a.s. convergence to mean}
        \frac{\sum_{i = 1}^n \tilde{\beta_i}}{n} \xrightarrow{\text{a.s.}} \mathbb{E}[\tilde{\beta_i}].
    \end{equation}
    Consequently, the complex Floquet condition converges almost surely to,
    \begin{equation}
        u(x + n L) \xrightarrow{\text{a.s.}} e^{\i\bigl(\alpha + \i\mathbb{E}[\tilde{\beta_i}]\bigr) n L}u(x).
    \end{equation}
    A direct computation yields that the expected decay over one unit cell can be evaluated via,
    \begin{equation}
        \mathbb{E}[\tilde{\beta_i}] = \mathbb{E}\left[\frac{1}{2L}\sum_{j = 1}^N \gamma_{i_j}\ell_j\right]
        = \frac{1}{2L}\sum_{j = 1}^N \ell_j \mathbb{E}[\gamma_{i_j}]
        = \frac{1}{2L}\sum_{j = 1}^N \ell_j \tilde{\gamma}.
    \end{equation}
    The assertion follows.
\end{proof}

We numerically illustrate the skin effect in a large chain of resonators. The gauge potential $(\gamma_{i_j})_{1 \leq j \leq N}$ is an independently and identically distributed random variable following a uniform probability distribution. Such a resonator chain is illustrated in Figure \ref{fig: Chain random Gauge potential}.

\begin{figure}[tbh]
    \centering
    \includegraphics[width=0.70\linewidth]{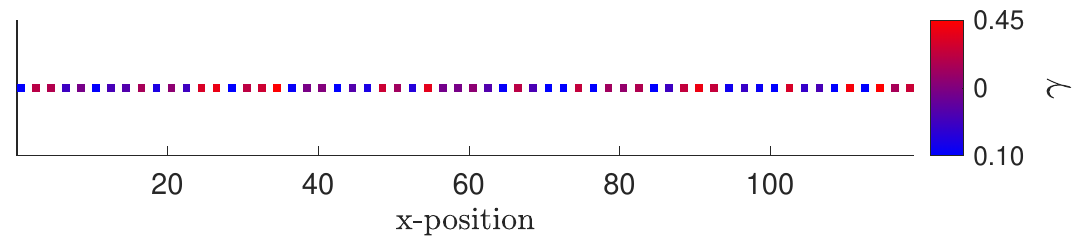}
    \caption{Resonator chain with random gauge potential. $60$ monomers with $s_1 = \ell_1 =1$ and gauge potential with expectation $\mathbb{E}[\gamma_{i_j}] = 0.25.$}
    \label{fig: Chain random Gauge potential}
\end{figure}

The eigenmodes of a finite resonator chain can be determined by applying the adjusted gauge capacitance matrix \eqref{eq: variable gauge capacitance}, for a gauge potential given by the random variable $(\gamma_{i_j})_{1 \leq j \leq N}$, and the results are shown in Figure \ref{Fig:SkinEffect}.
\begin{figure}[htb]
    \centering
    \subfloat[][Length $\ell_1 = \ell_2 = 1$, $\mathbb{E}\text{[}\gamma_{i_j}\text{]} = 0.25$ and spacings $s_1 = s_2 = 1$.]%
    {\includegraphics[width=0.3\linewidth]{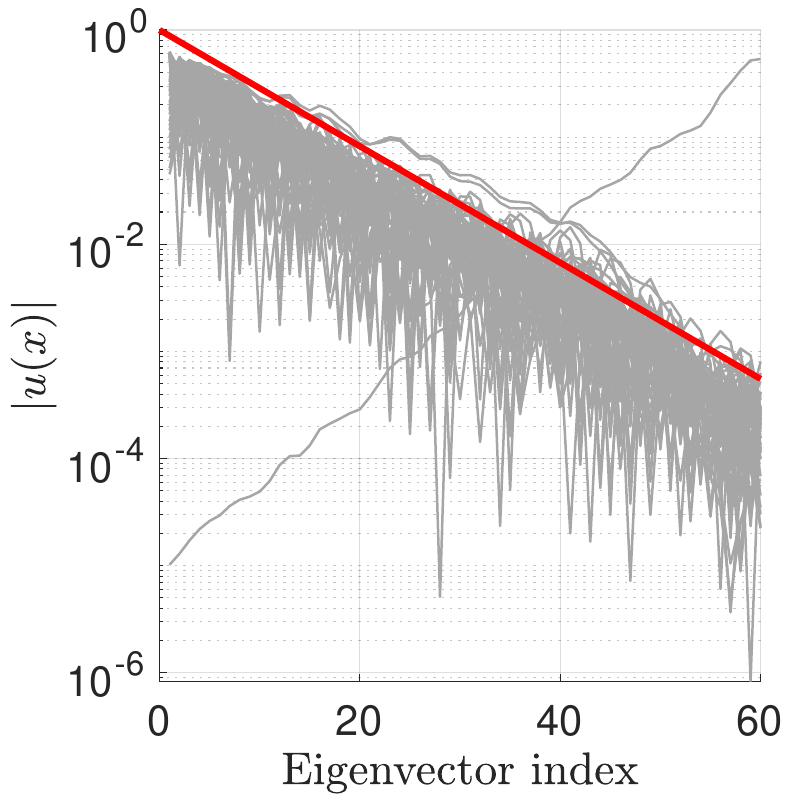}}\quad
    \subfloat[][Length $\ell_1 = \ell_2 = 1$, $\mathbb{E}\text{[}\gamma_{i_j}\text{]} = 0.25$ and spacings $s_1 = 1, s_2 = 2$.]%
    {\includegraphics[width=0.3\linewidth]{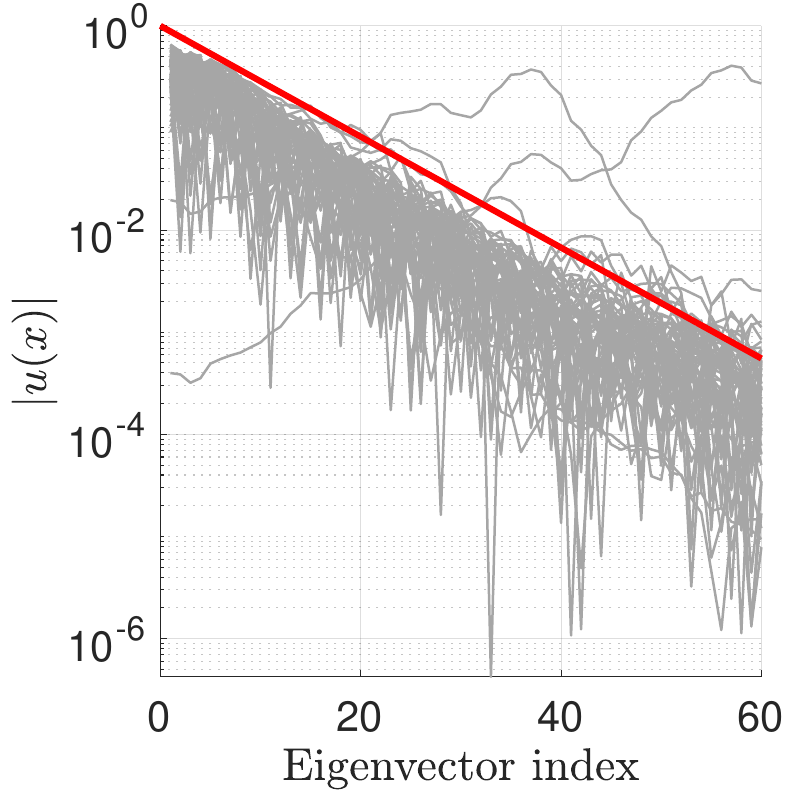}}\quad
    \subfloat[][Length $\ell_1 = \ell_2 = 1$, $\mathbb{E}\text{[}\gamma_{i_j}\text{]} = 0.55$ and spacings $s_1 = 1, s_2 = 2$.]%
    {\includegraphics[width=0.3\linewidth]{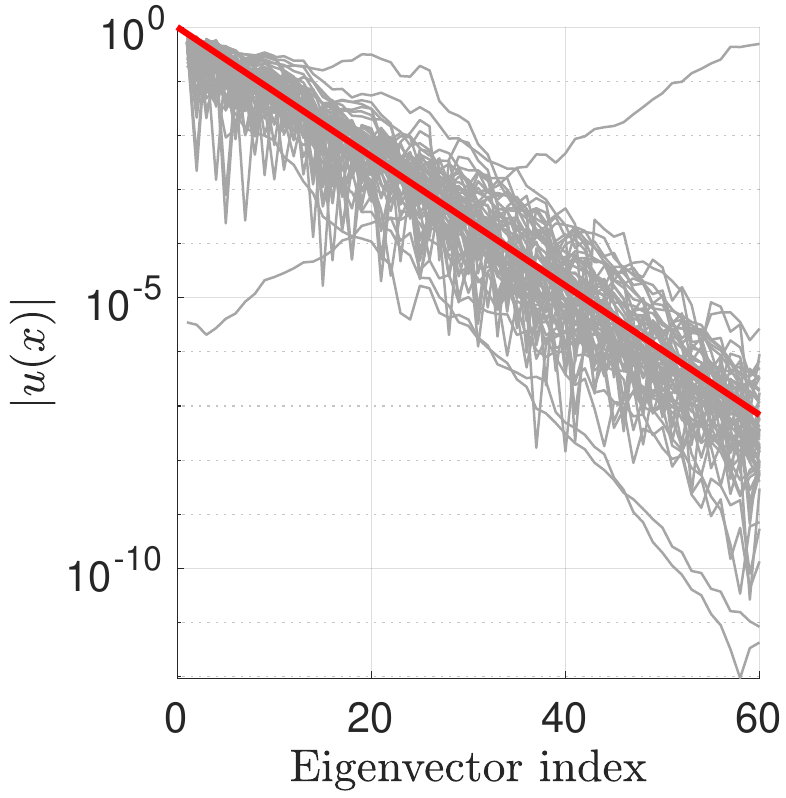}}
    \caption{The non-Hermitian skin effect with a random gauge potential applied to the resonators. The exponential decay rate given     by \eqref{eq: expected decay rate skin effect} (red) accurately forecasts the exponential decay rate of the eigenmodes (black) that were superimposed onto each other. In each case, there is a single outlier arising from the fact that the one-dimensional capacitance matrix has a one-dimensional kernel.}
    \label{Fig:SkinEffect}
\end{figure}

\subsubsection{Disordered chains via random spacings.} Consider a resonator chain composed of monomers (i.e. a fundamental cell with a single resonator), where dimers (i.e. a fundamental cell with a two resonators) are introduced at random positions. In this configuration, two competing localisation effects emerge. 
\begin{figure}[tbh]
    \centering
    \subfloat[][Monomer in a unit cell.]{{\includegraphics[width=0.34\linewidth]{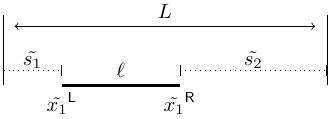} }}\qquad
    \subfloat[][Dimer in a unit cell.]{{\includegraphics[width=0.34\linewidth]{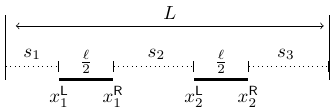} }\vspace{0.5mm} }
    \caption{The unit cells are chosen such that the cumulative length of the resonator contained within them equals $\ell$.}
   \label{Fig: Monomers and dimers}
\end{figure} 

On the one hand, the non-Hermitian skin effect causes the eigenmodes to condensate on the left side of the system.
On the other hand, the introduction of dimers into the system supports defect modes, where the eigenmode decays exponentially away from the introduced defects.
The decay rate of the composite structure can be predicted as follows. Consider introducing  a small number of dimers at arbitrary sites $i$ within the resonator chain. The exponential decay rate of the eigenmodes is predicted based on the eigenfrequency $\omega$,
    \begin{equation}\label{eq: composite decay skin}
        \tilde{\beta}(\omega) = - \frac{\gamma}{2}\ell \pm\beta(\omega)\lvert x-i\rvert,
    \end{equation}
where
    \begin{equation}\label{eq: gap function monomer}
        \beta(\omega) = \pm\frac{1}{L}\cosh^{-1}\left( \frac{\omega^2}{2\delta}-1 \right)
    \end{equation}
    is the gap band associated with a monomer chain (see \cite[Section 2.2.2.]{debruijn2024complexbandstructuresubwavelength}).
    
    To understand why \eqref{eq: composite decay skin} determines the decay rate, note that the resonator chain can be viewed as a monomer chain with isolated dimer defects. In this context, $\beta(\omega)$ represents the gap function, as defined in equation \eqref{eq: gap function monomer}.
    Consider the wave equation \eqref{eq: PDE variable beta}, choosing the decay density as
    \begin{equation}
        \beta(x) = \begin{cases}
            \frac{\gamma}{2}, & x \in D,\\
            +\beta(\omega), & x \leq i,\\
            -\beta(\omega), & x  > i,
        \end{cases}
    \end{equation}
and evaluating the cumulative decay rate over the unit cell via
    \begin{equation}
        \tilde{\beta}(\omega) = \int_0^L \beta(x) \d x
    \end{equation}
which yields the result \eqref{eq: composite decay skin}. In other words, a frequency $\omega$ that resides in the bulk of the monomer one finds $\beta(\omega) = 0$, thus the only effect that contributes to the decay comes from the skin effect. On the other hand, if a dimer is introduced, the top spectral band of the dimer supports an eigenfrequency $\omega$  which lies in the gap of the monomer, implying $\beta(\omega) \neq 0$. As a consequence, we observe a composite decay composed of the skin effect and defect localisation.

Next, we illustrate our findings numerically.
\begin{figure}[htb]
    \centering
    \includegraphics[width=0.7\linewidth]{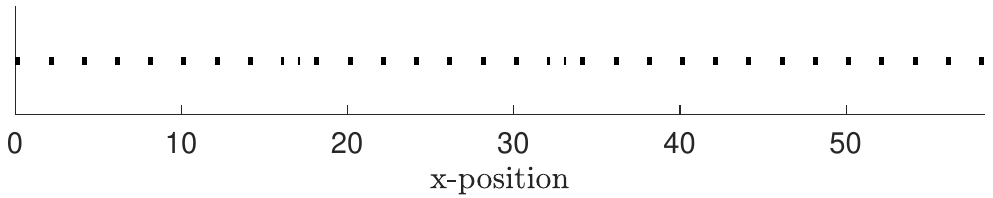}
    \caption{Randomly sampled dimers and monomers with a constant gauge potential applied to the resonators.}
    \label{fig: Disordered monomer and dimer}
\end{figure}
The monomers consist of resonators with length $\ell$ and spacings of length $s$, while the dimers are composed of two resonators, each with length $\ell/2$ and spacings of length $s/2$ (see Figure \ref{Fig: Monomers and dimers}). In this configuration, each unit cell has a total resonator length of $\ell$ while maintaining a consistent total unit cell length $L$. An example outcome of the disordered system is shown in Figure \ref{fig: Disordered monomer and dimer}.
\begin{figure}[htb]
    \centering
    \subfloat[][Eigenmodes associated with frequencies in the band of the monomer, these are superposed over each other.]{{\includegraphics[width=0.4\linewidth]{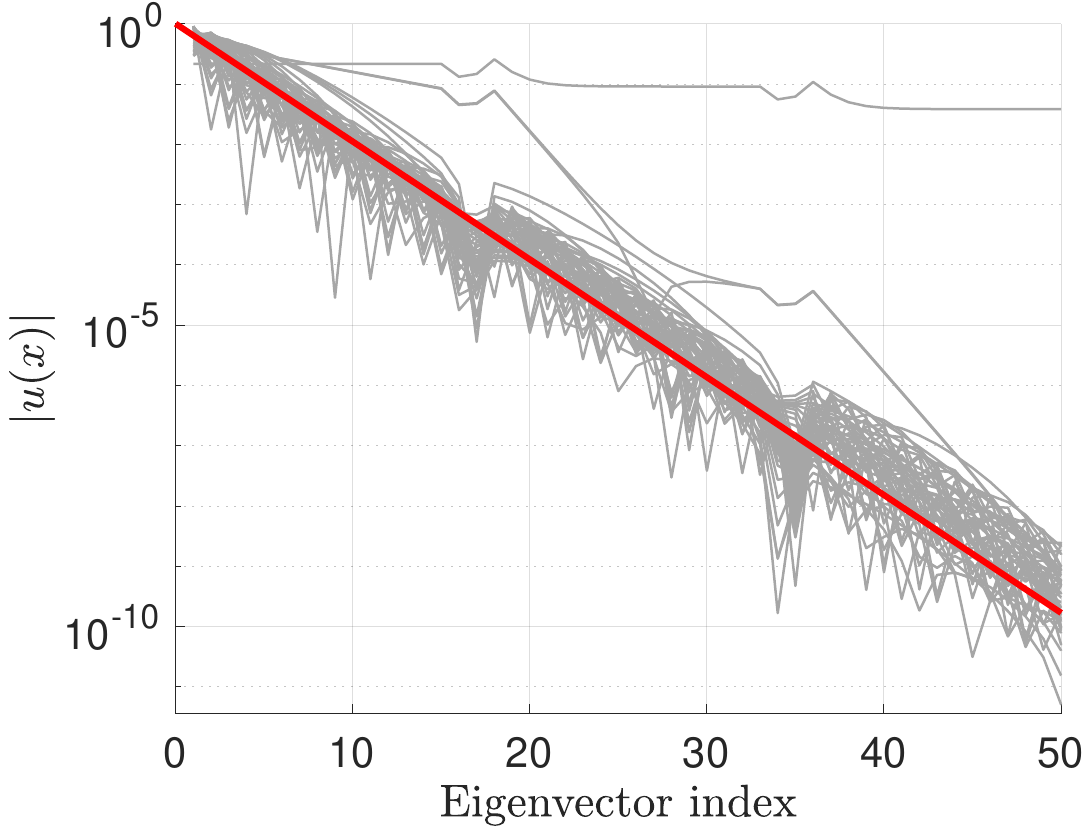} }}\qquad
    \subfloat[][Eigenmode associated to a frequency in the band gap of the monomer, supported by the top spectral band of the dimer.]{{\includegraphics[width=0.4\linewidth]{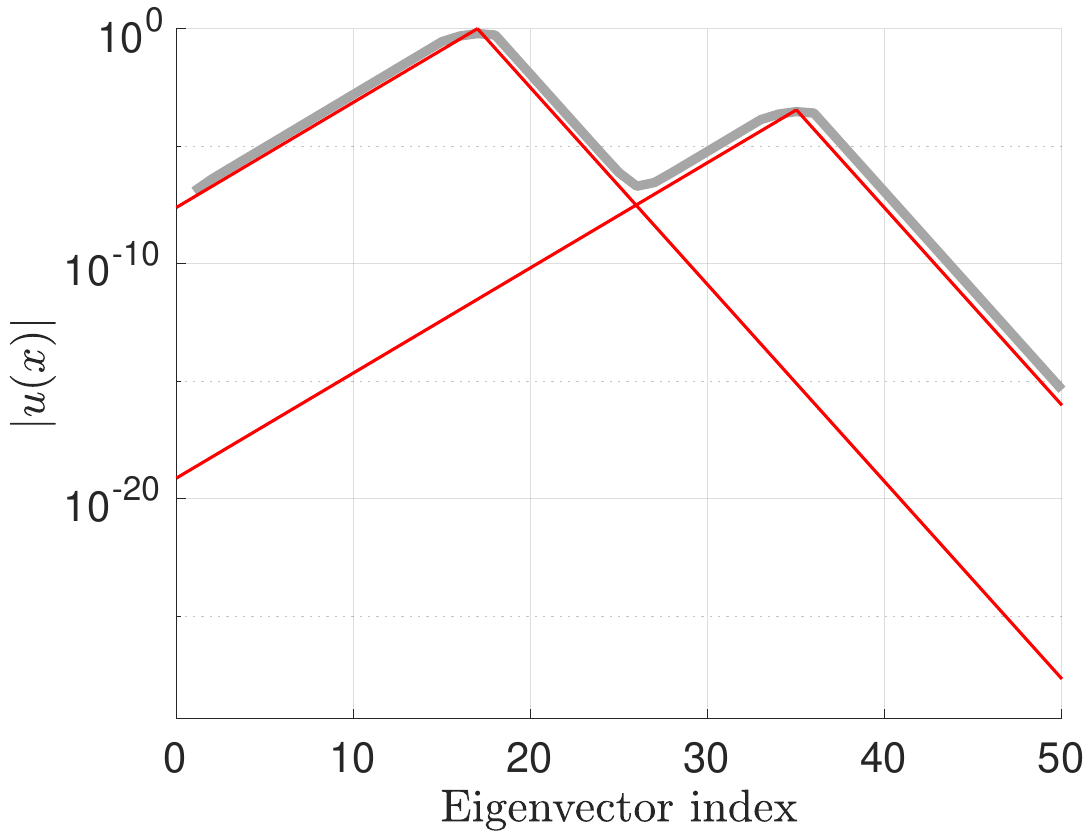} }} 
    \caption{Eigenmodes of the resonator chain presented in Figure \ref{fig: Disordered monomer and dimer}. (A) Bulk frequencies only experience the decay from the skin effect similarly as in Section \ref{Sec: Non-Hermitian skin effect 1D}. (B) The skin effect and defect localisation are working together, whereby the decay rate is correctly predicted by \eqref{eq: composite decay skin}. }
   \label{Fig: Decay in disordered systems}
\end{figure} 

Figure \ref{Fig: Decay in disordered systems} shows an example of a localised mode where the cumulative decay rate results in a mode which has different localisation length on either side of its support. These two decay rates can be accurately predicted according to the two values of \eqref{eq: composite decay skin}. 
An important observation from Figure \ref{Fig: Decay in disordered systems} is that an a priori estimate of the composite decay rate of the eigenmodes associated to the randomly introduced dimer is not possible because it relies on the position where the dimer was introduced into the chain. However, it is feasible to derive an a posteriori decay estimate using \eqref{eq: composite decay skin} once the positions of the dimers within the chain are known.

\subsubsection{Topologically protected interface modes.}

In this section, we illustrate that the classical SSH chain model exhibits a topologically protected interface mode. Topologically robust structures are of particular interest as the observed effects persist even when considering a non-pristine structure and can therefore be experimentally observed. Here, we consider systems of dimers with a geometric defect, so that at some point the repeating pattern of alternating separation distances is broken. This system is illustrated in Figure \ref{fig: geometrical defect}.

\begin{figure}[tbh]
    \centering
    \begin{adjustbox}{width=\textwidth}
        \begin{tikzpicture}
        \draw[-,thick,dotted] (-.5,-1) -- (-.5,2);
        \draw[|-|,dashed] (0,1) -- (1,1);
        \node[above] at (0.5,1) {$s_2$};
        \draw[ultra thick] (1,0) -- (2,0);
        \node[below] at (1.5,0) {$D_{2m+2}$};
        \draw[-,dotted] (1,0) -- (1,1);
        
        \draw[|-|,dashed] (2,1) -- (2.5,1);
        \node[above] at (2.25,1) {$s_1$};
        \draw[ultra thick] (2.5,0) -- (3.5,0);
        \node[below] at (3,0) {$D_{2m+3}$};
        \draw[-,dotted] (2.5,0) -- (2.5,1);
        \draw[-,dotted] (3.5,0) -- (3.5,1);
        \draw[-,dotted] (2,0) -- (2,1);
        
        \begin{scope}[shift={(+4,0)}]
        \draw[|-|,dashed] (-0.5,1) -- (0.5,1);
        \node[above] at (0,1) {$s_2$};
        \draw[ultra thick] (0.5,0) -- (1.5,0);
        \node[below] at (1,0) {$D_{2m+4}$};
        \draw[-,dotted] (0.5,0) -- (0.5,1);
        \node at (2.5,.5) {\dots};
        \end{scope}

        \begin{scope}[shift={(+6,0)}]
        \draw[ultra thick] (1,0) -- (2,0);
        \node[below] at (1.5,0) {$D_{4m-2}$};
        
        \draw[|-|,dashed] (2,1) -- (3,1);
        \node[above] at (2.5,1) {$s_2$};
        \draw[ultra thick] (3,0) -- (4,0);
        \node[below] at (3.5,0) {$D_{4m}$};
        \draw[-,dotted] (2,0) -- (2,1);
        \draw[-,dotted] (3,0) -- (3,1);
        \draw[-,dotted] (4,0) -- (4,1);
    \begin{scope}[shift={(+2,0)}]
        \draw[|-|,dashed] (2,1) -- (2.5,1);
        \node[above] at (2.25,1) {$s_1$};
        \draw[ultra thick] (2.5,0) -- (3.5,0);
        \node[below] at (3,0) {$D_{4m+1}$};
        %\draw[-,dotted] (1.5,0) -- (1.5,1);
        \draw[-,dotted] (2.5,0) -- (2.5,1);
        % \draw[-,dotted] (4,0) -- (4,1);
        \end{scope}
        % \node at (4.5,4.5) {\dots};
        \end{scope}
        
\begin{scope}[shift={(-4,0)}]

        \draw[|-|,dashed] (0.5,1) -- (1,1);
        \node[above] at (0.75,1) {$s_1$};
        \draw[ultra thick] (1,0) -- (2,0);
        \node[below] at (1.5,0) {$D_{2m}$};
        \draw[-,dotted] (1,0) -- (1,1);
        
        \draw[|-|,dashed] (2,1) -- (3,1);
        \node[above] at (2.5,1) {$s_2$};
        \draw[ultra thick] (3,0) -- (4,0);
        \node[below,fill=white] at (3.5,0) {$D_{2m+1}$};
        \draw[-,dotted] (2,0) -- (2,1);
        \draw[-,dotted] (3,0) -- (3,1);
        \draw[-,dotted] (4,0) -- (4,1);
        \end{scope}
        
        \begin{scope}[shift={(-8,0)}]
        \node at (.5,.5) {\dots};
        \draw[ultra thick] (1.5,0) -- (2.5,0);
        \node[below] at (2,0) {$D_{2m-2}$};
        
        \draw[|-|,dashed] (2.5,1) -- (3.5,1);
        \node[above] at (3,1) {$s_2$};
        \draw[ultra thick] (3.5,0) -- (4.5,0);
        \node[below] at (4,0) {$D_{2m-1}$};
        \draw[-,dotted] (2.5,0) -- (2.5,1);
        \draw[-,dotted] (3.5,0) -- (3.5,1);
        \draw[-,dotted] (4.5,0) -- (4.5,1);
        \end{scope}

            %Left most resonators
        \begin{scope}[shift={(-14,0)}]
        \draw[ultra thick] (1,0) -- (2,0);
        \node[below] at (1.5,0) {$D_{1}$};
        
        \draw[|-|,dashed] (2,1) -- (2.5,1);
        \node[above] at (2.25,1) {$s_1$};
        \draw[ultra thick] (2.5,0) -- (3.5,0);
        \node[below] at (3,0) {$D_{2}$};
        \draw[-,dotted] (2,0) -- (2,1);
        \draw[-,dotted] (2.5,0) -- (2.5,1);
        \draw[-,dotted] (3.5,0) -- (3.5,1);
\begin{scope}[shift={(+2,0)}]
        \draw[|-|,dashed] (1.5,1) -- (2.5,1);
        \node[above] at (2,1) {$s_2$};
        \draw[ultra thick] (2.5,0) -- (3.5,0);
        \node[below] at (3,0) {$D_{3}$};
        \draw[-,dotted] (2.5,0) -- (2.5,1);
        %\draw[-,dotted] (3.5,0) -- (3.5,1);
        % \draw[-,dotted] (4,0) -- (4,1);
        \end{scope}
        % \node at (4.5,4.5) {\dots};
        \end{scope}
        
        \end{tikzpicture}
    \end{adjustbox}
    \caption{Dimer structure with a geometric defect.}
    \label{fig: geometrical defect}
\end{figure}
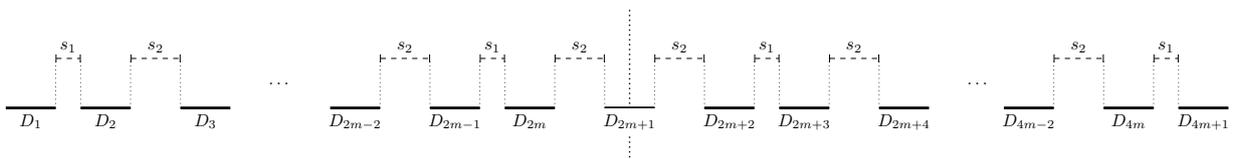

Following \cite[Section 2.2.3.]{debruijn2024complexbandstructuresubwavelength}, a necessary and sufficient condition for having a non-trivial band gap is $s_1 \neq s_2$, which will be a standing assumption throughout this section.

The interface is referred to as topologically protected because the observed localisation effects persist even when the system is perturbed. The following result summarises this idea and asserts that the localisation strength is also preserved.
\begin{corollary}
    Let the resonator spacings be perturbed by an amount $\varepsilon$. Then the exponential decay rate remains robust and is perturbed by at most an amount of order $\mathcal{O}(\varepsilon)$.
\end{corollary}

\begin{proof}
   If the resonator spacings are perturbed by an amount $\varepsilon$, the eigenfrequencies are known to shift by at most $\mathcal{O}(\varepsilon)$ (see \cite[Proposition 6.1]{Edge_Modes} for a rigorous statement). The exponential decay rate of the interface mode can be predicted as a function of the defect eigenfrequency (see \cite[Corollary 2.7.]{debruijn2024complexbandstructuresubwavelength}). 
   Prior research has demonstrated that the interface eigenfrequency of an SSH chain converges to some frequency $\omega_{\mathsf{i}}$ which can be explicitly characterised \cite[Theorem 5.5]{ammari2024exponentiallylocalisedinterfaceeigenmodes}. Taylor expanding the decay length yields
   \begin{align}\label{eq: decay length defect mode}
        \beta(\omega_{\mathsf{i}} + \varepsilon) &= \frac{1}{L}\operatorname{arcosh}\left(\frac{s_1 s_2}{2}\left( \frac{1}{s_1^2}+\frac{1}{s_2^2} - \left(\frac{ \bigl(\omega_{\mathsf{i}} + \mathcal{O}(\varepsilon)\bigr)^2}{\delta}-\frac{1}{s_1}-\frac{1}{s_2}\right)^2\right)\right)\\
        &= \beta(\omega_{\mathsf{i}}) + \mathcal{O}(\varepsilon).
    \end{align}
   This completes the proof.
\end{proof}
To put it differently, the localisation effects persist in a perturbed system and the decay length is also robust with respect to perturbations.

\section{Two-dimensional crystals}\label{sec: Two dimensional resonator}

In this section, we will consider a two-dimensional periodic structure of resonators. For band gap frequencies, we develop quasiperiodic layer-potential techniques to find evanescent solutions to the scattering problem and characterise the band gap functions of the system.

\subsection{Setting and problem formulation} 
In this section, we introduce the setting and summarise the main results from \cite{debruijn2024complexbandstructuresubwavelength} for two-dimensional crystals. A two-dimensional lattice of subwavelength resonators arranged periodically in a lattice $\Lambda$ is considered. The lattice $\Lambda$, its corresponding dual lattice
 $\Lambda^*$, and the unit cell $Y$ are as introduced as in Definition \ref{def: lattice},  Definition \ref{def: reciprocal lattice}, and Definition \ref{def: cell}, respectively.

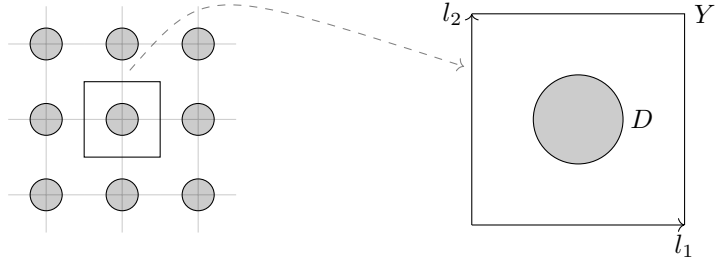
\begin{figure}[tbh]
	\centering
	\begin{tikzpicture}
	\begin{scope}[xshift=-4cm,scale=1]
	\coordinate (a) at (1,0);		
	\coordinate (b) at (0,1);	
	
	\draw (-0.5,-0.5) -- (0.5,-0.5) -- (0.5,0.5) -- (-0.5,0.5) -- cycle; 
	\draw[fill=black!20] (0,0) circle(6pt);
	\draw[opacity=0.2] (-0.5,0) -- (0,0)
	(0.5,0) -- (0,0)
	(0,-0.5) -- (0,0)
	(0,0.5) -- (0,0);
	
	\begin{scope}[shift = (a)]
	\draw[fill=black!20] (0,0) circle(6pt);
	\draw[opacity=0.2] (-0.5,0) -- (0,0)
	(0.5,0) -- (0,0)
	(0,-0.5) -- (0,0)
	(0,0.5) -- (0,0);
	\end{scope}
	\begin{scope}[shift = (b)]
	\draw[fill=black!20] (0,0) circle(6pt);
	\draw[opacity=0.2] (-0.5,0) -- (0,0)
	(0.5,0) -- (0,0)
	(0,-0.5) -- (0,0)
	(0,0.5) -- (0,0);
	\end{scope}
	\begin{scope}[shift = ($-1*(a)$)]
	\draw[fill=black!20] (0,0) circle(6pt);
	\draw[opacity=0.2] (-0.5,0) -- (0,0)
	(0.5,0) -- (0,0)
	(0,-0.5) -- (0,0)
	(0,0.5) -- (0,0);
	\end{scope}
	\begin{scope}[shift = ($-1*(b)$)]
	\draw[fill=black!20] (0,0) circle(6pt);
	\draw[opacity=0.2] (-0.5,0) -- (0,0)
	(0.5,0) -- (0,0)
	(0,-0.5) -- (0,0)
	(0,0.5) -- (0,0);
	\end{scope}
	\begin{scope}[shift = ($(a)+(b)$)]
	\draw[fill=black!20] (0,0) circle(6pt);
	\draw[opacity=0.2] (-0.5,0) -- (0,0)
	(0.5,0) -- (0,0)
	(0,-0.5) -- (0,0)
	(0,0.5) -- (0,0);
	\end{scope}
	\begin{scope}[shift = ($-1*(a)-(b)$)]
	\draw[fill=black!20] (0,0) circle(6pt);
	\draw[opacity=0.2] (-0.5,0) -- (0,0)
	(0.5,0) -- (0,0)
	(0,-0.5) -- (0,0)
	(0,0.5) -- (0,0);
	\end{scope}
	\begin{scope}[shift = ($(a)-(b)$)]
	\draw[fill=black!20] (0,0) circle(6pt);
	\draw[opacity=0.2] (-0.5,0) -- (0,0)
	(0.5,0) -- (0,0)
	(0,-0.5) -- (0,0)
	(0,0.5) -- (0,0);
	\end{scope}
	\begin{scope}[shift = ($-1*(a)+(b)$)]
	\draw[fill=black!20] (0,0) circle(6pt);
	\draw[opacity=0.2] (-0.5,0) -- (0,0)
	(0.5,0) -- (0,0)
	(0,-0.5) -- (0,0)
	(0,0.5) -- (0,0);
	\end{scope}
	\end{scope}
	
	\draw[dashed,opacity=0.5,->] (-3.9,0.65) .. controls(-2.9,1.8) .. (0.5,0.7);
	\begin{scope}[xshift=2cm,scale=2.8]	
	\coordinate (a) at (1,{1/sqrt(3)});		
	\coordinate (b) at (1,{-1/sqrt(3)});	
	\coordinate (Y) at (1.8,0.45);
	\coordinate (c) at (2,0);
	\coordinate (x1) at ({2/3},0);
	\coordinate (x0) at (1,0);
	\coordinate (x2) at ({4/3},0);
 
	\pgfmathsetmacro{\rb}{0.25pt}
	\pgfmathsetmacro{\rs}{0.2pt}\
 
	\draw[->] (-0.5,-0.5) -- (-0.5,0.5) node[left]{$l_2$}; 
	\draw[->] (-0.5,-0.5) -- (0.5,-0.5) node[below]{$l_1$}; 
	\draw (0.5,-0.5) -- (0.5,0.5) -- (-0.5,0.5);
	\draw[fill=black!20] (0,0) circle(6pt);
	\draw (0.3,0) node{$D$};
	
	\draw (0.5,0.5) node[right]{$Y$};
	\end{scope}
	\end{tikzpicture}
	\caption{Illustration of the crystal in the case of a square lattice with a single resonator inside the unit cell ($N=1$).} \label{fig:square_lattice}
\end{figure}

In two dimensions, the resonators $D_1,D_2,\dots,D_N\subset Y$ are disjoint, connected sets with boundaries in $C^{1,s}$ for some $0<s<1$. As before, we let $N$  denote the number of resonators inside $Y$, and let $D = \bigcup_{i=1}^N D_i$. An example of the setting, with a single resonator in a square lattice, is given in \Cref{fig:square_lattice}.

For $\alpha \in Y^*$ in the real Brillouin zone and $\beta\in \R^2$, we consider the two-dimensional analogue of \eqref{eq:u_pde}:

\begin{equation}\label{eq:Helmholtz2D}
    \begin{cases}
		\ds\Delta u + k^2u = 0, & \text{in } Y \setminus \overline{D}, \\
		\ds\Delta u + k_i^2u = 0,  & \text{in } D_i, \\
		\ds u\rvert_+ -u\rvert_- =0, &\text{on }\partial D,\\
		\ds\frac{\partial u}{\partial \nu}\bigg|_{-} - \delta \frac{\partial u}{\partial \nu} \bigg|_{+} = 0, & \text{on } \partial D, \\
        \ds u(x + \ell) = e^{\i (\alpha+\i \beta) \cdot \ell}u(x), &\text{for all }\ell \in \Lambda,
    \end{cases}
\end{equation}
where $\nu$ stands for the outward-facing normal derivative. Here, $|_\pm$ denotes the limit from outside or inside $D$, respectively.

\subsection{The band gap  Green's function} For a real quasimomentum $\alpha$, the $\alpha$-quasiperiodic Green's function $G^{\alpha,k}$ satisfies
\begin{equation}
    \Delta G^{\alpha,k}(x) + k^2G^{\alpha,k}(x) = \sum_{m \in \Lambda}\delta(x-m)e^{\i  \alpha \cdot m},
\end{equation}
and is given by
\begin{equation}\label{eq: real Green's function}
    G^{\alpha,k}(x) = \frac{1}{\lvert Y \rvert} \sum_{q \in \Lambda^*}\frac{e^{\i (\alpha + q)\cdot x}}{k^2-\lvert \alpha + q\rvert^2}.
\end{equation}
The solution $u(x)$ satisfies the complex Floquet condition $u(x +\ell) = e^{\i  \alpha\cdot \ell}e^{-\beta \cdot \ell}u(x)$.
Since $\beta \in \R^2$ is a vector, it defines not only the rate but also the direction of decay. By a change of function, we may find a solution that is no longer decaying:
\begin{equation}\label{eq: exponential decay 2D}
    v(x) := e^{\beta\cdot x}u(x).
\end{equation}
Substituting into \eqref{eq:u_pde}, we find that $v$ satisfies
\begin{equation}\label{eq:v_pde}
    \begin{cases}
		\ds\Delta v - 2\beta\cdot \nabla v + (k^2+|\beta|^2)v = 0, & \text{in } Y \setminus \overline{D}, \\
		\ds\Delta v - 2\beta\cdot \nabla v + (k_i^2+|\beta|^2)v = 0, & \text{in } D_i, \\
		\ds v\rvert_+ -v\rvert_- =0, &\text{on }\partial D,\\
		\ds\frac{\partial v}{\partial \nu}\bigg|_{-} - (\beta\cdot \nu) v - \delta \left(\frac{\partial v}{\partial \nu} \bigg|_{+} - (\beta\cdot \nu) v\right) = 0, & \text{on } \partial D, \\
        \ds v(x + \ell) = e^{\i \alpha \cdot \ell}v(x), &\text{for all }\ell \in \Lambda.
    \end{cases}
\end{equation}
We emphasise that $v$ now satisfies the real Floquet condition given by the last line of \eqref{eq:v_pde}. The band gap  Green's function $\Tilde{G}$ can be defined similarly:
\begin{equation}
    \Tilde{G}^{\alpha,\beta,k}(x) = e^{\beta \cdot x}G^{\alpha,k}(x).
\end{equation}
As a consequence, the band gap  Green's function satisfies the following equation
\begin{equation}\label{eq: formula band gap greens function}
    \Delta \Tilde{G}^{\alpha,\beta,k}(x) -2 \beta\cdot \nabla \Tilde{G}^{\alpha,\beta,k}(x) + \bigl(k^2 + \lvert \beta \rvert^2\bigr) \Tilde{G}^{\alpha,\beta,k}(x) = \sum_{m \in \Lambda} e^{\i  \alpha \cdot m} \delta(x - m),
\end{equation}
and is  given by 
\begin{equation}\label{eq: Green's function}
    \Tilde{G}^{\alpha,\beta,k}(x) = \frac{1}{\lvert Y \rvert} \sum_{q \in \Lambda^*} \frac{e^{\i(\alpha + q) \cdot x}}{k^2 + \lvert \beta \rvert^2 -2\i \beta \cdot (\alpha + q) - \lvert \alpha + q\rvert^2}.
\end{equation}
The series defining $\Tilde{G}^{\alpha, \beta,k}(x)$ breaks down if
\begin{equation}
    k^2 + \lvert \beta \rvert^2 -2\i \beta \cdot (\alpha + q) - \lvert \alpha + q\rvert^2= 0,
\end{equation}
for some $q \in \Lambda^*$. At these singularities, the homogeneous part of \eqref{eq: formula band gap greens function} has non-trivial solutions, and the Green's function cannot be uniquely defined.

The solution of the scattering problem described by \eqref{eq:v_pde} can now be formulated using layer potential techniques.
\begin{definition}[Layer potentials]
    Let $D \subset Y$ be a domain in $\R^2$ with a boundary $\partial D \in C^{1,s}$, with some $0 < s < 1$. Let $\nu$ denote the unit outward normal to $\partial D$. For $k > 0$, let $\tilde{\mathcal{S}}^{\alpha,\beta,k}_D$ be the single layer potential and $(\Tilde{\mathcal{K}}^{\alpha,\beta,k}_D)^*$ be the Neumann-Poincaré operator associated to the band gap Green's function $\Tilde{G}^{\alpha,\beta,k}$ defined in \eqref{eq: Green's function}. For any given density $\phi \in L^2(\partial D),$
    \begin{align}
        \Tilde{\mathcal{S}}_D^{\alpha,\beta,k}[\phi](x) &= \int_{\partial D} \Tilde{G}^{\alpha,\beta,k}(x-y)\phi(y) \d \sigma(y),\qquad  x \in \R^2, \label{def: single layer potential}\\
       (\Tilde{\mathcal{K}}^{\alpha,\beta,k}_D)^*[\phi](x) &= \int_{\partial D} \frac{\partial \Tilde{G}^{\alpha,\beta,k}(x-y)}{\partial \nu(x)}\phi(y) \d\sigma(y),\quad  x \in \partial D.\hspace{19mm}
    \end{align}
\end{definition}
The solution to \eqref{eq:v_pde} takes the following representation,
\begin{equation}\label{eq:intrep}
    v(x) = \begin{cases}
        \Tilde{\mathcal{S}}_D^{\alpha,\beta,k}[\phi](x), \hspace{2mm}&\text{in } Y \setminus \overline{D},\\
         \Tilde{\mathcal{S}}_{D_i}^{\alpha,\beta, k_i}[\psi](x), \hspace{2mm}&\text{in } D_i,
    \end{cases}
\end{equation}
for some unknown densities $\phi\in L^2(\partial D)$ and $\psi_i \in L^2(\partial D_i)$. The theorem outlined below is the fundamental theorem of capacitance, characterising at leading order in $\delta$ the eigenfrequencies of the scattering problem \eqref{eq:v_pde}.

\begin{theorem}\label{thm: subwavelength resonant frequencies}\cite[Theorem 3.8]{debruijn2024complexbandstructuresubwavelength}
    Consider a system of $N$ subwavelength resonators in the unit cell $Y$, assume
    \begin{equation}
        \lvert \alpha + q \rvert^2 \neq \lvert \beta \rvert^2 \text{ or } (\alpha + q)\cdot \beta \neq 0,
    \end{equation}
    for all $q \in \Lambda^*$ and assume that $\Tilde{\mathcal{S}}^{\alpha,\beta,0}_D$ is invertible. As $\delta \to 0$, there are, up to multiplicity, $N$ subwavelength resonant frequencies $\omega^{\alpha, \beta}_n$, for $n = 1, \dots, N$, which satisfy the asymptotic formula
    \begin{equation}
        \omega^{\alpha, \beta} = \sqrt{\delta\lambda^{\alpha, \beta}_n} + \mathcal{O}(\delta), \quad n = 1,\dots, N,
    \end{equation}
    where $\{\lambda^{\alpha, \beta}_n\}$ are the $N$ eigenvalues of the generalised capacitance matrix $\mathcal{C}^{\alpha, \beta} \in \C^{N \times N}$, given by
    \begin{equation}\label{eq: def Capacitance matrix}
        \mathcal{C}_{ij}^{\alpha,\beta} = - \frac{v_i^2}{\lvert D_i \rvert}\int_{D_i} e^{-\i  \beta \cdot x}\psi_j \d \sigma, \quad \psi_i = (\tilde{\mathcal{S}}^{\alpha,\beta,0}_D)^{-1}[e^{\beta \cdot x} \chi_{\partial D_i}].
    \end{equation}
\end{theorem}
For small enough $\beta$ the single layer potential is known to be invertible (see \cite[Lemma 3.6]{debruijn2024complexbandstructuresubwavelength}), however, invertibility fails for general $\beta$.
At points where the invertibility of the single layer potential fails, the gap functions experience singularities, which push the complex bands deep inside the band gap.

\subsection{Non-uniqueness of solutions.} \label{sec: Green Function Singular}

In this section, we investigate the parameter-valued points for which the single layer potential fails to be invertible. This significantly impacts band functions as they experience singularities (see Figure \ref{Fig: Singularities in the Capacitance matrix}). 
From a physical perspective, these singularities represent a decay cap, which imposes a limit on the decay rate of the system. This decay cap is of crucial importance and will be explored in concrete situations, along with numerical simulations, in Section \ref{Sec: Bandstructure for defect modes}. We move forward by analysing the origin of these singularities.

\begin{figure}[tbh]
    \centering
    \subfloat[][Top view.]{{\includegraphics[height=0.35\linewidth]{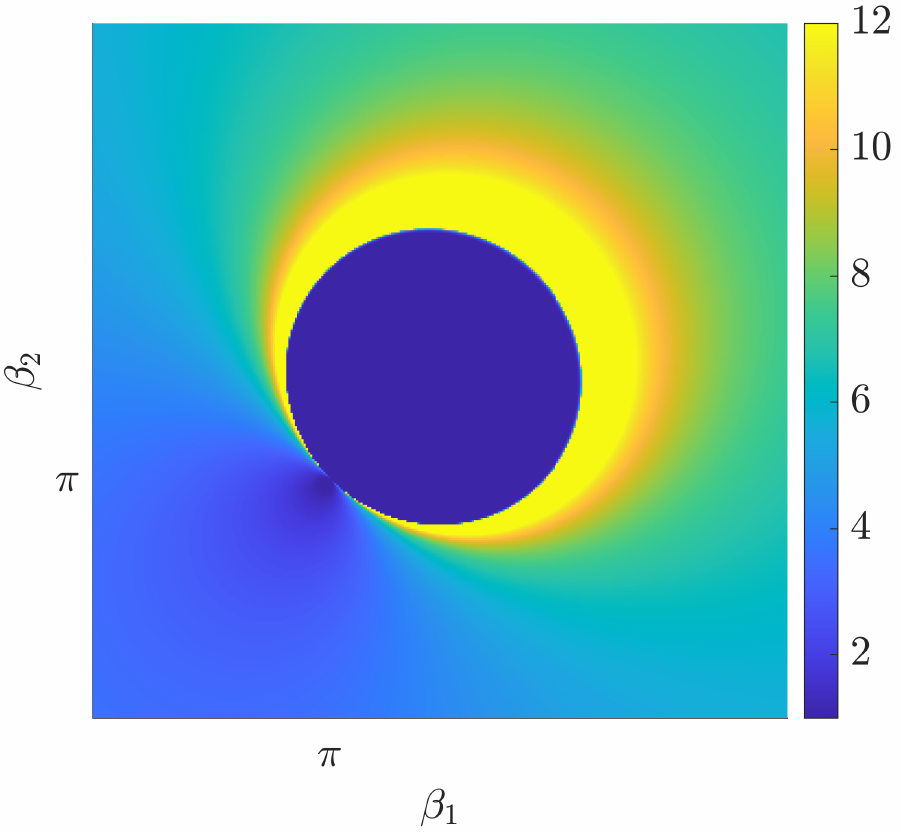} }}\qquad
    \subfloat[][Side view.]{{\includegraphics[height=0.35\linewidth]{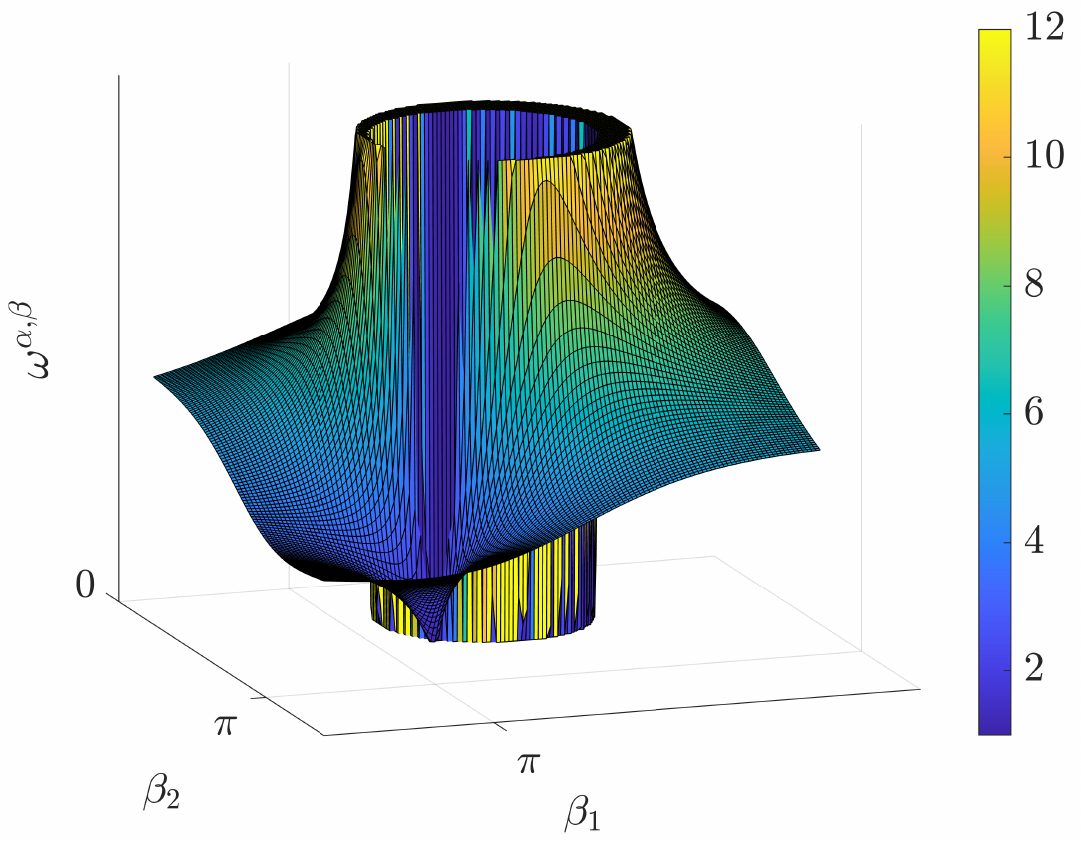} }} 
    \caption{Surface plot for the band function $\omega^{\alpha, \beta}$ for fixed $\alpha$. The singularity of the band function is due to the non-invertibility of the single layer potential at certain complex quasimomenta. Computation performed for $\alpha = [\pi, \pi]$ and $ R = 0.005$.}
   \label{Fig: Singularities in the Capacitance matrix}
\end{figure}

\subsubsection{Free-space problem.}
Let us first examine the free-space problem, where there are no interactions with resonators.
Physically, this means that the external source is no longer needed to excite the system; the system is now resonating at one of its natural frequencies. This phenomenon is commonly referred to as an empty resonance \cite[Section 5.4.2.1]{ammari.fitzpatrick.ea2018Mathematical}. The free-space problem reads,
\begin{equation}\label{eq: free-space PDE}
    \begin{cases}
        \Delta v - 2\beta \cdot \nabla v  + (k^2 + \lvert \beta \rvert^2)v = 0,  \quad \text{in } Y, \\
        v(x + \ell) = e^{\i \alpha \cdot \ell}v(x), \quad \ell \in \Lambda.
    \end{cases}
\end{equation}
For the purpose of examining the subwavelength regime, we fix the frequency $k= 0$. Consequently, our focus will be on analysing \eqref{eq: free-space PDE} under this condition.
\begin{lemma}\label{lem: free-space Green}
    The free-space problem \eqref{eq: free-space PDE} has two non-trivial solutions,
    \begin{equation}
        v(x) = e^{\i\kappa x},
    \end{equation}
    at high-symmetry points  $\kappa = \alpha_i$, $i \in \{1, 2 \},$ that satisfy
    \begin{equation}\label{eq:Rayleigh}
        \lvert \alpha_i + q \rvert^2 = \lvert \beta \rvert^2 \quad \text{and}\quad     ( \alpha_i  + q) \perp \beta, \quad q \in \Lambda^*.
    \end{equation}
\end{lemma}

\begin{proof}
    Using the ansatz $v(x) = e^{\i \kappa x}$ on \eqref{eq: free-space PDE} leads to,
    \begin{equation}
        \lvert \beta \rvert^2 -2\i \beta \cdot \kappa - \lvert \kappa\rvert^2= 0,
    \end{equation}
    which is equivalent to
    \begin{equation} \label{eq: singularity points}
        \lvert \kappa \rvert^2 =  \lvert \beta \rvert^2 \quad \text{and}\quad \kappa \perp \beta.
    \end{equation}
    This precisely yields two $\kappa$ that satisfy $\kappa_1 = - \kappa_2$.
    As the solution also has to satisfy the real Floquet condition, i.e.
    \begin{align}
               v(x + \ell) &= e^{\i \kappa \cdot \ell}v(x)  \\
               &= e^{\i \alpha \cdot \ell}v(x),
    \end{align}
    it readily follows that
    \begin{equation}
        \kappa = \alpha + q, \hspace{2mm} q \in \Lambda^*.
    \end{equation}
    This completes the proof.
\end{proof}

\subsubsection{Infinitesimally small resonators.} \label{sec: Infinitely dilute resonator}
We seek to investigate parameter points where the single layer potential has a non-trivial kernel, meaning that there is a non-trivial $\phi$ such that
\begin{equation}
    v(x) =  \Tilde{\mathcal{S}}_D^{\alpha,\beta, 0}[\phi](x) = 0, \quad \forall x \in \partial D.
\end{equation}
This encourages us to examine an infinitesimally dilute resonator located at the centre, which corresponds to a single point subject to Dirichlet boundary conditions,

\begin{equation}\label{eq: Dilue resonator PDE}
    \begin{cases}
        \Delta v - 2\beta \cdot \nabla v + \lvert \beta \rvert^2 v = 0, \quad \text{in } Y, \\
        v(0) = 0,\\
        v(x + \ell) = e^{\i \alpha \cdot \ell}v(x), \quad \ell\in \Lambda.
    \end{cases}
\end{equation}

\begin{lemma}\label{lemma: Dilute resonator PDE}
    The boundary value problem \eqref{eq: Dilue resonator PDE} admits a non-trivial solution given by
    \begin{equation}\label{eq: Sol infinite dilute}
        v(x) = \sin(\kappa x)
    \end{equation}
    at the high-symmetry points given by \eqref{eq:Rayleigh}.
\end{lemma}

\begin{proof}
    By superposition of the two solutions found in Lemma \ref{lem: free-space Green}, $v(x) = e^{\i \kappa x} - e^{- \i \kappa x}$ is also a solution to the quasiperiodic free-space problem. Furthermore, since $v(0) = 0$, the proof is complete.
\end{proof}

As we shall see, when the resonators have a small but non-zero width, the solutions we found in Lemma \ref{lemma: Dilute resonator PDE} will perturb into non-trivial kernel functions of the single layer potential.
 
\subsubsection{Large resonators.}
The next result treats the case of an arbitrary resonator width.

\begin{lemma}\label{lemma: Zero on the interior}
    Suppose that the single layer potential $\Tilde{\mathcal{S}}_D^{\alpha,\beta,k}$ has a non-trivial kernel function $\phi$, and define $v(x) = \Tilde{\mathcal{S}}_D^{\alpha,\beta,k}[\phi](x)$. Then, for $x \in D$, it follows $v(x) = 0$.
\end{lemma}

\begin{proof}
    Since $\phi \in \operatorname{ker}\bigl(\mathcal{S}_D^{\alpha,\beta, \omega}\bigr)$, we have $u = 0$ on $\partial D$. The boundary value problem reads
    \begin{equation}\label{eq: boundary value roblem laplace}
        \begin{cases}
            \Delta v - 2 \beta \cdot \nabla v + \lvert \beta \rvert^2 v = 0, \\
            v|_{\partial D} = 0.
        \end{cases} 
    \end{equation}
    Let $u(x) = e^{-\beta x} v(x)$, then \eqref{eq: boundary value roblem laplace} becomes
    \begin{equation}
        \begin{cases}
            \Delta u = 0, \\
            u|_{\partial D} = e^{-\beta\cdot x} v(x)|_{\partial D} = 0.
        \end{cases} 
    \end{equation}
    Therefore $u(x) = 0, \hspace{2mm} \forall x \in D$, implying that $v(x) = 0, \hspace{2mm} \forall x \in D.$
\end{proof}

\subsection{Numerical illustrations.}\label{sec: SLP kenerel numerics}
This section numerically illustrates the parameter points where the single layer potential is not of full rank. For clarity of the presentation, we postpone the discussion of numerical methods to Section \ref{Sec: Numerical Methods}.

To detect parameter points where the single layer potential is non-invertible, we plot the smallest singular value of the matrix approximation of $\tilde{\mathcal{S}}^{\alpha,\beta,0}_D$.
We illustrate the points where the single layer potential lacks full rank in Figure \ref{Fig: Singular values of single layer potential}. At certain values of $\beta$, the smallest singular value of the single layer potential approaches $0$, indicating the presence of a non-trivial kernel function $\phi$.

\begin{figure}[htb]
    \centering
    \subfloat{{\includegraphics[height=0.35\linewidth]{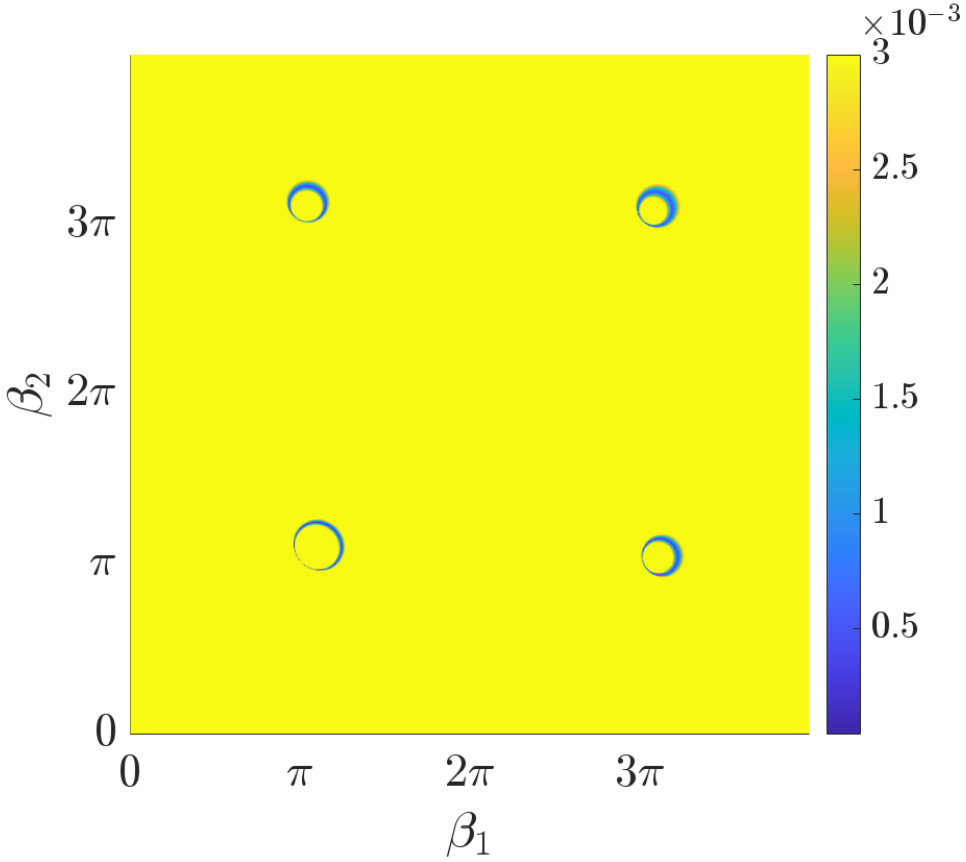} \qquad    
    \includegraphics[height=0.35\linewidth]{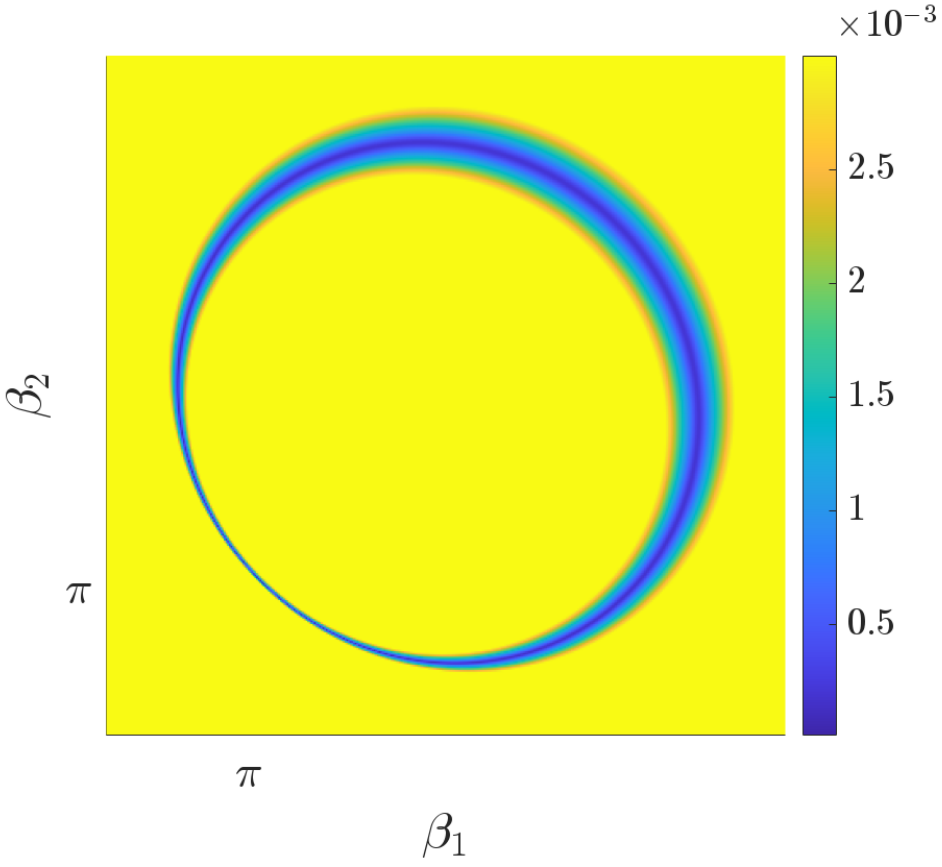} }}
    \caption{Smallest singular value of the single layer potential $\tilde{\mathcal{S}}^{\alpha, \beta, \omega}_D$. In certain regions (in blue), the smallest singular value of the single layer potential approaches zero, indicating that the single layer potential possesses a non-trivial kernel. Computation performed for a resonator of size $R = 0.03$, $\alpha = [\pi, \pi]$ and multipole basis of $N = 5$.}
   \label{Fig: Singular values of single layer potential}
\end{figure}

In Figure \ref{Fig: Field solution 0 multipole dilute} we compute the field solution $v(x) = \Tilde{\mathcal{S}}_D^{\alpha,\beta,k}[\phi](x)$, for $x\neq \partial D$, at parameter points such that the single layer potential has a non-trivial kernel $\phi$. As the resonator width tends to zero, the exterior field resembles the solution found in Lemma \ref{lemma: Dilute resonator PDE}.  We also note that, according to \cite[Lemma 3.6.]{debruijn2024complexbandstructuresubwavelength}, $\beta$ must be sufficiently large in order for a non-trivial kernel to appear.
To put it differently, the introduction of a complex quasimomentum $\Tilde{\mathcal{S}}^{\alpha,\beta, \omega}_R$, must perturb the quasiperiodic single layer potential $\mathcal{S}_D^{\alpha, 0, \omega}$ enough to induce a non-trivial kernel. 

\begin{figure}[htb]
    \centering
    \subfloat[][Field solution $\Re\bigl(v(x)\bigr)$ with resonator radius $R = 0.1$.]{{\includegraphics[width=0.45\linewidth]{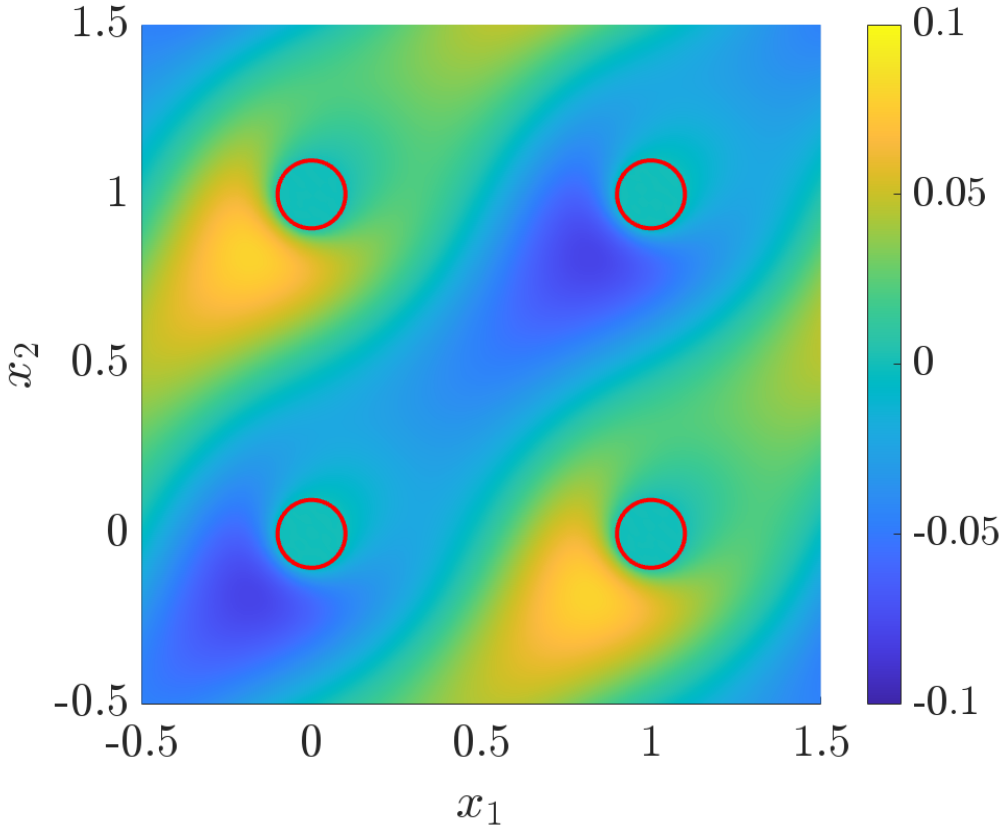} }}\qquad
    \subfloat[][Field solution $\Re\bigl(v(x)\bigr)$ with resonator radius $R = 0.3$.]{{\includegraphics[width=0.45\linewidth]{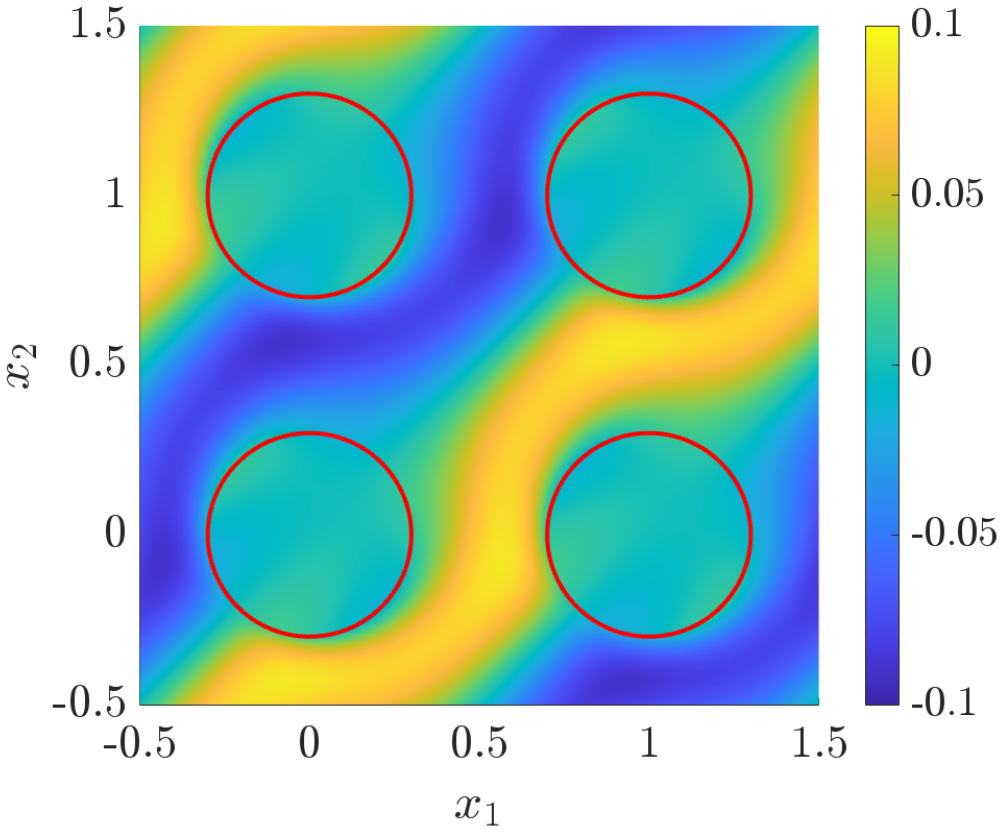} }}\\
    \subfloat[][Field solution $\Im\bigl(v(x)\bigr)$ with resonator radius $R = 0.00001$.]{{\includegraphics[width=0.45\linewidth]{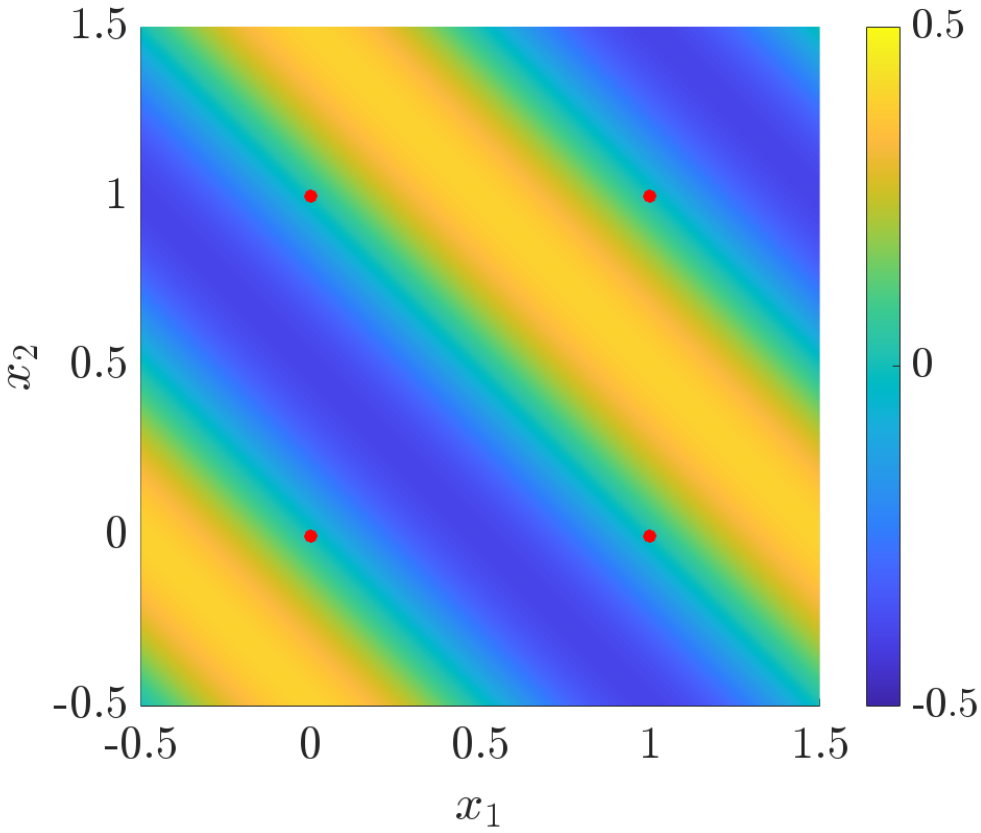} }}\qquad
    \subfloat[][Field solution $\bigl\lvert v(x)\bigr\rvert$ with resonator radius $R = 0.00001$.]{{\includegraphics[width=0.45\linewidth]{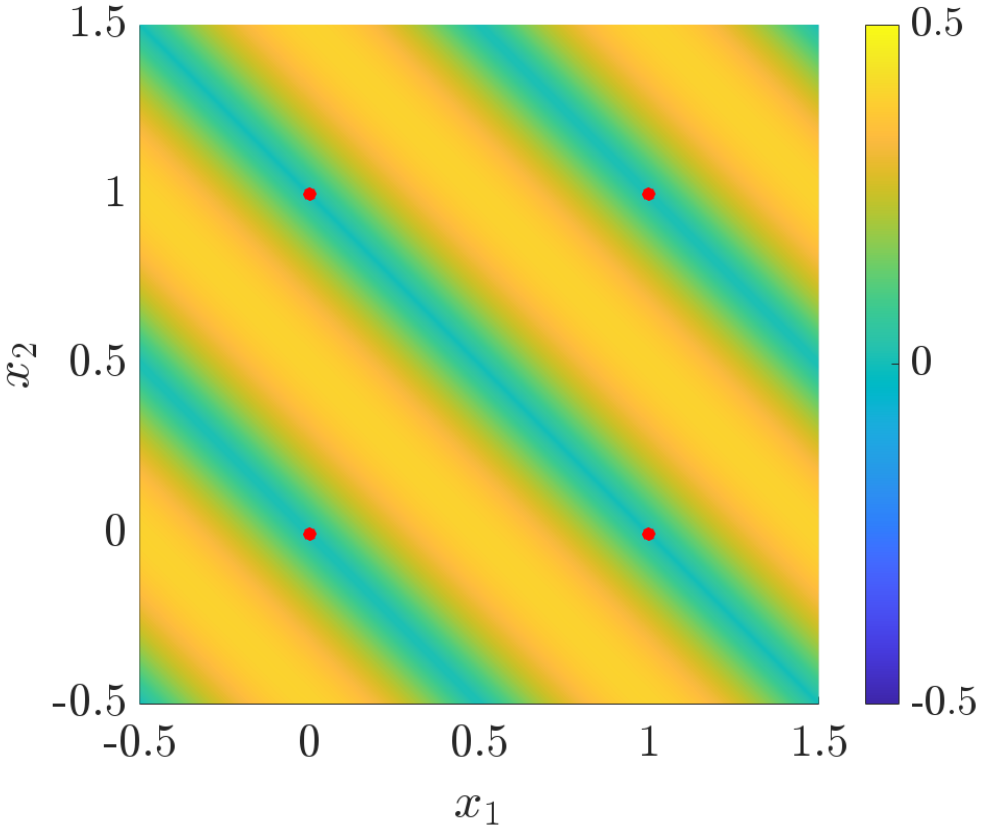} }} 
    \caption{Solution of the field $v(x)$ poised at parameter valued points $\alpha, \beta$ such that $\tilde{\mathcal{S}}^{\alpha, \beta, k}_D$ has non-trivial kernel $\phi$. (A) and (B): by Lemma \ref{lemma: Zero on the interior}, it follows that $v(x) = 0, \hspace{2mm}\forall x \in D$. (C) and (D): by Lemma \ref{lemma: Dilute resonator PDE}, the field resembles a sinusoidal wave. To show that the field solution is $0$ on the resonator, its edges are highlighted in red.}
   \label{Fig: Field solution 0 multipole dilute}
\end{figure}

\section{Numerical methods}\label{Sec: Numerical Methods}
The computational bottleneck in numerical simulations arises from evaluating the single layer potential. This challenge comes from the fact that the series that defines the Green’s function is only conditionally convergent. This section aims to introduce accelerated and precise implementations of the single layer potential. Appendix \ref{sec: Convergence rates and runtime} contains an analysis of runtime and truncation errors.

\begin{remark}\label{rem: error goal}
    According to Theorem \ref{thm: subwavelength resonant frequencies}, the resonances are determined with an error margin of $\mathcal{O}(\delta)$. As a result, our goal is to achieve a numerical error no greater than $\mathcal{O}(\delta)$ when implementing the single layer potential. Typically, we set $\delta = 10^{-3}$.
\end{remark}

\subsection{Direct computation.}
In this section, we present a non-accelerated method for the computation of the single layer potential.
The key idea is that we may obtain a series definition of the single layer potential which is absolutely convergent even though the Green's function is only conditionally convergent. 

In the case of a circular resonator $D$ we may implement the multipole expansion method (see, for example, \cite[Appendix C]{BandGapBubbly}) to represent the single layer potential in a Fourier basis. We can then discretise the problem by truncating the resulting series. The matrix representation of the complex single layer potential, in the Fourier basis $\{e^{\i n \theta}\}$, is given by
\begin{equation}\label{eq: SLP direct computation}
    (\tilde{\S}^{\alpha, \beta, \omega}_D)_{m,n}  = \frac{2\pi r}{\lvert Y \rvert} \sum_{q \in \Lambda^*} \frac{e^{\i  \psi(n-m)}\i^{m+n} (-1)^n\mathbf{J}_{m}\bigl(r\lvert(\alpha + q )\rvert\bigr)\mathbf{J}_{n}\bigl(r\lvert(\alpha + q )\rvert\bigr)}{(k^2 + \lvert \beta \rvert^2) - 2 \i \beta \cdot (\alpha + q)  -\lvert \alpha + q \rvert^2}.
\end{equation}
Due to the asymptotic behaviour of the Bessel functions as $|\alpha + q|$ becomes large,
\begin{equation}\label{eq:besselasymp}
    \mathbf{J}_n\bigl(r|\alpha + q|\bigr) \approx \frac{1}{\sqrt{r\lvert \alpha + q \rvert}},
\end{equation}
the sum defining the single layer potential entries is absolutely convergent. However, in numerical computations, we refrain from using this representation of the single layer potential as the truncation error is $\mathcal{O}(N^{-1})$, where $N$ is the lattice size.

\subsection{Accelerated lattice sums.}\label{sec: original accelerated lattice sum}
A significant limitation of directly computing the single layer potential \eqref{eq: SLP direct computation} is the poor convergence of its associated lattice sum. The work in \cite[Section 3.3]{debruijn2024complexbandstructuresubwavelength} presents an enhanced formulation of this lattice sum,
\begin{align}
    \Tilde{G}^{\alpha, \beta,k}_{R}(x) &:= \Tilde{G}^{\alpha, \beta,k}(x) -\Tilde{G}^{\alpha, 0,k}(x)\\
    &= \frac{1}{\lvert Y \rvert}\sum_{q \in \Lambda^*}e^{\i (\alpha + q)\cdot x}\left(\frac{1}{(k^2 + \lvert \beta \rvert^2) -  2 \i \beta \cdot (\alpha + q) - \lvert \alpha + q\rvert^2}- \frac{1}{k^2 - \lvert \alpha + q \rvert^2 }\right)\\
    &=\frac{1}{\lvert Y \rvert}\sum_{q \in \Lambda^*}e^{\i (\alpha + q)\cdot x} \frac{ 2 \i \beta \cdot (\alpha + q)-\lvert \beta \rvert^2}{\bigl((k^2 + \lvert \beta \rvert^2)- 2 \i \beta \cdot (\alpha + q) -  \lvert \alpha + q \rvert^2\bigr)\left(k^2 - \lvert \alpha + q\rvert^2\right)},\label{eq: Grenns_lattice sum}
\end{align}
where the representation for $\Tilde{G}^{\alpha, 0,k}(x)$ is algebraically convergent with a high order of decay. The truncation error associated with $\Tilde{G}^{\alpha, \beta,k}_{R}(x)$ in relation to the reciprocal lattice $\Lambda^*$ is characterised by the order of magnitude $\mathcal{O}(N^{-3})$.
One potential issue with this approach is that it creates an artificial singularity when $\lvert \alpha \rvert \approx k$. This results in catastrophic cancellation and precision loss in numerical calculations. This occurs because both $\Tilde{G}^{\alpha, \beta,k}_{R}(x)$ and $\Tilde{G}^{\alpha, 0,k}(x)$ become singular at $\alpha \approx k = 0$. Although these singularities are expected to cancel each other analytically, numerically this approach fails because it involves subtracting two very large numbers.

\subsection{Kummer's transformation.}
Kummer's transformation (see \cite[Section 2.13.1]{ammari.fitzpatrick.ea2018Mathematical} for an overview) allows us to get a better understanding of the complex band structure for $\alpha \approx 0$. The Green's function can be transformed into an absolutely convergent series without introducing a new singularity. Applying Kummer's transformation to the complex Green's function \eqref{eq: Green's function} yields,
\begin{align}
    \tilde{G}_K^{\alpha, \beta, k}(x) &= \frac{1}{\lvert Y \rvert} \sum_{q \in \Lambda^*} \frac{e^{\i(\alpha + q) \cdot x}}{k^2 + \lvert \beta \rvert^2 -2\i \beta \cdot (\alpha + q) - \lvert \alpha + q\rvert^2}\\
    &= \frac{e^{\i \alpha \cdot x}}{k^2 + \lvert \beta \rvert^2 - 2\i \beta \cdot \alpha- \lvert \alpha \rvert^2} - \sum_{q \in \Lambda^*\setminus\{0\}} \frac{e^{\i(\alpha + q)\cdot x}}{ \lvert q \rvert^2} \nonumber \\
    & \qquad\qquad+  \sum_{q \in \Lambda^*\setminus\{0\}} e^{\i(\alpha + q) \cdot x}\left( \frac{1}{k^2 + \lvert \beta \rvert^2 - 2\i \beta \cdot (\alpha + q)- \lvert \alpha + q \rvert^2} + \frac{1}{\lvert q \rvert^2}\right). \label{eq: Green function Kummer}
\end{align}
The second series converges absolutely, exhibiting a truncation error on the order of $\mathcal{O}(N^{-3})$. The first lattice sum admits the representation,
\begin{align}
    \sum_{q \in \Lambda^* \setminus \{0\}} \frac{e^{\i q \cdot x}}{\lvert n \rvert^2} = A_1(x) + A_2(x),
\end{align}
where
\begin{equation}
    \begin{aligned}
        A_1(x) = & \frac{1}{24} - \frac{\ln(2)}{4\pi} - \frac{1}{4}(x_2 - x_1)  + \frac{1}{4}(2x_2^2 - x_1^2) - \frac{1}{8\pi} \ln \left( \sinh^2(\pi x_2) + \sin^2(\pi x_1) \right) \\
        & + \frac{1}{4\pi} \sum_{n_1 = 1}^{+ \infty} \frac{\cos(2 \pi n_1 x_1)}{n_1} \frac{e^{2\pi n_1 x_2} + e^{-2\pi n_1 x_2}}{e^{2\pi n_1} - 1}
    \end{aligned}
\end{equation}
as well as
\begin{equation}
    \begin{aligned}
        A_2(x) = & \frac{1}{24}- \frac{\ln(2)}{4\pi}- \frac{1}{4}(x_1-x_2) + \frac{1}{4}(2x_1^2 -x_2^2)- \frac{1}{8 \pi}\ln \left( \sinh^2(\pi x_1) + \sin^2(\pi x_2)\right)\\ &+ \frac{1}{4\pi} \sum_{n_1 = 1}^{+\infty}\frac{\cos(2\pi n_1 x_2)}{n_1}\frac{e^{2\pi n_1x_1}+e^{-2\pi n_1 x_1}}{e^{2\pi n_1}-1},
    \end{aligned}
\end{equation}
where the series is exponentially convergent for $\lvert x \rvert <1$, and the sums can thus be truncated. Typically, we use $N = 2$ which yields a truncation error $\mathcal{O}(10^{-7})$.

Subsequently, we examine the asymptotic behaviour when $\alpha$ is small, since the situation where $\alpha$ is away from $0$ can be handled by employing the representation discussed in Section \ref{sec: original accelerated lattice sum}. The Green's function is represented using Kummer's transformation \eqref{eq: Green function Kummer} becomes,
\begin{align}
    \tilde{G}_K^{0, \beta, k}(x) &= \frac{1}{k^2 + \lvert \beta \rvert^2} - \sum_{q \in \Lambda^*\setminus\{0\}} \frac{e^{\i q\cdot x}}{ \lvert q \rvert^2} \nonumber \\
    & \qquad+  \sum_{q \in \Lambda^*\setminus\{0\}} e^{\i q \cdot x}\left( \frac{1}{k^2 + \lvert \beta \rvert^2 - 2\i \beta \cdot q- \lvert q \rvert^2} + \frac{1}{\lvert q \rvert^2}\right) \label{eq: def Green Kummer}.
\end{align}
The single layer potential can be expressed using a multipole expansion of order $M$ as
\begin{align}
    (\tilde{\S}^{0, \beta, k}_D)_{m,n} &= \frac{1}{2\pi}\bigl\langle e^{\i m \theta}, \tilde{\S}_D^{0, \beta, k}[e^{\i n \varphi}]\bigr\rangle \hspace{2mm}\in \C^{(2M + 1) \times (2M + 1)}\\
    &=  \frac{r}{2\pi}\int_0^{2\pi} e^{-\i m\theta} \int_0^{2\pi} \Tilde{G}^{0, \beta, k}_{K}\bigl(r(e^{\i  \theta}- e^{\i  \varphi})\bigr) e^{\i n\varphi}  \d\varphi \d\theta \label{eq: SLP Kummer formula}
\end{align}
and can be evaluated numerically using the approach outlined in Appendix \ref{appendix: SLP Kummer's method}.

\subsection{Outside field solution.}
In this section, we introduce a technique for calculating the external field. By \eqref{eq:intrep} the field solution to the scattering problem \eqref{eq:v_pde} is given by,
    \begin{equation}\label{eq: outside solution}
        v(x) = \Tilde{\mathcal{S}}_D^{\alpha,\beta,k}[\phi](x), \quad \forall x \in Y,
    \end{equation}
    for some $\phi$. We express $\phi$ in a multipole basis
\begin{equation}\label{eq: multipole expansion}
    \phi = \sum_{n \in \Z} a_n e^{\i  n \theta}.
\end{equation}
A direct computation yields
\begin{align}
    v(x) &= \int_0^{2\pi} \Tilde{G}^{\alpha, \beta, k} ( x - re^{\i \varphi}) \sum_{n \in \mathbb{Z}} a_n e^{\i n \varphi} r \, \d\varphi \\
    &= \sum_{n \in \mathbb{Z}} \frac{a_n 2\pi r (-\i)^n}{|Y|} \sum_{q \in \Lambda^*} \frac{e^{\i (\alpha + q)\cdot x}\mathbf{J}_n\bigl(r\lvert( \alpha + q)\rvert\bigr) e^{\i n \operatorname{arg}(\alpha + q)}}{k^2 + |\beta|^2 - 2\i \beta \cdot (\alpha + q) - |\alpha + q|^2},\label{eq:outside}
\end{align}
where we used the integral representation 
\begin{equation}
    \int_0^{2\pi}e^{-\i (\alpha + q)\cdot r e^{\i \varphi}}e^{\i n \varphi} \d \varphi = 2 \pi (-\i)^n \mathbf{J}_n\bigl(r\lvert( \alpha + q)\rvert\bigr) e^{\i n \operatorname{arg}(\alpha + q)}.
\end{equation}
Due to \eqref{eq:besselasymp}, the sum in \eqref{eq:outside} converges absolutely and the truncation error is $\mathcal{O}(N^{-1.5})$. However, due to cancellation among the summands, the actual error is expected to be smaller.
For numerical calculations, we typically use a lattice $\Lambda = (-N, N)^2$, where $N = 100$.
Care should be taken as $r \mapsto 0$ in the dilute limit, as here the Bessel function's contribution is insufficient to enhance the series' convergence. Thus, it is necessary to have larger truncation limits for the lattice sum.

\section{Complex band structure for defected materials}\label{Sec: Bandstructure for defect modes}

    In this section, we present a full study of two-dimensional defect modes. We demonstrate that the exponential decay rate of the defected mode can be accurately predicted by the complex band structure. Moreover, we extend the Floquet transform to complex quasimomenta. In this way, we illustrate that the defect modes may undergo a phase transition in which the real part of the quasimomentum shifts as the frequency varies within the band gap.
    
    \subsection{Defected materials.}\label{sec:defect}
    We begin by revisiting the computational framework presented in \cite{ouranderson2022, ammari2024functional}, allowing us to find subwavelength resonances within a defected two-dimensional resonator lattice.
    
    \subsubsection{Generalised capacitance matrix.} 
    The subwavelength resonances may be computed using the quasiperiodic capacitance matrix. By considering a truncated version of the inverse Floquet transform of the quasiperiodic capacitance matrix, the subwavelength resonances may be numerically computed. 

	\begin{definition} \label{defn:GQCM}
		For a system of $N\in\N$ resonators $D_1,\dots,D_N$ in $Y$ we define the \emph{generalised quasiperiodic capacitance matrix}, denoted by $\widehat{\mathcal{C}}^\alpha=(\widehat{\mathcal{C}}^\alpha_{ij})\in\mathbb{C}^{N\times N}$, to be the square matrix given by
		\begin{equation*}
			\widehat{\mathcal{C}}^\alpha_{ij}=\frac{\delta v_i^2}{|D_i|} \widehat{C}^\alpha_{ij}, \quad i,j=1,\dots,N,
		\end{equation*}
        where the quasiperiodic capacitance matrix $\widehat{C}^\alpha_{ij}$ is defined as \eqref{eq: def Capacitance matrix} with $\beta = 0$.
	\end{definition}
    The following theorem lets us recover the subwavelength resonances from the quasiperiodic capacitance matrix.
	\begin{theorem}\cite[Theorem 3.11.]{ammari2024functional} \label{thm:res_periodic}
		Let $d\in \{2,3\}$ and $0<d_l\leq d$. Consider a system of $N$ subwavelength resonators in the periodic unit cell $Y$, and assume $|\alpha| > c > 0$ for some constant $c$ independent of $\omega$ and $\delta$. As $\delta\to0$, there are $N$ subwavelength resonant frequencies, which satisfy the asymptotic formula
		\begin{equation*}
			\omega_n^\alpha = \sqrt{\lambda_n^\alpha}+O(\delta^{3/2}), \quad n=1,\dots,N.
		\end{equation*}
		Here, $\{\lambda_n^\alpha: n=1,\dots,N\}$ are the eigenvalues of the generalised quasiperiodic capacitance matrix $\widehat{\mathcal{C}}^\alpha\in\mathbb{C}^{N\times N}$, which satisfy $\lambda_n^\alpha=O(\delta)$ as $\delta\to 0$.
	\end{theorem} 
	The spectrum of the unperturbed periodic array consists of a countable collection of spectral bands $\omega^\alpha$, which depend continuously on $\alpha\in Y^*$. The bands $\omega^\alpha$ consist of propagating Bloch modes. We will be particularly interested in studying structures that experience exponentially localised Bloch modes. This is achieved by introducing a defect into the structure, causing one of the resonant frequencies to fall within the band gap.

 We introduce the Floquet transform 
 $\F: \ \bigr(\ell^2(\Lambda)\bigr)^N \to \left(L^2(Y)\right)^N$ and its inverse $\I: \ \left(L^2(Y)\right)^N \to \bigr(\ell^2(\Lambda)\bigr)^N $, which are given by
	$$\F[\phi](\alpha) := \sum_{\ell\in \Lambda} \phi(\ell)e^{\iu \alpha \cdot \ell}, \qquad  \I[\psi](\ell) := \frac{1}{|Y^*|}\int_{Y^*} \psi(\alpha) e^{-\iu \alpha \cdot \ell} \d \alpha.$$
	We can now define the real-space generalised capacitance coefficients, using the inverse Floquet transform, as
	$$\mathcal{C}^\ell = \I[\widehat{\mathcal{C}}^\alpha](\ell),$$	
	indexed by the real-space variable $\ell \in \Lambda$. If $\uf \in \ell^2(\Lambda^N,\Cb^N)$, we define the operator $\Cf$ on $\ell^2(\Lambda^N,\Cb^N)$ by
$$\Cf \uf (m) = \sum_{n\in \Lambda} \mathcal{C}^{m-n} \uf(n).$$
For a two-dimensional lattice, the capacitance matrix amounts to doubly infinite multilevel Toeplitz operator corresponding to the symbol $\widehat{\mathcal{C}}^\alpha$. The eigenvalue problem reduces to
 \begin{equation}
      \Cf \uf = \omega^2 \uf,
 \end{equation}
 and characterises by Theorem \ref{thm:res_periodic} the subwavelength resonances $\omega$ and their corresponding resonant modes $\uf$.

\subsubsection{Band gap resonant frequency.} \label{sec: band gap resonant frequency} 
When exciting a pristine structure at a band gap frequency, as sketched in Figure \ref{fig: lattice with point defect}, the resulting solution will decay exponentially away from the excitation point. Correspondingly, we define the discrete Green's function  as $(\Cf - \omega^2 I) g_i = v_i$, where $v_i = \delta_{ij}$, which means that the $i$-th entry is $1$ while all other entries are $0$. The matrix $\Cf - \omega^2 I$ is invertible because the gap frequency does not lie within the spectrum of the capacitance matrix.
The capacitance matrix yields a modal decomposition of the Green's function,
\begin{equation}\label{eq: Green resolvent kernel}
    G := (\Cf- \omega^2 I)^{-1},
\end{equation}
where
\begin{equation}
    G = \begin{pmatrix}
        | & & | \\
        g_1 &\cdots& g_n\\
        | & & |
    \end{pmatrix}.
\end{equation}
The Green's function \eqref{eq: Green resolvent kernel} can therefore be viewed as the resolvent kernel of the capacitance matrix $\Cf$.
The eigenmodes' decay rate can be determined by assessing the decay rate of the Green's function.

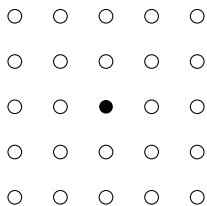
\begin{figure}[ht]
    \centering
    \begin{tikzpicture}[scale=0.6]
        \def\discradius{0.15} % Radius of the disks
        \def\gridsize{5}      % Size of the grid (5x5)

        % Draw lattice
        \foreach \x in {1,...,\gridsize} {
            \foreach \y in {1,...,\gridsize} {
        
                \ifnum\x=3 \ifnum\y=3
                    \fill[black] (\x,\y) circle (\discradius); 
                \else
                    \draw[black] (\x,\y) circle (\discradius); 
                \fi\fi
                \ifnum\x=3 \else \draw[black] (\x,\y) circle (\discradius); \fi
            }
        }
    \end{tikzpicture}
    \caption{Truncated infinite lattice which is excited by  a point source in the centre. The excitation produces an exponentially localised solution within the band gap.}
    \label{fig: lattice with point defect}
\end{figure}

\subsubsection{Numerical computations.} We demonstrate that the decay length of the defect eigenmode at the gap frequency $\omega^*$ is accurately predicted by the complex quasimomentum $\tilde{\beta}$. This means that there exists a gap function that satisfies the condition $\omega^* = \omega(\alpha, \tilde{\beta})$. Using the approach presented in Section \ref{sec: band gap resonant frequency}, we compute a defect eigenmode and measure its decay length. We compare this with the analytic expression for the complex band structure presented in Theorem \ref{thm: subwavelength resonant frequencies}. This method enables the calculation of the decay length for each gap frequency corresponding to the associated mode.

\begin{figure}[htb]
    \centering
    \subfloat[][Band functions (black) gap functions (red) and thecay length of the defect mode (blue cross).]{{\includegraphics[width=0.7\linewidth]{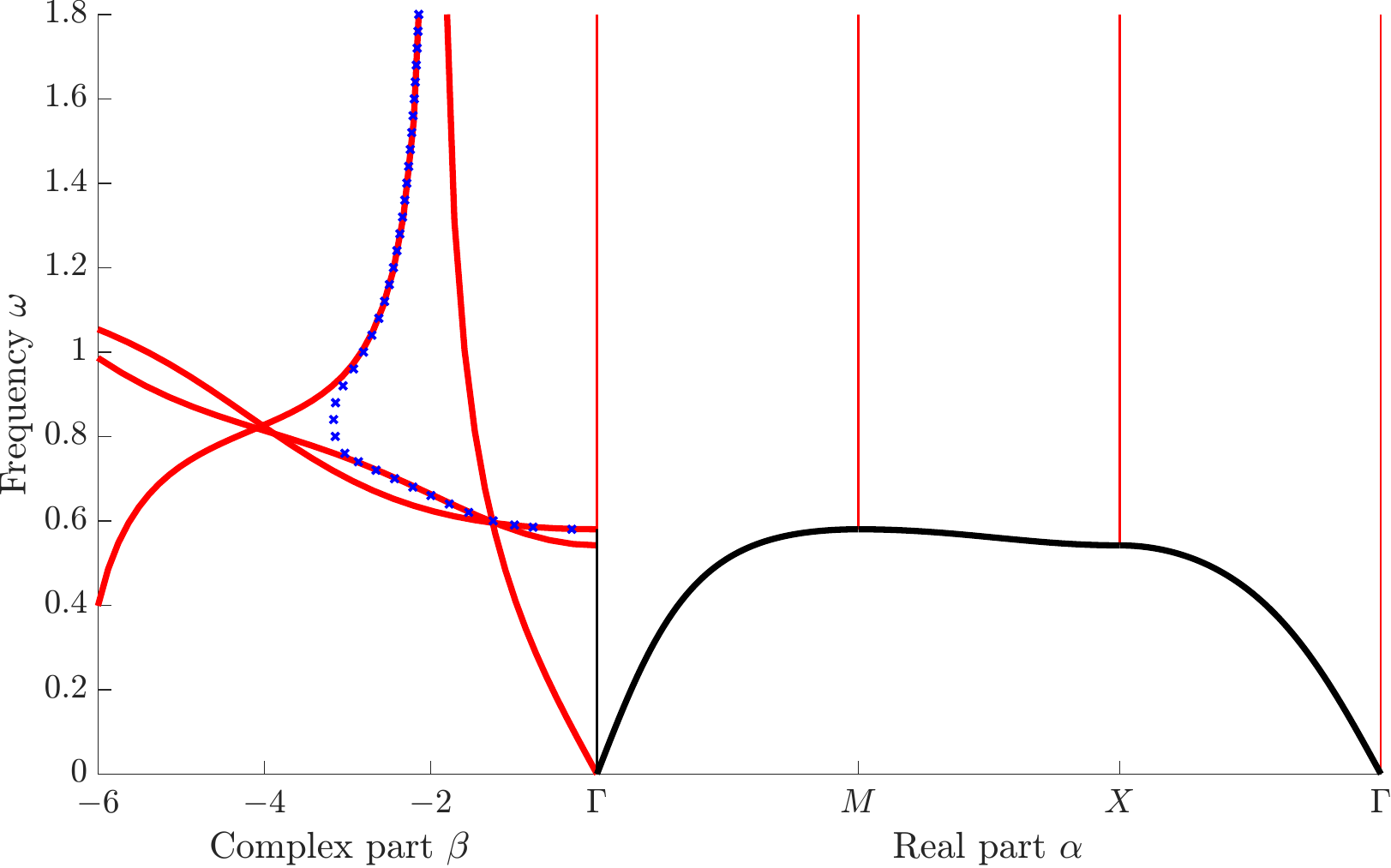} }}\hfill
    \subfloat[][Path for $\alpha$ and $\beta$.]{{\includegraphics[width=0.23\linewidth]{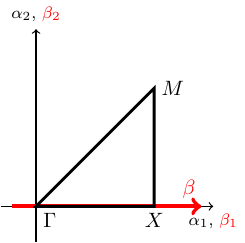} \vspace{11pt} }}
    \caption{(A) The complex band structure (red) accurately predicts the decay length of the defect eigenmode (blue cross). The decay length of the defect mode is transitioning between gap bands. (B) The decay direction is a horizontal path over the resonator lattice.}
   \label{fig: 2DHorizontalDecay and defect}
\end{figure}

A notable observation in Figure \ref{fig: 2DHorizontalDecay and defect} is that the decay rate transitions between the gap bands, suggesting that the defect mode is a combination of several quasiperiodic modes. In other words, the real quasimomentum $\alpha$ varies over the band gap. This is one of the fundamental differences between one- and two-dimensional systems: in two-dimensional crystals, there is no analogue of \cite[Theorem 2.8.]{debruijn2024complexbandstructuresubwavelength} that states that $\alpha$ stays fixed inside the band gap. This will be investigated in more detail in Section \ref{sec: displacement of alpha}.

In the spectral plot shown in Figure \ref{fig: 2DHorizontalDecay and defect}, the gap band corresponding to $\alpha = [0,0]$ exhibits a vertical asymptote. This asymptote coincides with the point where the single layer potential becomes rank-deficient (see Sections \ref{sec: Green Function Singular} and \ref{sec: SLP kenerel numerics}) resulting in a cap of the decay length of corresponding modes. In other words, when the distance between $\omega$ and the band edge increases, the decay rate stagnates and does not grow beyond a certain point.

The group velocity, which is the speed at which the wave envelope (after factoring out the overall decay) propagates through space, is defined as
\begin{equation}
    \mathtt{v}_g = \frac{\partial \omega(k)}{\partial k}.
\end{equation}
When approaching the singularity in the band function, the group velocity $\mathtt{v}_g$ of the eigenmode $v$ tends towards infinity, and the solution $v$ tends to a standing wave as illustrated in Figure \ref{Fig: Field solution 0 multipole dilute}.

\begin{remark}
    The method used for the computation of the gap bands is independent of the one used to determine the decay rate of the defect eigenmode. The fact that we get very close agreement in both methods illustrates the correctness of the numerical implementation for the gap bands.
\end{remark}

\subsection{Complex Floquet transform.} In this section, we introduce an analogous concept to the real Floquet transform applicable to band gap modes. This is achieved by generalising the Floquet transform to complex quasimomenta.
The primary result of this section demonstrates that the defect mode $u(x)$ is composed of several quasiperiodic modes. We numerically show the shift in the real quasimomentum for frequencies located within the band gap.

\subsubsection{Complex Floquet transform.} The Floquet transform is a well-known tool for studying periodic systems. An in-depth treatment of this subject can be found in \cite[Chapter 4]{kuchment2016overview}.
For the reader's convenience, we revisit a few fundamental definitions.

\begin{definition}\label{def: real floquet transform}
    The \emph{Floquet transform} $\mathcal{U}_\Lambda$ acts on functions $u \in L^2_\text{loc}(\R^n)$ as follows,
    \begin{align} \label{eq: real floquet transform}
        u(x) \mapsto \mathcal{U}_\Lambda\bigl[u(x, \alpha )\bigr]  &:= \sum_{\ell \in \Lambda} u(x - \ell)e^{\i\alpha  \cdot \ell}.
    \end{align}
\end{definition}
The Floquet transform serves as a counterpart to the Fourier transform for quasiperiodic functions. We will now provide a corresponding definition of the complex Floquet transform applicable to complex quasimomenta. We would like to emphasise that the series that defines the complex Floquet transform \eqref{eq: real floquet transform} is exponentially diverging when simply introducing a complex quasimomentum. The key observation is that at gap frequencies the underlying eigenmode is exponentially decaying which is balancing out the exponential growth of the Floquet transform. The next result formalises this idea.

\begin{theorem}\label{Thm: existence complex Floquet transform}
        Let $u$ be a complex Bloch mode for some $\beta \in \R^n \setminus \{0\}$, i.e., 
        \begin{equation}\label{eq: exponential factor}
            \bigl|u(x + \ell) \bigr| = e^{\beta \cdot \ell} \bigl|u(x)\bigr|, \hspace{2mm} \forall \ell \in \Lambda.
        \end{equation}
        Then the Floquet transform associated to a complex quasimomentum is well-defined and given by
      \begin{equation}
            u(x) \mapsto \mathcal{U}_\Lambda\bigl[u(x, \alpha + \i \beta)\bigr]  := \sum_{\ell \in \Lambda} u(x - \ell)e^{\i(\alpha + \i \beta) \cdot \ell}.
        \end{equation}
\end{theorem}

\begin{proof}
    We may define
    \begin{equation}
        v(x) = e^{-\beta \cdot x}u(x)
    \end{equation}
    then by \eqref{eq: exponential factor} it holds $|v(x + \ell)| = |v(x)|$, and it is sufficient to show that
    \begin{equation}
         \mathcal{U}_\Lambda\bigl[u(x, \alpha + \i \beta)\bigr] =  \mathcal{U}_\Lambda\bigl[v(x, \alpha )\bigr], 
    \end{equation}
    because the Floquet transform for a real quasimomentum of $v(x)$ is well-defined by \eqref{eq: real floquet transform}. It holds,
      \begin{align}
            \mathcal{U}_\Lambda\bigl[u(x, \alpha + \i \beta)\bigr] &= \sum_{\ell \in \Lambda} u(x - \ell)e^{\i(\alpha + \i \beta) \cdot \ell}\\
            &= \sum_{\ell \in \Lambda} e^{\beta\cdot  \ell} v(x - \ell) e^{\i \alpha\cdot  \ell}e^{-\beta\cdot  \ell}\\
            &= \sum_{\ell \in \Lambda} v(x- \ell) e^{\i\alpha \cdot \ell}\\
            &= \mathcal{U}_\Lambda \bigl[v(x, \alpha) \bigr].
        \end{align}
    The statement follows.
\end{proof}

For ease of notation, we define $u^{\alpha, \beta} :=\mathcal{U}_\Lambda\bigl[u(x, \alpha + \i \beta)\bigr] $. From the definition of the complex Floquet transform, we obtain the following properties.
\begin{lemma}
    The complex Floquet transform is $\Lambda$-automorphic with respect to $x$, i.e.
    \begin{equation}
        \mathcal{U}_\Lambda[u(x + m, \alpha + \i \beta)] = e^{\i (\alpha + \i \beta) \cdot m}\mathcal{U}_\Lambda[u(x, \alpha + \i \beta)], \quad m \in \Lambda.
    \end{equation}
\end{lemma}
\begin{proof}
Applying the definition of the complex Floquet transform, one verifies that,
    \begin{align}
        \mathcal{U}_\Lambda[u(x + m, \alpha + \i \beta)] &= \sum_{\ell \in \Lambda} u(x-\ell + m) e^{\i(\alpha + \i \beta)\cdot  \ell}\\
        &= \sum_{\ell \in \Lambda} u(x-\ell) e^{\i (\alpha + \i \beta)\cdot m} e^{\i(\alpha + \i \beta)\cdot\ell}\\
        &= e^{\i(\alpha + \i \beta)\cdot m}\sum_{\ell \in \Lambda} u(x-\ell) e^{\i (\alpha + \i \beta)\cdot\ell}.
    \end{align}
    The statement follows.
\end{proof}
Put differently, the complex Floquet transform of $u(x)$ with respect to $\alpha$ and $\beta$ satisfies the complex Floquet conditions.

\begin{lemma}
    The complex Floquet transform is periodic for the real quasimomentum $\alpha$, i.e.
    \begin{equation}
        \mathcal{U}_\Lambda\bigl[u(x, (\alpha + \i \beta) + m)\bigr] = \mathcal{U}_\Lambda\bigl[u(x, \alpha + \i \beta)\bigr], \hspace{2mm}\forall m \in \Lambda^*.
    \end{equation}
\end{lemma}
\begin{proof}
A direct computation yields:
    \begin{align}
        \mathcal{U}_\Lambda[u(x, (\alpha + \i \beta) + m)] &= \sum_{\ell \in \Lambda} u(x-\ell) e^{\i(\alpha + \i \beta)\cdot\ell} e^{\i m \cdot \ell}\\
        &= \sum_{\ell \in \Lambda} u(x-\ell) e^{\i(\alpha + \i \beta)\cdot\ell},
    \end{align}
    where we used $e^{\i m \cdot \ell} = 1$, which readily follows from the fact that $m \cdot \ell = 2 \pi \Z^n$.
\end{proof}
In other words, the complex Floquet transform is periodic with respect to the real part of the quasimomentum, therefore $\alpha$ can be seen as an element on the torus $\mathbb{T}^n$. In contrast to the real quasimomentum, the imaginary part is an element of $\R^n$.

The connection between the real space mode and the quasiperiodic modes is established through the inverse complex Floquet transform.
\begin{definition}
    The \emph{inverse complex Floquet transform} of $u^{\alpha, \beta}$ is,
    \begin{equation}\label{eq: def inverse floquet transform}
    u(x) = \frac{1}{2\pi}\int_{Y^*} e^{-\i(\alpha +\i \beta)\cdot \ell}u^{\alpha, \beta}(x) \d \alpha, \quad x \in Y- \ell,\hspace{2mm} \forall \ell \in \Lambda,
\end{equation}
where $\beta$ is fixed to the exponential decay rate of the eigenmode.
\end{definition}

\subsubsection{Numerical computations.} \label{sec: displacement of alpha}
This section aims to numerically demonstrate the displacement of the real component of the quasimomentum for frequencies inside the band gap. By introducing a defect into the resonator lattice, the system supports an eigenfrequency within the band gap (see Section \ref{sec: band gap resonant frequency}). The quasiperiodicity of such a defect mode can be determined as follows. The real component of the quasimomentum can be found via the truncated (complex) Floquet transform as described in \cite[Section 4]{ammari.davies.ea2023Spectrala}, which is defined as,
\begin{equation}\label{eq: truncated Floquet transform}
    (\hat{u})_{\alpha, \beta} := \sum_{m \in \Lambda_t} \mathbf{u}_m e^{\i( \alpha +\i \beta) \cdot m}.
\end{equation}
The vector $(\mathbf{u}_m)_{m \in \Lambda_t}$ is of length $N$. $\Lambda_t$ denotes the truncated infinite resonator lattice. The complex part of the quasimomentum $\beta$ is fixed to the decay length of the eigenmode $(\mathbf{u}_m)_{m \in \Lambda_t}$ and can be computed via the same method introduced in Section \ref{sec: band gap resonant frequency}. We denote this frequency-specific decay length by $\tilde{\beta}$.  We define
\begin{equation}\label{eq: density amplitude}
    \hat{U}(\alpha, \tilde{\beta}) := \lVert (\hat{u})_{\alpha, \tilde{\beta}} \rVert_2,
\end{equation}
where $\lVert \cdot \rVert_2$ denotes the $L^2$ norm. Note that the expression $\hat{U}(\alpha, \tilde{\beta})$ is dependent solely on the variable $\alpha$. The function $ \hat{U}(\alpha, \tilde{\beta})$ exhibits distinct peaks that we can assign as the real-valued quasimomenta $\alpha$ of the eigenmodes.

In Figure \ref{Fig: 2D_alpha_path_1_shift} we show a surface plot of $ \hat{U}(\alpha, \tilde{\beta})$ for the discrete Green's function shown in  Figure \ref{fig: 2DHorizontalDecay and defect}. After this consideration, the transitions between the gap bands can be understood as follows: the decay rate of the defect mode $u$ is predicted by the complex band associated with the peak $\alpha $ of \ref{Fig: 2D_alpha_path_1_shift}. Note that increasing the truncation length of the lattice produces sharper peaks of $\alpha$. The value of $  \alpha $ is selected to maximise the amplitude density $\hat{U}(\alpha, \beta)$  as it will be the largest contribution to the real space defect mode $u$. Moreover, the decay length is predicted by the $\alpha$-quasiperiodic gap band, where $\alpha = \max_{\alpha \in Y^*} \lVert (\hat{u})_{\alpha, \tilde{\beta}} \rVert_2.$

\begin{figure}[tbh]
    \centering
    \subfloat[][$\omega = 0.60$.]{{\includegraphics[width=0.33\linewidth]{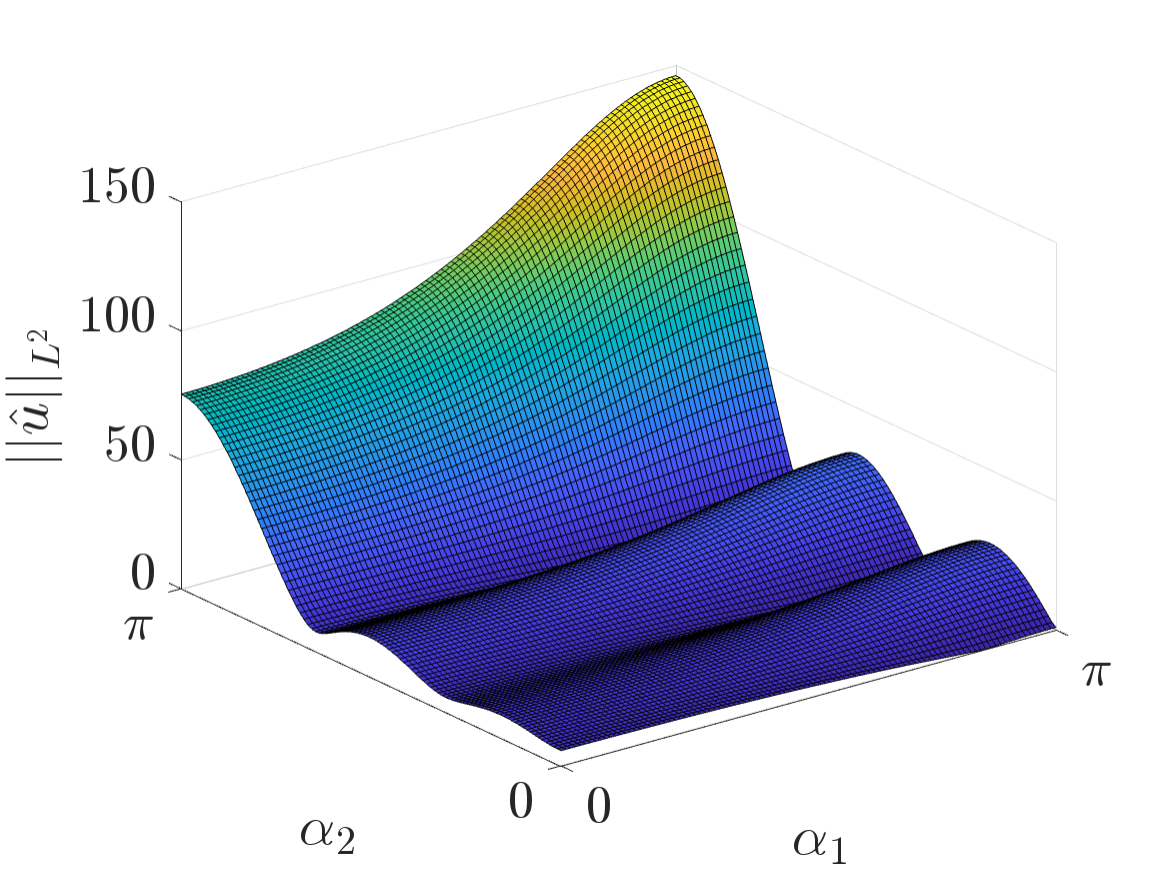}}}
    \subfloat[][$\omega = 0.70$.]{{\includegraphics[width=0.33\linewidth]{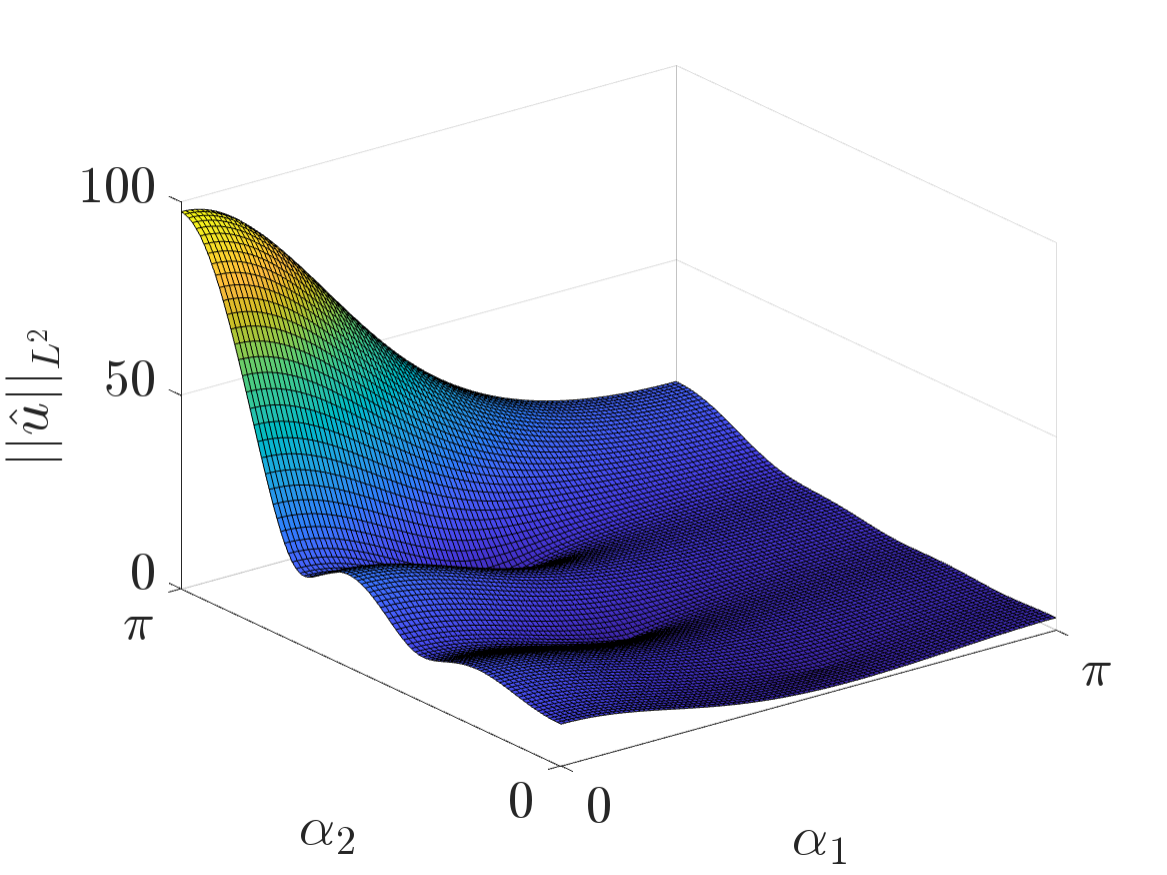}}}
    \subfloat[][$\omega = 0.90$.]{{\includegraphics[width=0.33\linewidth]{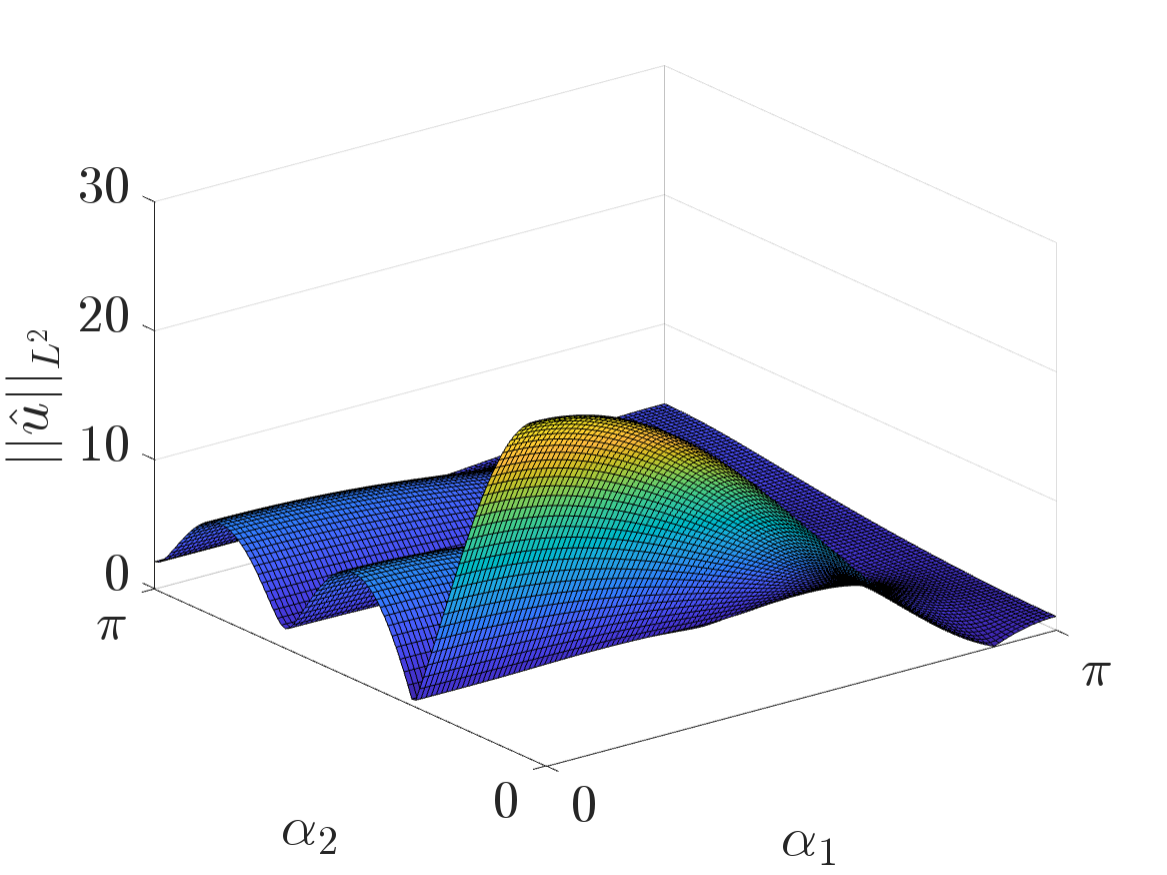}}}
    
    \caption{$\hat{U}(\alpha, \tilde{\beta})$ for given gap frequencies $\omega$ and corresponding decay rate $\tilde{\beta}$. The shift in peaks illustrates a phase transition of the defect mode in Figure \ref{fig: 2DHorizontalDecay and defect}, where the real part of the quasimomentum shifts depending on the frequency $\omega$. (A) to (B): The transition between bands $\alpha = [\pi, \pi]$ to $\alpha = [\pi,0]$. (B) to (C): The transition between bands $\alpha = [\pi, 0]$ to $\alpha = [0,0]$. An animation can be found in the \href{https://github.com/yannick2305/PhotonicBandGaps}{GitHub} Section II.4 (c.f. Section \ref{Sec: Data availability}).
 }
   \label{Fig: 2D_alpha_path_1_shift}
\end{figure} 
\begin{figure}[tbh]
    \centering
    \includegraphics[width=0.85\linewidth]{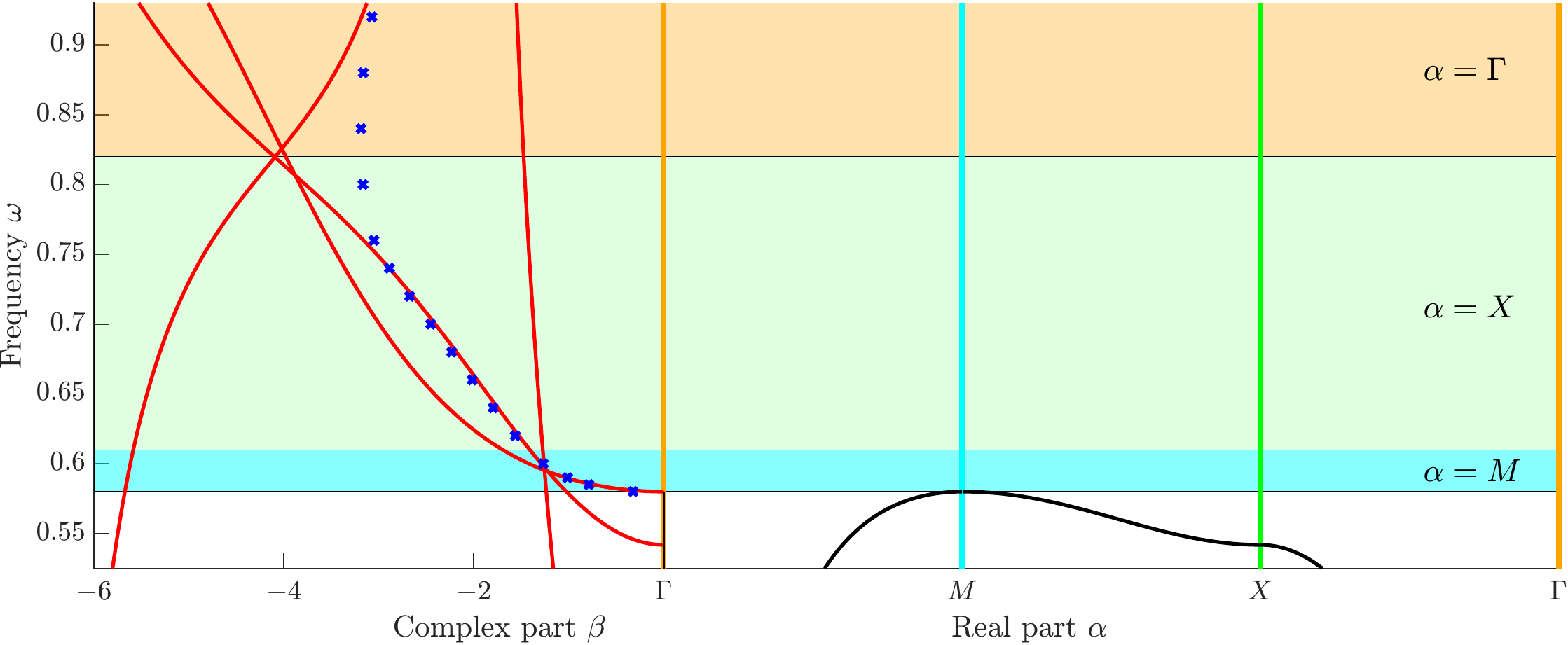}
    \caption{Phase diagram of the local oscilations of  the discrete Green's function as the frequency varies within the band gap. Following Figure \ref{Fig: 2D_alpha_path_1_shift}, within the band gap, not only the  decay rate, but also the real quasimomentum, varies. This plot illustrates that the phase $\alpha$ transitions depending on the frequency, allowing us to partition the band gap into three regions associated to $\alpha = M$, $\alpha =X$, and $\alpha =\Gamma$, respectively.}
    \label{fig: Band Transition}
\end{figure}
The results of Figure \ref{Fig: 2D_alpha_path_1_shift} can alternatively be summarised as in Figure \ref{fig: Band Transition}. This shows how the band gap of Figure \ref{fig: 2DHorizontalDecay and defect} is separated in three domains, each corresponding to different values of the real part of the quasimomentum associated to the discrete Green's function.

\subsubsection{Higher decaying Bloch modes}
By exploiting the phase transition found in the previous section, it is possible to exhibit defect modes with higher decay rate than what is given by the Green's function studied in Section \ref{sec: band gap resonant frequency}. Namely, by employing an $\alpha$-quasiperiodic source placed in the lower resonator row, we are able to excite a specific branch of the complex band structure.
In the same fashion as \eqref{eq: Green resolvent kernel} we define the $\alpha$-quasiperiodic Green's function as follows,
\begin{equation}
    (\Cf - \omega^2I) g^\alpha = \sum_{m=(i,j)}e^{\i \alpha \cdot m} \delta(j).
\end{equation}
In Figure \ref{fig: 2DQuasiperiodicDecay} we have plotted, similarly as in Section \ref{sec: band gap resonant frequency}, the decay length of the $\alpha$-quasiperiodic Green's function. By exclusively exciting a $[\pi, \pi]$-quasiperiodic resonance, the decay rate is no longer capped. The measured decay lengths follow the gap function associated to
$\alpha = [\pi, \pi]$, and increase as the distance from $\omega$ to the edge of the band increases.

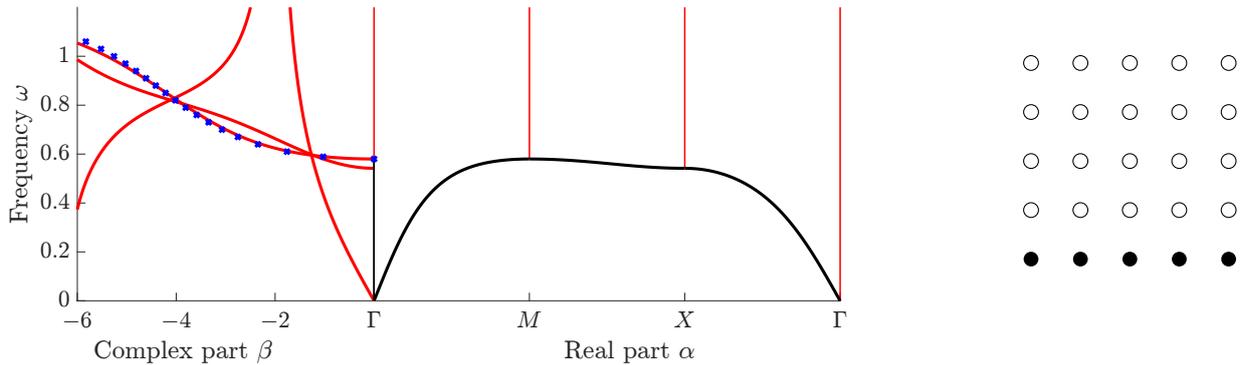
\begin{figure}[tbh]
    \centering
    \subfloat[][Complex band structure: band functions (black) gap functions (red) and the simulated decay (blue).]{{\includegraphics[width=0.68\linewidth]{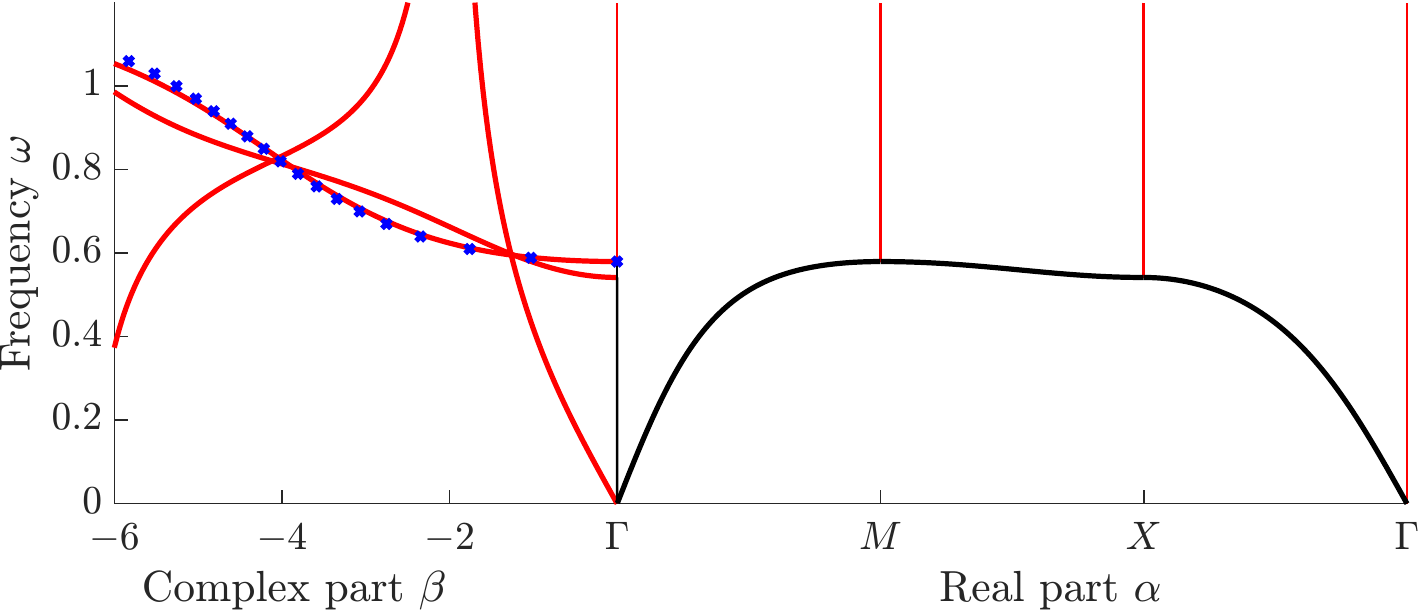} }}\hfill
    \subfloat[][$\alpha$ -quasiperiodic sources.]{\begin{tikzpicture}[scale=0.65]
        \def\discradius{0.15} % Radius of the disks
        \def\gridsize{5}      % Size of the grid (5x5)

        % Draw lattice
        \foreach \x in {1,...,\gridsize} {
            \foreach \y in {1,...,\gridsize} {
                % Fill all disks in the first row (y = 1) with black
                \ifnum\y=1
                    \fill[black] (\x,\y) circle (\discradius); % First row disks filled with black
                \else
                    \draw[black] (\x,\y) circle (\discradius); % All other disks outlined (no fill)
                \fi
            }
        }
    \end{tikzpicture}\vspace{13mm}}
   \caption{(B) shows an $\alpha$-quasiperiodic source which excites the system from below. In (A), we demonstrate that  this excites the complex band $\alpha = [\pi, \pi]$ well beyond the decay limit seen in Figure \ref{fig: 2DHorizontalDecay and defect}. }
   \label{fig: 2DQuasiperiodicDecay}
\end{figure}

\section{Dimensionally deficit lattices}\label{Three dimensional resonators}
Next, we consider a case where the dimension of the lattice is strictly smaller than the spatial dimension. Such cases have been considered in \cite{ammari.davies.ea2020Topologically,ammari2023convergence}, where algebraically (rather than exponentially) localised modes have been observed. As we shall see, this algebraic decay obstructs the formation of a complex band structure.

\subsection{One-dimensional resonator chain.}\label{sec: Existence of band gaps}
We consider a one-dimensional chain of spherical resonators in three spatial dimensions. The Helmholtz transmission problem \eqref{eq:Helmholtz2D} must now be supplemented by an outgoing radiation condition in the directions perpendicular to the lattice $\Lambda$. In this case, a capacitance matrix formulation analogous to \ref{thm:res_periodic} has been developed, where $\widehat{\mathcal{C}}^\alpha_{ij}$ now refers to the capacitance coefficients of one-dimensional lattice. We omit the details, and refer the reader to \cite{ammari.davies.ea2020Topologically,ammari2023convergence} for a complete analysis.

\begin{figure}[tbh]
    \centering
    \resizebox{0.7\textwidth}{!}{
    \begin{circuitikz}
    \tikzstyle{every node}=[font=\LARGE]
    
    % Define spheres with shading
    \shade[ball color=gray] (11.5,23.25) circle (0.7cm);
    \shade[ball color=gray] (14.5,23.25) circle (0.7cm);
    \shade[ball color=gray] (17.5,23.25) circle (0.7cm);
    \shade[ball color=gray] (20.5,23.25) circle (0.7cm); 
    \shade[ball color=gray] (23.5,23.25) circle (0.7cm);
    \shade[ball color=gray] (26.5,23.25) circle (0.7cm);
    \shade[ball color=gray] (29.5,23.25) circle (0.7cm);
    \shade[ball color=gray] (32.5,23.25) circle (0.7cm); 

    % Add dots to the left and right of the last sphere
    \node at (11.5-1.5,23.25) {\Huge \dots}; 
    \node at (32.5+1.6,23.25) {\Huge \dots}; 
    
    \end{circuitikz}
    }
    \caption{$1D$ chain of $3D$ spherical resonators. The resonators are arranged in a periodic pattern solely along a single lattice direction and therefore do not form a perfect crystal.}
    \label{fig: 3D resonators in 1D chain.}
\end{figure}
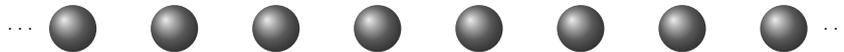  

If the unit cell is periodically replicated solely along one of the three lattice vectors defining the unit cell, as shown in Figure \ref{fig: 3D resonators in 1D chain.}, the exponential localisation of the gap mode is no longer ensured.

\subsection{Numerical results.}  We proceed analogously to Section \ref{sec: band gap resonant frequency}, and define the discrete Green's function by \eqref{eq: Green resolvent kernel}. We then determine the decay length of the gap modes by exciting the resonator chain at a band gap resonant frequency and measuring the decay length of the discrete Green's function.

\begin{figure}[tbh]
    \centering
    \includegraphics[width=0.7\linewidth]{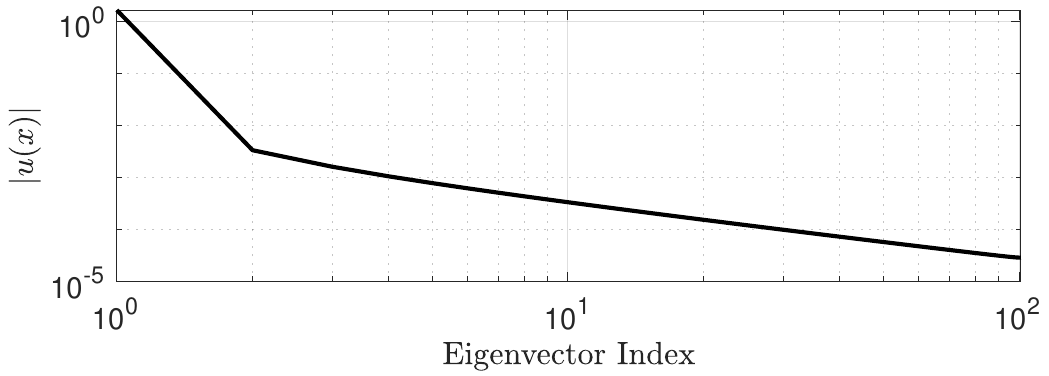}
    \caption{When excited at a band gap frequency the eigenmode decays algebraically. This is because for systems where the resonator dimension is smaller that the lattice dimension, there is no complex band structure, which guarantees the exponential decay of a gap mode.}
    \label{fig: Decay 3D}
\end{figure}

In Figure \ref{fig: Decay 3D}, we show the decay of the discrete Green's function in a loglog-scale. Analogously to Section \ref{sec: Non-periodic gauge potential.}, where the unit cell was not duplicated periodically across all lattice dimensions, the decay rate of a defect mode in a resonator chain is affected when the resonator's dimension is misaligned with the spatial dimension, which leads to the absence of a complex band structure. 

\section{Concluding remarks.} \label{Sec: concluding remarks}
We have identified a groundbreaking connection between the complex band structure and various non-Hermitian problems. We presented these findings on defect modes and the non-Hermitian skin effect, applicable to both one-dimensional and two-dimensional resonator systems. Furthermore, we addressed the numerical bottleneck mentioned in our previous study \cite[Section 4]{debruijn2024complexbandstructuresubwavelength} and introduced a numerically stable and efficient method to calculate the complex band structure of two-dimensional crystals.
Additionally, the constraints of the complex band structure were highlighted, as its presence is assured solely in systems that demonstrate complete periodicity.

In future research, we plan to expand our scope to encompass disordered systems in both one-dimensional and two-dimensional contexts. This will enable the exploration of a broader range of periodic structures, potentially incorporating correlated random variables. 
For two-dimensional systems, we plan to study different resonator arrangements, such as, for instance, honeycomb crystals. 
Additionally, it is important to expand the complex band structure from the context of the Helmholtz scattering problem to encompass a broader range of elliptic and periodic partial differential equations.

\section{Data availability} \label{Sec: Data availability}
The software developed in this work are openly available at the following repository:
\url{https://github.com/yannick2305/PhotonicBandGaps}.

%\section{Conflict of interest}
%The authors have no conflict of interest to declare.

\section{Acknowledgement}
The authors thank Habib Ammari for insightful discussions and Hanwen Zhang for providing an adaptive numerical Gaussian integration tool. Y.D.B. thanks Thea Kosche for the numerous discussions.

\appendix

\section{Numerical methods} 

\subsection{Kummer's transformation.}\label{appendix: SLP Kummer's method}
This section shows how to derive a matrix representation of the single layer potential, utilising an accelerated expression of the Green's function via Kummer's method. Following \eqref{def: single layer potential} the single layer potential is defined as,
\begin{align}
    (\tilde{\S}^{0, \beta, k}_D)_{m,n} &= \frac{1}{2\pi}\bigl\langle e^{\i m \theta}, \tilde{\S}_D^{0, \beta, k}[e^{\i n \varphi}]\bigr\rangle \\
    &=  \frac{r}{2\pi}\int_0^{2\pi} e^{-\i m\theta} \int_0^{2\pi} \Tilde{G}^{0, \beta, k}_{K}\bigl(r(e^{\i  \theta}- e^{\i  \varphi})\bigr) e^{\i n\varphi}  \d\varphi \d\theta,
\end{align}
where the Green's function is given by \eqref{eq: def Green Kummer}. The path along the perimeter of the circular resonator  is given by,
\begin{equation}
    x = \begin{pmatrix}
        x_1\\
        x_2
    \end{pmatrix} = \begin{pmatrix}
        r\cos(\theta) - r\cos(\varphi)\\
        r\sin(\theta) - r\sin(\varphi)
    \end{pmatrix}.
\end{equation}
We examine each term of the Green's function individually.
\begin{equation}
    \frac{r}{2\pi} \int_0^{2\pi} \int_0^{2\pi} \frac{e^{-\i m \theta} e^{\i n \varphi}}{k^2 + \lvert \beta \rvert^2}\d \theta \d \varphi = \frac{r}{2\pi}\frac{(2\pi)^2}{k^2 + \lvert \beta \rvert^2}\delta_{m,0}\delta_{n,0}.
\end{equation}
The first lattice sum can be treated as follows:

\begin{align}
    &\frac{r}{2\pi} \int_0^{2\pi} e^{-\i m \theta}\int_0^{2\pi} \left( \frac{1}{12} -\frac{\ln(2)}{2\pi} + \frac{1}{4}\bigg(\bigl(r\cos(\theta) - r\cos(\varphi)\bigr)^2 + \bigl(r\sin(\theta) - r\sin(\varphi)\bigr)^2\bigg) \right)e^{\i n \varphi}\d\theta \d \varphi\\
    &= \frac{r}{2\pi} \int_0^{2\pi} e^{-\i m \theta}\int_0^{2\pi} \left( \frac{1}{12} -\frac{\ln(2)}{2\pi} + \frac{r^2}{2}\right)e^{\i n \varphi}\d\theta \d \varphi -\frac{r^3}{4\pi} \int_0^{2\pi} \int_0^{2\pi} e^{-\i m \theta}\cos(\theta-\varphi)e^{\i n \varphi} \d\theta \d \varphi\\
    &= \frac{r}{2\pi}(2 \pi)^2\left( \frac{1}{12} -\frac{\ln(2)}{2\pi} + \frac{r^2}{2} \right)\delta_{n,0}\delta_{m,0} - \frac{r^3}{4\pi} \left( 2\pi^2 \left( \delta_{m,1} \delta_{n,1} + \delta_{m,-1} \delta_{n,-1} \right) \right)\\
    &= 2\pi r\left( \frac{1}{12} -\frac{\ln(2)}{2\pi} + \frac{r^2}{2} \right)\delta_{m,0}\delta_{n,0} - \frac{r^3 \pi}{2}(\delta_{m,1} \delta_{n,1} + \delta_{m,-1} \delta_{n,-1}),
\end{align}
where we employed the identity,
\begin{equation}
    \frac{1}{4}\bigl(r\cos(\theta) - r\cos(\varphi)\bigr)^2 + \bigl(r\sin(\theta) - r\sin(\varphi)\bigr)^2 = \frac{r^2}{2}\bigl(1-\cos(\theta-\varphi)\bigr).
\end{equation}
The logarithmic components can be addressed in the following way; the integrand may be simplified by taking into account that
\begin{align}
    \cos(\theta)- \cos(\varphi) &= \sin(\theta + \frac{\pi}{2}) - \sin(\varphi + \frac{\pi}{2})\\
    \sin(\theta) -\sin(\varphi) &= \cos(\theta + \frac{\pi}{2}) - \cos(\varphi + \frac{\pi}{2}).
\end{align}
Consequently, it is observed that
\begin{align}
    &\ln\bigg[\sinh^2\bigg(\pi r \bigl(\sin(\theta)-\sin(\varphi)\bigr)\bigg) + \sin^2\bigg(\pi r \bigl(\cos(\theta)-\cos(\varphi)\bigl)\bigg)\bigg]\\
    &= \ln\bigg[\sinh^2\bigg(\pi r\bigl(\cos(\theta + \frac{\pi}{2})-\cos(\varphi + \frac{\pi}{2}) \bigr)\bigg) + \sin^2\bigg(\pi r \bigl(\sin(\theta + \frac{\pi}{2})-\sin(\varphi + \frac{\pi}{2})\bigr)\bigg)\bigg].
\end{align}
Given that the function possesses a $2\pi$ periodicity, the integrands are identical, thus it suffices to compute

\begin{equation}\label{eq: integral over log singularity}
    \frac{r}{\pi} \int_0^{2\pi} e^{\i n\varphi} \int_0^{2\pi} e^{-\i  m \theta} \ln\bigg[\sinh^2\bigg(\pi r\bigl(\sin(\theta)-\sin(\varphi)\bigr)\bigg) + \sin^2\bigg(\pi r \bigl(\cos(\theta)-\cos(\varphi)\bigr)\bigg)\bigg] \d \theta \d\varphi.
\end{equation}
Deriving an explicit solution for this integral seems beyond reach. Additionally, evaluating \eqref{eq: integral over log singularity} is numerically difficult, given that the singularity in the logarithm disrupts the convergence of discretised versions of the integral. Instead, we use an adapted Gaussian integration method developed by Prof. Vladimir Rokhlin and Hanwen Zhang. The external integrand, being a periodic and smooth function, allows us to implement the trapezoidal rule in a exponentially convergent manner.

For the last term in the lattice sum, we truncate the exponentially convergent series to just two terms. The resulting function is smooth and periodic, so we implement the trapezoidal rule, which is exponentially convergent in the number of discretisation points.

The second lattice sum can be evaluated similarly as in \cite[Section A.3.]{debruijn2024complexbandstructuresubwavelength}, whereby the integrals are represented using Bessel functions,
\begin{align}
    &\frac{r}{2\pi}\int_0^{2\pi} \int_{0}^{2\pi} e^{-\i m \theta}\left( \sum_{q \in \Lambda^*\setminus\{0\}} e^{\i q \cdot x}\left( \frac{1}{k^2 + \lvert \beta \rvert^2 - 2\i \beta \cdot q- \lvert q \rvert^2} + \frac{1}{\lvert q \rvert^2}\right)\right) e^{\i n \varphi} \d \theta \d \varphi\\
    &= \frac{2\pi r}{\lvert Y \rvert} \sum_{q \in \Lambda^*\setminus\{0\}} \frac{(k^2 + \lvert \beta \rvert^2 -2\i\beta\cdot q)e^{\i  \psi(n-m)}\i^{m+n} (-1)^n\mathbf{J}_{m}\bigl(r\lvert q \rvert\bigr)\mathbf{J}_{n}\bigl(r\lvert q \rvert\bigr)}{\bigl(k^2 + \lvert \beta \rvert^2 - 2\i \beta \cdot q- \lvert q \rvert^2\bigr)\lvert q \rvert^2}.
\end{align}
In this manner, we may evaluate the single layer potential in a multipole basis.

\subsection{Convergence rates and runtime.}\label{sec: Convergence rates and runtime} We have explored three different methods for calculating the single layer potential, each offering unique advantages. Since all three computations involve truncating a lattice sum, we aim to illustrate convergence and runtime with respect to the truncation size.

\begin{figure}[tbh]
    \centering
\includegraphics[width=0.75\linewidth]{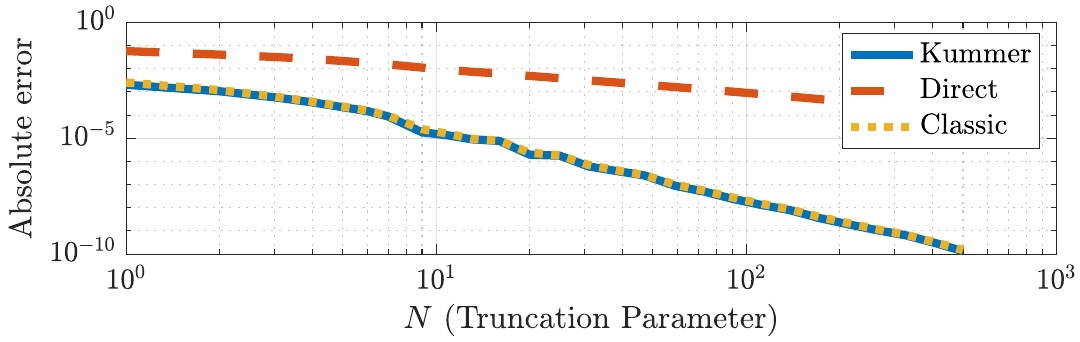}
    \caption{Convergence of the single layer potential w.r.t. the truncation $N$. Computation performed with a multipole of $2$ and against a lattice of size $[-N, N]\times[-N, N]$ with $N = 5000$. }
    \label{Fig: Error estimate single layer potential}
\end{figure}

From Figure \ref{Fig: Error estimate single layer potential}, one deduces that the convergence rate of the three methods is algebraic. However, transformed lattice sum computations converge significantly faster than the direct lattice sum: $\mathcal{O}(N^{-1})$ in direct computations versus $\mathcal{O}(N^{-3})$ in transformed methods. Therefore, as stated in Remark \ref{rem: error goal}, to attain an acceptable error using the direct method requires a lattice size of $N = 1000$, whereas the transformed approaches only need a lattice size of $N = 10$. We emphasise that both Kummer's method and the classical method exhibit the same convergence rate. However, with Kummer's method, we can compute the gap function even for $\alpha \approx 0$.
By leveraging \texttt{MATLAB}'s built-in parallel computing capabilities and rewriting the code in a vectorised form, we significantly improve runtime performance.

\begin{figure}[tbh]
    \centering
    \includegraphics[width=0.75\linewidth]{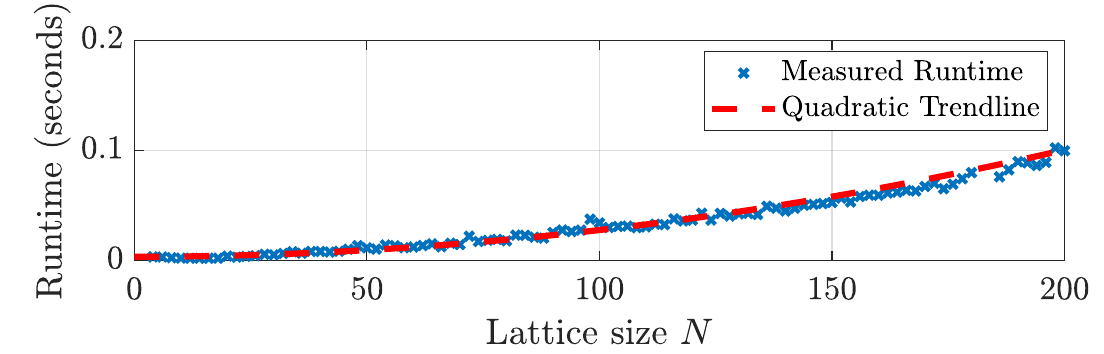}
    \caption{Runtime to generate the single layer potential for a lattice sum over a truncated lattice of size $[-N, N]\times[-N, N]$ and using a multipole of order $3$.}
    \label{fig: Runtime SLP}
\end{figure}

The fitted quadratic trend line $y = aN^2 + c $ for the runtime in Figure \ref{fig: Runtime SLP}, with $ a = 2.3418 \times 10^{-6} $ and $ c = 0.001 $, suggests a quadratic dependence of the runtime on the lattice size $ N $. However, since the coefficient $ a $ is relatively small, the dominant term in the equation is linear, indicating that within the observed range of $ N $, the runtime behaves approximately linear. 
\begin{remark}
    Our analysis intentionally focuses on the execution time required for generating the single layer potential, given that the calculation of the capacitance matrix from the single layer potential is unaffected by variations in lattice size and can thus be regarded as a constant factor in the runtime analysis.
\end{remark}

\printbibliography

\end{document}